\documentclass[final, 11pt]{article}
\usepackage[ruled, linesnumbered]{algorithm2e}
\usepackage{mathtools}
\usepackage{mdwlist}
\usepackage{graphicx}
\usepackage{amsthm}
\usepackage{amsfonts}
\usepackage{amsmath}
\usepackage{amssymb}
\usepackage{fancyhdr}
\usepackage{titlesec}
\usepackage{indentfirst}
\usepackage{booktabs}
\usepackage{verbatim}
\usepackage{color}
\usepackage{latexsym}
\usepackage{amscd}
\usepackage{esvect}
\usepackage{graphicx}
\usepackage{upgreek}
\usepackage{subcaption}   
\usepackage{caption}
\usepackage[page,header]{appendix}
\usepackage{titletoc}
\usepackage{mathrsfs}
\usepackage[numbers]{natbib}
\usepackage{float}
\usepackage{lipsum}
\usepackage[export]{adjustbox} 
\numberwithin{equation}{section}
\newtheorem{theorem}{Theorem}[section]
\newtheorem{lemma}[theorem]{Lemma}

\newtheorem{remark}[theorem]{Remark}

\allowdisplaybreaks



\def\e{\varepsilon}

\SetKwInOut{Offline}{Offline}
\SetKwInOut{Online}{Online}
\newtheorem{assumption}{Assumption}

\begin{document}
\title{A bi-fidelity method for the multiscale Boltzmann equation 
with random parameters 
\thanks{L. Liu was partially supported by the funding DOE--Simulation Center for Runaway Electron Avoidance and Mitigation, project No.\,DE-SC0016283.
X. Zhu was supported by the Simons Foundation (504054).}
}

\date{\today}
\author{Liu Liu\footnote{The Oden Institute for Computational Engineering and Sciences, University of Texas at Austin, Austin, TX, USA (lliu@ices.utexas.edu)}, 
Xueyu Zhu \footnote{Department of Mathematics, University of Iowa, Iowa City, IA, USA (xueyu-zhu@uiowa.edu)}
}

\maketitle

\abstract{In this paper, we study the multiscale Boltzmann equation with multi-dimensional random parameters by a bi-fidelity stochastic collocation (SC) method developed in \cite{NGX14, ZNX14, ZX17}. By choosing the compressible Euler system as the low-fidelity model, we adapt the bi-fidelity SC method to combine computational efficiency of the low-fidelity model with high accuracy of the high-fidelity (Boltzmann) model. With only a small number of  high-fidelity asymptotic-preserving solver runs for the Boltzmann equation, the bi-fidelity approximation can capture well the macroscopic quantities of the solution to the Boltzmann equation in the random space. A priori estimate on the accuracy between the high- and  bi-fidelity solutions together with a convergence  analysis is established. Finally, we present extensive numerical experiments to verify the efficiency and accuracy of our proposed method. 
}

{\bf Keywords.} Boltzmann equation, uncertainty, bi-fidelity models, multiple scales, stochastic collocation

\section{Introduction}
Kinetic equations are widely used in classical fields such as rarefied gas, plasma physics, astrophysics, 
also in emerging areas such as semiconductor device modeling, social and biological sciences. 
They model the non-equilibrium dynamics of a large number of particles from a statistical point of view \cite{Chapman}. 
The Boltzmann equation, as one of the most fundamental kinetic equations \cite{Cercignani-Book}, is an integro-differential equation describing the time evolution of probability density distribution of particles in rarefied gas. 

There have been extensive studies on the Boltzmann and related kinetic models, both in theory and numerical computation \cite{Villani, GL14}. 
One of the main computational challenges is that the kinetic problems often encounter multiple temporal and spatial scales, characterized by the Knudsen number, the dimensionless mean free path, that may vary in 
orders of magnitude in the computational domain, covering the regimes from fluid, transition to rarefied. 
Asymptotic-preserving (AP) schemes, which preserve the asymptotic transition from one scale to another at the discrete level, have been shown to be an effective computational paradigm in the past two decades \cite{Jin99, Jin12}.  
They allow efficient numerical approximations in all regimes--coarse mesh and large time steps can be used even in the fluid dynamic regime, without numerically resolving the small Knudsen number.
For the space inhomogeneous Boltzmann equation, AP schemes were first designed using BGK-operator penalty based method \cite{FJ}. 
Other approaches include the exponential integrator based methods \cite{DP11} or micro-macro decomposition \cite{Lemou-MM, GJL-MM}. Despite the successful development of AP solvers in recent years, the complexity and memory cost for computing the collision term still remains challenging \cite{PR00, FastSpectral, Gamba-Harsha, Gamba-Hu}. 

Another challenge, which has been ignored in the community until recently, is the issue of {\it uncertainties in kinetic models.}
Kinetic equations, derived from $N$-body Newton's equations via the mean-field limit \cite{BGP},
typically contain an integral operator modeling interactions between particles. 
 Calculating the collision kernel from first principles is extremely complicated and not possible for complex particle systems, thus only empirical formulas are used for general particles \cite{Cercignani}, which means that the collision kernel contains many uncertainties. 
 Other sources of uncertainties {\color{black} may also come from 
 inaccurate} measurements of the initial or boundary data, 
 forcing or source terms. 
See for example \cite{jin2015asymptotic,DPZ-Review,HJ-Review, LJ-UQ, Esther, LL-LB, LL-BP} 
for recent efforts on uncertainty quantification for kinetic equations, in particular  \cite{SHJ17,HJS-Boltz}, 
where numerical schemes for the Boltzmann equation with multi-dimensional random inputs have been studied.

{\color{black} One widely used method} in uncertainty quantification is stochastic collocation (SC) method, especially in conjunction with the gPC expansion and high-performance grids such as sparse grids. There have been many works developed, for example \cite{Babuska, Roberts, NXiu, Webster, GWZ14, Xiu07, Schwab}.  
One of the challenges central to collocation approaches is the simulation cost. 
For many complex systems, in particular, the multiscale Boltzmann equation with multi-dimensional random inputs we are studying, 
an accurate high-fidelity deterministic simulation can be so time-consuming and memory demanding that only a few high-fidelity simulations can be afforded. 
As many stochastic algorithms such as SC require repetitive implementations of the deterministic solver, the overall 
accurate stochastic simulation can be difficult and even computationally infeasible. 

Fortunately, there usually exist some approximate, less complex low-fidelity models for practical problems. 
Compared to the high-fidelity models, these low-fidelity models usually contain simplified physics and/or are simulated on a coarser physical mesh, and consequently, own a cheaper computational cost. 
Although their accuracy may not be high, the low-fidelity models are designed in such a way that they can resolve or capture certain important features of the underlying problem and produce reliable and qualitative predictions. 
Despite numerous multi-fidelity algorithms have been developed in different communities from different perspectives  \cite{giles2008multilevel,cliffe2011multilevel,YangTT18,kennedy2000predicting,tuo2014surrogate,perdikaris2017nonlinear,peherstorfer2016optimal,peherstorfer2016multifidelity,ng2012multifidelity,eldred2017multifidelity,NGX14, ZNX14, ZX17}, there have not been many attempts for kinetic equations with uncertainty in the multi-fidelity setting, except for the recent work \cite{dimarco2019multi}. In \cite{dimarco2019multi}, the authors
take the steady state or approximated time-dependent solution with the 
asymptotic behavior close to the fluid limit as control variate models to accelerate the convergence of standard Monte Carlo methods. 

{\color{black}Despite numerous multi-fidelity algorithms have been developed for different applications \cite{MF1, MF2, MF3, MF4, MF5, MF6}},
the goal of our work is to adapt the bi-fidelity method developed in \cite{NGX14, ZNX14, ZX17}
to efficiently approximate the high-fidelity solutions of the Boltzmann equation with multi-dimensional random parameters and multiple scales. 
In rarefied gas dynamics, fluid models are derived when the mean free path of a particle is very small compared to the typical macroscopic length. One can perform a Hilbert or Chapman-Enskog expansion
(\cite{Cercignani-Book}) of the solution to the Boltzmann equation in powers of the Knudsen number $\e$. At the leading order in $\e$, the distribution function approaches 
a local equilibrium -- a Maxwellian whose parameters are the fluid variables (density, mean velocity and temperature) governed by the compressible Euler equations. 
When $\e$ is small, the Euler equations provide us a good accuracy in the physical space and numerical efficiency. 
Motivated by the above observation, we take advantage of this multiscale nature of the kinetic problem and choose the low-fidelity model based on the Euler equations in this work. More specifically, we connect
an $\e$-dependent microscopic model (Boltzmann equation) and its macroscopic model (Euler equations) through the corresponding macroscopic quantities of the Boltzmann equation. Our numerical experiments demonstrate that the bi-fidelity solutions can approximate the high-fidelity solutions well at a much reduced computational cost.  

This paper is outlined as follows: In Section \ref{sec:1}, we give an introduction of the Boltzmann equation and its macroscopic equations. Section \ref{sec:2} reviews the framework of SC method with multifidelity models and discusses how it is adapted to our problem under study. Section \ref{sec:3} establishes the convergence of the bi-fidelity approximation to the high-fidelity solution under suitable assumptions. In section \ref{sec:4}, we provide extensive numerical experiments to illustrate the effectiveness and efficiency of our proposed method, 
where kinetic, fluid and mixed regimes are carefully examined. We conclude the paper in Section \ref{sec:5}. 

%------------------------------------------------------------------------------------
\section{Introduction of the Boltzmann equation}
\label{sec:1}
\subsection{The Boltzmann equation with uncertainty}
We first give an introduction to the classical (deterministic) Boltzmann equation, known as one of the most celebrated kinetic equations for rarefied gas. 
A dimensionless form reads
\begin{equation}\label{Boltz}
\partial_t f + v\cdot \nabla_x f= \frac{1}{\varepsilon}\mathcal Q(f, f), \qquad (x,v)\in\Omega\times\mathbb R^{d_v}, 
\end{equation}
where $f(t,x,v)$ is the probability density distribution function, modeling the probability of finding a particle at time $t>0$, at position {\color{black}
$x$ in a bounded domain $\Omega \subset\mathbb R^{d_x}$},
with velocity $v\in\mathbb R^{d_v}$, {\color{black}where $d_x$ and $d_v$ are the dimensions of the $x$ and $v$ variables.}
Periodic boundary condition is considered. 
The parameter $\e$ is the Knudsen number defined as the ratio of the mean free path over a typical length scale such as the size of the spatial domain. 
The collision operator $\mathcal Q$ is a quadratic integral operator modeling the binary elastic collision between particles, and is given by 
\begin{equation} \mathcal Q(f, f)(v) = \int_{\mathbb R^{d_v}}\int_{\mathbb S^{d_v -1}} 
B(|v-v_{\ast}|, \cos\theta) \left(f(v^{\prime})f(v_{\ast}^{\prime}) - f(v)f(v_{\ast})\right) d\sigma dv_{\ast}. 
\end{equation}
{\color{black}$(v, v_{\ast})$ and $(v^{\prime}, v_{\ast}^{\prime})$ are the velocity pairs before and after the collision, during which the momentum and energy are conserved; thus
$(v^{\prime}, v_{\ast}^{\prime})$ can be expressed in terms of $(v, v_{\ast})$ as follows:}
\begin{align}
\begin{cases}
&\displaystyle v^{\prime}= \frac{v+v_{\ast}}{2} + \frac{|v-v_{\ast}|}{2}\sigma, \\[4pt]
&\displaystyle v_{\ast}^{\prime} = \frac{v+v_{\ast}}{2} - \frac{|v-v_{\ast}|}{2}\sigma, \end{cases}
\end{align}
with the vector $\sigma$ the scattering direction varying on the unit sphere $\mathbb S^{d_v -1}$. 
The collision kernel $B$ is a non-negative function of the form 
$B(v, v_{\ast},\sigma) = B(|v-v_{\ast}|, \cos\theta)$, where $\cos\theta = \frac{\sigma\cdot (v-v_{\ast})}{|v-v_{\ast}|}$ is the deviation angle.
We consider the variable hard sphere (VHS) model \cite{Bird}, with 
a commonly used form for the collision kernel: 
{\color{black}\begin{equation}\label{Kernel} B(v, v_{\ast},\sigma) =  b\, |v-v_{\ast}|^{\gamma}, \qquad -d_v < \gamma \leq 1, \end{equation}
where $b$ is a positive constant, $\gamma>0$ corresponds to the hard potential, and $\gamma<0$ is the soft potential. }
Denote {\color{black}the vector}\begin{equation}
\label{mdef}
m(v) = \left(1, v, \frac{|v|^2}{2}\right)^{\top}, \end{equation}
then 
\begin{equation}\label{Q_cons}\int_{\mathbb R^{d_v}}\mathcal Q(f, f)m(v)\, dv = 0, \end{equation}
which correspond to the conservation of mass, momentum and kinetic energy of the collision operator. 

The celebrated Boltzmann's H-theorem gives the dissipation of entropy (\cite{Cercignani-Book}): 
$$ \int_{\mathbb R^d} \mathcal Q(f, f)\ln f\, dv \leq 0. $$
Furthermore, the equality holds if and only if $f$ reaches the equilibrium state 
\begin{equation}\label{Maxwell} M(v)_{\rho, u, T} = \frac{\rho}{(2\pi T)^{\frac{d_v}{2}}}\exp\left(-\frac{|v-u|^2}{2T}\right), 
\end{equation}
which is known as the Maxwellian.
Here $\rho$, $u$ and $T$ are  the density, bulk velocity and temperature, respectively: 
\begin{equation}\label{macro}\rho = \int_{\mathbb R^{d_v}} f(v)\, dv, \qquad u = \frac{1}{\rho}\int_{\mathbb R^{d_v}} f(v)v\, dv, 
\qquad T = \frac{1}{d_v \rho} \int_{\mathbb R^{d_v}} f(v)|v-u|^2\, dv. \end{equation}

There are many sources of uncertainties in the Boltzmann equation, 
such as the initial data, boundary data, and collision kernel. We introduce the Boltzmann equation with uncertainty
\begin{align}
\label{model}
\left\{
\begin{array}{l}
\displaystyle \partial_t f + v\cdot\nabla_x f  =
\frac{1}{\varepsilon}\mathcal Q_z(f, f),   \\[4pt]
\displaystyle  f(0,x,v,z)= f_{I}(x,v,z), \qquad
(x,v)\in\Omega\times\mathbb R^{d_v},  z\in I_z. 
\end{array}\right.
\end{align}
Here $z\in I_z$ is a $d$-dimensional random parameter with probability distribution $\pi(z)$ known in priori characterizing 
the uncertainty in the system. Without loss of generality, 
we only consider periodic boundary condition in space throughout this paper. 

\subsection{The macroscopic fluid equations}  
\label{fluid}
When the Knudsen number $\varepsilon>0$ becomes very small,
the macroscopic fluid dynamics describing the evolution of averaged quantities such as the density $\rho$, momentum $\rho u$ 
and temperature $T$ of the gas, namely, the compressible Euler or Navier-Stokes equations, become adequate \cite{Bardos93, BGP}. 

{\color{black}Multiplying (\ref{model}) by $m(v)$ and integrating with respect to $v$, by using the conservation property of $\mathcal Q$
given in (\ref{Q_cons}), {\color{black}one gets} a non-closed system of conservation laws
\begin{equation}\label{INS} \partial_t \begin{pmatrix} \rho \\ \rho u \\ E \end{pmatrix}
+ \nabla_x \cdot  \begin{pmatrix} \rho u \\ \rho u \otimes u + \mathbb P \\
E u + \mathbb P u + \mathbb Q  \end{pmatrix} = 0, 
\end{equation}
where $E$ is the total energy defined by
\begin{equation}\label{Energy}
E=\langle \frac{1}{2}|v|^2 f \rangle = \frac{1}{2}\rho |u|^2 + \frac{d_v}{2}\rho\, T,  \end{equation}
{\color{black}with $\langle\, \cdot\, \rangle$ denoted as a velocity average of the argument,} 
$$\langle g \rangle = \int_{\mathbb R^{d_v}} g(v)\, dv. $$
Here $\mathbb P = \langle (v-u)\otimes (v-u)f \rangle$ is the pressure tensor, 
and $\mathbb Q = \frac{1}{2}\langle (v-u) |v-u|^2 f \rangle$ is the heat flux vector. 
Note that the variables $\rho$, $u$ and $E$ in \eqref{INS} depend on the random parameter $z$. }

When $\varepsilon\to 0$, $f\to M(v)_{\rho, u, T} $. We can approximate $f$ by $M(v)_{\rho, u, T}$ and use the expression 
(\ref{Maxwell}), $\mathbb P$ and $\mathbb Q$ become 
$$ \mathbb P = p\, \mathbb I, \qquad \mathbb Q = 0, $$
where $p = \rho\, T$ is the pressure, $\mathbb I$ is the identity matrix. Then (\ref{INS}) reduces to the compressible Euler equations
of gas dynamics for a mono-atomic gas: 
\begin{equation}\label{Euler}
\partial_t \begin{pmatrix} \rho \\ \rho u \\ E \end{pmatrix} + 
\nabla_x \cdot \begin{pmatrix} \rho u \\ \rho u\otimes u + p\,\text{I} \\ (E+p) u \end{pmatrix} = 0, 
\end{equation}
which is known as a first order approximation with respect to $\varepsilon$ to the Boltzmann equation (\ref{model}). 
By the Chapman-Enskog expansion, 
the compressible Navier-Stokes equations give a second order approximation in $\e$ to the distribution function of the Boltzmann equation \cite{Cercignani-Book}. 

\begin{comment}
It is worth noting that since $\rho$, $u$, $T$ do not depend on the velocity in \eqref{INS}, solving the deterministic Euler equation numerically is much more efficient in term of memory and computation time compared to solving the deterministic Boltzmann equation (\ref{Boltz}) due to its high-dimensional nature (in physical space). 
Together with the ease of solving the Euler system and its first-order approximation to the Boltzmann equation, these serve as the motivations of why we choose the Euler equation as the low-fidelity model in our numerical tests. 
\end{comment}

%------------------------------------------------------------------------------------
\section{A stochastic collocation method with bi-fidelity models}
\label{sec:2}

In this section, we first briefly review an efficient bi-fidelity approximation to the high-fidelity solution studied in 
\cite{NGX14, ZNX14}, then we shall discuss
the motivation of our choices of the low-fidelity model in our current study. 

\subsection{A bi-fidelity Algorithm}
Assume we have access to the high-fidelity solutions $u^H(z)$ and low-fidelity solutions $u^L(z)$. Let 
$M$ be the number of affordable low-fidelity simulation runs, which can be very large. $N$ denotes the number of high-fidelity simulation runs that can be afforded and is often very small, 
i.e., $M\gg N$. 
Let $\gamma=\{z_1, \cdots, z_k\}$, $k\geq 1$ be a set of sample points in $I_z$. Denote the low-fidelity snapshot matrix 
$u^L(\gamma) = [u^L(z_1), \cdots, u^L(z_k)]$ and the corresponding low-fidelity approximation space
$$ U^L(\gamma) = \text{span}\{u^L(\gamma)\} = \text{span}\{ u^L(z_1), \cdots, u^L(z_k) \}. $$  
Similarly, the high-fidelity snapshot matrix and the correponding high-fidelity approximation as follows:
$$ u^H(\gamma) = [u^H(z_1), \cdots, u^H(z_k)], \qquad U^H(\gamma)=\text{span}\{u^H(\gamma)\}. $$

The bi-fidelity algorithm for approximating the high-fidelity solution consists of offline and online stages. In the offline stage, we employ the cheap low-fidelity model to explore the parameter space to find the most important parameter points. Within the online stage, we learn the approximation rule from the low-fidelity model for any given $z$, and apply it to construct the bi-fidelity approximation. The detailed algorithm is summarized in Algorithm.\ref{alg}.

\begin{algorithm}[ht]
\caption{bi-fidelity approximation}
\label{alg}
\label{BiFi-pod}

\Offline

Select a sample set $\Gamma = \{z_1, z_2, \hdots, z_M\}\subset I_z $.

Run the low-fidelity model $u_l(z_j)$ for each $z_j \in \Gamma$.

Select $N$ ``important" points from $\Gamma$ and denote it by $\gamma_N=\{z_{i_1}, \cdots z_{i_N} \} \subset\Gamma$. Construct the low-fidelity approximation space $U^L(\gamma_N)$.

Run high-fidelity simulations at each sample point of the selected sample set $\gamma_N$. Construct the high-fidelity approximation space $U^H(\gamma_N)$. 

\Online

For any given $z$, run the low-fidelity model to get the corresponding low-fidelity solution $u^L(z)$ and  compute the low-fidelity coefficients by projection:
$$u^L(z) \approx \mathcal{P}_{U^L(\gamma_N)}u^L = \sum_{k=1}^N c_k(z)u^L(z_k).$$

Construct the bi-fidelity approximation by applying the sample approximation rule learned from the low-fidelity model:
$$u^B(z)  = \sum_{k=1}^N c_k(z)u^H(z_k).$$

\end{algorithm}
\begin{comment}
{\bf bi-fidelityalgorithm} \\
$\bullet$ Low-fidelity simulations
\begin{itemize*}
\item -- Let $\Gamma=\{ z_1, \cdots, z_M \} \subset I_z$ be a set of points (rich enough to cover the parameter space $I_z$); 
\item -- Run the low-fidelity models at each point of $\Gamma$ to obtain $u^L(\Gamma)$ and $U^L(\Gamma)$. 
\end{itemize*}
$\bullet$ High-fidelity simulations
\begin{itemize*}
\item -- Select a subset of $N$ ``important" points from $\Gamma$ and denote it by $\gamma=\{z_{i_1}, \cdots z_{i_N} \} \subset\Gamma$; 
\item -- Run the high-fidelity models at each point of $\gamma$ to obtain $u^H(\gamma)$ and $U^H(\gamma)$.  
\end{itemize*}
$\bullet$ Bi-fidelity approximation
\begin{itemize*}
\item -- For every point $z\in I_z$ where an accurate solution is sought, use $U^L(\gamma)$ to define a projection operator $\mathcal P$ 
such that $u^L(z) \approx \mathcal P[u^L(\gamma)]$. 
\item -- Apply the same operator $\mathcal P$ to the high-fidelity solution ensemble $u^H(\gamma)$ to obtain 
an approximation solution at $z$: 
$$ u^B(z) := \mathcal P[u^H(\gamma)], $$
and $u^B(z)$ is the sought multi-fidelity solution that approximates $u^H(z)$. 
\end{itemize*}
\end{comment}

Most of the steps in this algorithm are straightforward. It would be instructional to provide details for Step 3 (point selection) and Step 6 (bi-fidelity reconstruction).

{\bf Point selection}. To select the subset $\gamma_N$, 
we shall search the parameter space by the greedy algorithm proposed in \cite{NGX14, ZNX14}. 
Start with a trivial subspace $\gamma_0=\varnothing$, and assume that the first $k-1$ important points $\gamma_{k-1}=\{ z_{i_1}, \cdots, z_{i_{k-1}}\} \subset \Gamma$ have been selected. We shall choose the next point $z_{i_k} \in \Gamma$ as the point that maximizes the 
distance between its corresponding low-fidelity solution and the approximation space $U^L(\gamma_{k-1})$, spanned by the low-fidelity solutions on the existing point set $\gamma_{k-1}$, i.e., 
\begin{equation}\label{Point} z_{i_k} = \arg\max_{z\in\Gamma}\text{dist}(u^L(z), U^L(\gamma_{k-1})), 
\qquad \gamma_k = \gamma_{k-1} \cup z_{i_k}, \end{equation}
where $\text{dist}(v,W)$ is the distance function between $v\in u^L(\Gamma)$ 
and subspace $W\subset U^L(\Gamma)$. 
The greedy procedure essentially serves the purpose of searching the linear independent basis set in the parameterized low-fidelity solution space.
We remark that the whole algorithm allows an efficient implementation by standard linear algebra operations. See \cite{NGX14, ZNX14} for more technical details. 

{\bf Bi-fidelity approximation}. 
In the offline stage, we have constructed the low- and high-fidelity approximation space, $U^L(\gamma_N)$ (step 3) and $U^H(\gamma_N)$ (step 4), respectively. During the online stage, for any 
given sample point $z\in I_z$, we shall project the corresponding low-fidelity solution $u^L(z)$ onto the low-fidelity approximation space
$U^L(\gamma_N)$: $$u^L(z) \approx \mathcal P_{U^L(\gamma_N)}[u^L(z)] = \sum_{k=1}^N c_k(z) u^L(z_{i_k}), $$ 
where 
$\mathcal P_{V}$ is the projection operator onto a Hilbert space $V$ 
and the corresponding projection coefficients $\{ c_k\}$ are computed by the following projection:
\begin{equation}\label{Gc}
 {\bf G}^L {\bf c} = {\bf f}, \qquad {\bf f} = (f_k)_{1\leq k\leq N}, \qquad f_k = \langle u^L(z), u^L(z_{i_k})\rangle^L, 
 \end{equation}
where ${\bf G}^L$ is the Gramian matrix of $u^L(\gamma_N)$, defined by 
\begin{equation}\label{GM}
 ({\bf G}^L)_{ij} = \langle u^L(z_{i_k}), u^L(z_{j_k}) \rangle^L, \qquad 1 \leq k, i, j \leq N, 
\end{equation}
with $\langle\cdot,\cdot\rangle^L$ the inner product associated with the approximation space $U^L(\gamma_N)$.

These low-fidelity coefficients $\{c_k\}$ serve as the surrogate of the corresponding high-fidelity coefficients of $u^H(z)$. Therefore,
the sought bi-fidelity approximation of $u^H(z)$ can be constructed as follows:  

\begin{equation}\label{UB}
 u^B(z) = \sum_{k=1}^N c_k(z) u^H(z_{i_k}). 
\end{equation}

We emphasize that if the low-fidelity model can mimic the variations of the high-fidelity model in the parameter space,  the low-fidelity coefficients can be a good approximation of the corresponding high-fidelity coefficients for a given sample $z$. We refer interested readers to \cite{NGX14, ZNX14,hampton2018practical} for details of the error analysis and justifications.

{\color{black}
It is worth noting that since the number of low-fidelity basis is} typically small ($\mathcal{O}(10)$ in our numerical tests), the cost of computing the low-fidelity projection coefficients by solving the linear system \eqref{Gc} is negligible. The dominant cost of the online step is one low-fidelity simulation run. If the low-fidelity solver is much cheaper than the high-fidelity solver, the speedup during the online stage can be significant. 

\subsection{The high- and low-fidelity models in our problem}
\label{HLmodel}
Our purpose is to efficiently approximate  high-fidelity solutions for the uncertain Boltzmann equation (\ref{model}) for a fixed $z$, which is solved by a deterministic AP solver discussed in section \ref{AP-section}. 
It is well known that existing solvers for deterministic kinetic equations are time-consuming and memory demanding due to its high-dimensional nature in the physical space. With the random parameter, it is more challenging to fully sweep the multi-dimensional parameter space by solving the Boltzmann equation repeatedly, especially given the complicated nonlinear collision operator in our model. 

To mitigate this computational cost, we consider to choose  the compressible Euler equations (\ref{Euler}) as our low-fidelity model. It is a first-order approximation to the Boltzmann equation, which can mimic the variations of macroscopic quantities of the Boltzmann equation in the fluid regime up to a certain accuracy. Besides, it is worth noting that 
the macroscopic quantities do not depend on the velocity $v$ in \eqref{INS}. Therefore, 
solving the deterministic Euler equation is much easier and more efficient 
{\color{black}in terms of} memory and computational time compared to solving the deterministic Boltzmann equation (\ref{Boltz}). 
These facts motivate us to choose the Euler equation as the low-fidelity model in our numerical experiments. 
{\color{black}A comparison of the computation cost (CPU time) for the two models are given in Section \ref{sec:4}.}

\subsubsection{A high-fidelity solver}
\label{AP-section}

To solve the high-fidelity model Boltzmann equation, we shall resort to a high-fidelity asymptotic-preserving (AP) solver. There have been many works in developing robust numerical schemes for kinetic equations in the framework of the asymptotic-preserving scheme, see for example \cite{Lemou-NS, Lemou-MM, Klar, JPT, Perthame, GJL-MM}. 
As pointed in \cite{FJ}, AP  scheme for the kinetic equation has two major merits:
1) as the Knudsen number $\e$ go to zero, it automatically becomes a consistent and stable scheme for the limiting fluid equation, with the stability condition independent of $\e$
(i.e., $\Delta t$ independent of $\e$); 2) the implicit collision terms can be implemented explicitly, free of Newton-type nonlinear algebraic solvers. 
Compared with multi-physics domain decomposition methods \cite{DJM}, AP schemes avoid the coupling of physical equations of different scales where coupling conditions and interface locations are difficult to determine. In contrast with many existing multiscale solvers,  the AP schemes only require solving one equation -- the kinetic equation and it becomes a robust macroscopic solver automatically when $\e\to 0$. 

For our problem, we shall employ an AP scheme developed in \cite{FJ} for the deterministic rescaled Boltzmann equation \eqref{Boltz} as our high-fidelity solver. 
The main idea of \cite{FJ} is to penalize the collision term $\mathcal Q(f,f)$ by the BGK operator
$P(f)=M-f$ which can be inverted easily, thus the scheme can be solved explicitly. Let the initial distribution function be $f_{\text{in}}$ and consider periodic boundary conditions.   
The basic scheme consists of the following two major steps:
\begin{itemize}
\item[1.] 
We first discretize \eqref{Boltz}  in time by the following first-order semi-discrete scheme:
\begin{equation}\label{S1}\frac{f^{n+1}-f^n}{\Delta t} +  v \cdot \nabla_x f^n = \frac{\mathcal Q(f^n, f^n)- \beta^n(M^n-f^n)}{\varepsilon} + 
\frac{\beta^n(M^{n+1}-f^{n+1})}{\varepsilon}. \end{equation}
$f^{n+1}$ can be rewritten as follows: 
\begin{equation}\label{eqn-f} f^{n+1} = \frac{1}{1+\frac{\Delta t}{\e}\beta^n}\left[f^n - \Delta t\, v\cdot \nabla_x f^n + 
\frac{\Delta t}{\e}\left(\mathcal Q(f^n, f^n) - \beta^n(M^n - f^n) + \beta^n M^{n+1}\right)\right], 
\end{equation}
where $\beta$ is some constant {\color{black}that depends on the spectral radius of the linearized collision operator of $\mathcal Q$ around the Maxwellian $M$. We refer to \cite[Section 2]{FJ} on the intuition and justification of choosing $\beta$.} For example, one can set 
$$\beta^n = \sup\left| \frac{\mathcal Q(f^n, f^n)}{f^n - M^n}\right|, $$
and other choices $\beta$ are also available \cite{FJ, GJL-MM}.
We numerically evaluate the collision term $\mathcal Q(f^n, f^n)$ in \eqref{S1} by applying the fast spectral method developed in \cite{FastSpectral}. 

\item[2.]
Though the above equation appears to be implicit due to $M^{n+1}$, it can be solved explicitly, thanks to the conservation property of 
$\mathcal Q(f,f)$ and the BGK operator $P(f)$. 
By multiplying the equation \eqref{S1} with the vector $m(v)$ in \eqref{mdef}, we can get the following equation: 
\begin{equation}\label{Macro-f} \frac{W^{n+1} - W^n}{\Delta t} + \nabla_x \cdot\langle v m f^n \rangle =0. 
\end{equation}
where  $W:=(\rho, \rho u, E)$ that consists of the macroscopic quantities (mass, momentum and energy).
With $W^{n+1}$, one computes $M^{n+1}$ from the Maxwellian \eqref{Maxwell}. Finally, we can update $f^{n+1}$ explicitly from \eqref{eqn-f}. 
\end{itemize}
For the spatial discretization in \eqref{eqn-f}, we employ a second order upwind MUSCL scheme as in \cite{FJ}, and a second order minmod slope limiter
is used to suppress possible spurious oscillations near discontinuities or sharp gradients \cite{LeVeque}. 
In addition to \eqref{Macro-f}, a second order TVD scheme with a minmod slope limiter is also applied, 
see \cite{Lemou-NS,GJL-MM} for details of implementation.  

\subsection{A low-fidelity solver}
For the low-fidelity model, instead of solving the Euler system (\ref{Euler}) directly, we shall semi-discretize its equivalent form (\ref{Macro-f}) with $f$ replaced by the Maxwellian $M(v)_{\rho,u,T}$, i.e., 
\begin{equation}\label{Euler-scheme} \frac{W^{n+1} - W^n}{\Delta t} + \nabla_x \cdot\langle v m M^n \rangle =0, 
\end{equation}
where the relation between $W:=(\rho, \rho u, E)$ and $M$ is given in (\ref{Energy}) and (\ref{Maxwell}). The initial data of $\rho$, $u$ and $E$ are obtained from the initial distribution $f_{\text{in}}$ for the Boltzmann equation, by using \eqref{macro} and \eqref{Energy}. That is, the initial data for the low- and high-fidelity models are consistent.  
The scheme \eqref{Euler-scheme} is numerically solved in the same way as equation (\ref{Macro-f}) in the 
AP solver for the Boltzmann equation. Since Euler system is marching the macroscopic quantities, instead of marching the distribution solution $f$ to the Boltzmann {\color{black} equation, 
the scheme} \eqref{Euler-scheme} can be solved with a much reduced computational cost and memory consumption.

\begin{remark}
We remark that instead of taking the solution $f$ to the Boltzmann equation via the scheme
(\ref{eqn-f}) as the high-fidelity solutions, we consider its corresponding macroscopic quantities of interest 
$\mathcal{U}^H(z)=[\rho^H(z),u^H(z),T^H(z)]^\top$ as the high-fidelity snapshot solutions in order to connect the macroscopic quantities computed from the low-fidelity models. 
The low-fidelity solutions $\mathcal{U}^L(z)=[\rho^L(z),u^L(z), T^L(z)]^\top$ we considered are computed from the Euler system by using (\ref{Euler-scheme}). During the point selection step to construct $\gamma_N$ in Algorithm 1, we shall select the important parameter points based on the concatenated macroscopic quantity snapshot, namely $\mathcal{U}^L(z)=[\rho^L(z),u^L(z), T^L(z)]^\top$.
\end{remark}

\begin{remark}
We acknowledge that there could be other choices of low-fidelity models that lead to more accurate bi-fidelity approximation, e.g., the compressible Navier-Stokes equations. To estimate if the low-fidelity model would be useful for constructing a reasonable accurate bi-fidelity approximation, one can explore an a priori estimate developed in the recent work \cite{hampton2018practical}.
\end{remark}

%------------------------------------------------------------------------------------
\section{Accuracy and convergence analysis}
\label{sec:3}

To establish the accuracy and convergence results,  we first give a summary of the hypocoercivity framework and notations used in \cite{LJ-UQ}, 
then introduce the relation between the solutions to the Boltzmann and compressible Euler system in suitable norms. 
To study the difference between the high- and bi-fidelity solutions, 
one can split it into two parts: the projection error 
and the remainder. 
In section \ref{sec:p1}, 
{\color{black}we show the estimate for the projection error in Theorem \ref{Thm1}}. 
In section \ref{sec:p2}, we give the regularity of high-fidelity solution, then {\color{black}prove}
the accuracy and convergence results of our bi-fidelity method adapted to the Boltzmann equation in Theorem \ref{MainThm}. 

\subsection{The projection error}
\label{sec:p1}
The subject of hydrodynamic limits and rigorous derivations of macroscopic models such as the fundamental PDEs of fluid mechanics 
from the kinetic theory of gases is a challenging task and has been studied for decades, see for example \cite{DP-Lions, TPLiu, Russel, Golse1, Golse2, JLM, Guo}. 
We shall show that for each fixed $z\in I_z$, 
the error between solutions to the Euler system and the macroscopic quantities (\ref{macro}) 
obtained from the Boltzmann equation (with consistent initial data) is small and of order $\e$, which will be described in (\ref{h_Error}). 

{\bf Hypocoercivity framework}. First, we review the hypocoercivity framework and notations for the norms used in \cite{LJ-UQ}. 
Let $f$ be the solution to the Boltzmann equation (\ref{Boltz}). 
Consider a linearization around the global equilibrium and perturbation of $f$: 
$$ f = \mathcal M + \e \sqrt{\mathcal M}\, h, $$
with $$\mathcal M(v) = \frac{1}{(2 \pi)^{\frac{d_v}{2}}} e^{-\frac{|v|^2}{2}}, $$
then $h$ satisfies the perturbed equation 
\begin{equation}
\label{h-PE}
 \partial_t h + v\cdot\nabla_x h = \frac{1}{\e}\mathcal L(h) + \mathcal F(h, h),  
\end{equation} 
where the linearized operator $\mathcal L$ and the nonlinear operator $\mathcal F$ are given by 
\begin{align*}&\displaystyle 
\mathcal L(h) = \left(\sqrt{\mathcal M}\right)^{-1} \left[\mathcal Q(\sqrt{\mathcal M}h, \mathcal M)
+ \mathcal Q(\mathcal M, \sqrt{\mathcal M}h)\right], \\[2pt]
&\displaystyle \mathcal F(h, h) = 2 \left(\sqrt{\mathcal M}\right)^{-1}\mathcal Q(\sqrt{\mathcal M}h, \sqrt{\mathcal M}h). 
\end{align*}
Denote $\partial_l^j := \partial/\partial v_j\, \partial/\partial x_l$ for multi-indices $j$ and $l$. Introduce the following Sobolev norms: 
\begin{align}
\label{norm}
\begin{split}
\displaystyle ||h||_{H_{x,v}^s}^2 & = \sum_{|j| + |l| \leq s} || \partial_l^j h||_{L_{x,v}^2}, \qquad
||h||_{H_{x,v}^{s,r}} = \sum_{|\nu|\leq r} ||\partial^{\nu} h||_{H_{x,v}^s}^2, \\[4pt]
\displaystyle ||h||_{H_{x,v}^s H_z^r} & = \int_{I_{z}} ||h||_{H_{x,v}^{s,r}}\, \pi(z)d z, 
\qquad ||h||_{H_{x,v}^{s,r} L_{z}^{\infty}} = \sup_{z\in I_{z}} ||h||_{H_{x,v}^{s,r}}. 
\end{split}
\end{align}

Refer to \cite[Theorem 2.5]{MB15}, we extend its analysis to our 
case of the Boltzmann equation in the acoustic regime. 
Let $h_{\e}$ be the perturbed solution to the linearized equation (\ref{h-PE}). Suppose the initial data for (\ref{h-PE}) and (\ref{Euler}) are consistent for each $z$. 
If the initial distribution $h_{in} \in \text{Null}(\mathcal L)$ and $h_{in}\in H_{x,v}^s$, then for each $z$, $\left(h_{\e} \right)_{\e>0}$ converges strongly to 
$$ h(t,x,v,z) = \left[\rho(t,x,z) + v\cdot u(t,x,z) + \frac{1}{2}(|v|^2 - d_v) T(t,x,z)\right]\mathcal M(v)$$
in $L_{[0,T]}^2 H_x^s L_v^2$ as the Knudsen number $\e\to 0$, where $\rho$, $u$, $T$ (with $E$ obtained by \eqref{Energy}) satisfy the Euler system (\ref{Euler}). 
We adapt our acoustic scaling to \cite[Theorem 2.5]{MB15} and get the follows: For all $z\in I_z$, 
if $h_{in}$ belongs to $H_x^s L_v^2$,
then 
\begin{equation}\label{h_Error}
\sup_{t\in[0,\infty]} || h(t,z) - h_{\e}(t,z)||_{L^2_{x,v}} \leq 
\sup_{t\in[0,\infty]} || h(t,z) - h_{\e}(t,z)||_{H_x^s L_v^2}\leq 
C \max\{ \e, \, \e V_T(\e)\}, 
\end{equation}
where $\forall T>0$, $V_T(\e)$ is defined as
$$V_T(\e)=\sup_{t\in [0, T]}||h(t,z) - h_{\e}(t,z)||_{L_x^{\infty}L_v^2} \to 0, \quad\text{ as   }\e\to 0. $$

{\bf Error splitting}.
Let 
$\left\langle\,\cdot\,\right\rangle^H$ be an inner product space corresponding to the high-fidelity solution and $\left|\left|\,\cdot\,\right|\right|^H$ be the corresponding induced norm, see \cite{NGX14}. For each $z$, 
to study the total error $\left|\left| u^H(z) - u^B(z)\right|\right|^H$, one can split it into two parts: 
\begin{equation}
\label{total-error}
 \left|\left| u^H(z) - u^B(z)\right|\right|^H \leq \left|\left| u^H(z) - P_{U^H(\gamma_N^L)}u^H(z) \right|\right|^H + \left|\left| P_{U^H(\gamma_N^L)}u^H(z) - u^B(z)\right|\right|^H.     
\end{equation}
\cite[Lemma 4.3]{NGX14} shows the estimate for the second term: 
\begin{equation}\label{TermII}
 \left|\left| P_{U^H(\gamma_N^L)}u^H(z) - u^B(z)\right|\right|^H \leq c \left|\left| P_{U^H(\gamma_N^L)}u^H(z)\right|\right|^H + 
\left|\left| \sqrt{{\bf G}^H}({\bf G}^L)^{-1} {\bf Q}\,{\bf f}^L\right|\right|, 
\end{equation}
{\color{black}where we refer $c$ to $\e_1 + \e_2 + \e_1 \e_2$ in \cite[Lemma 4.3]{NGX14}, which is small, based on the reasonable assumptions made there.}
The last term above is related to the non-invertibility of high-fidelity Gramian matrix and usually negligible. Here ${\bf G}^L$ (or ${\bf G}^H$) is the Gramian matrix of $u^L(\gamma)$ (or $u^H(\gamma)$) given by \eqref{GM}, the vector ${\bf f}^L$ has entries
$$ f_n^L = \left\langle u^L({z}_n), u^L(z) \right\rangle^L, $$
and ${\bf Q}:={\bf I} - {\bf P}$ is the orthogonal projection onto its kernel (with ${\bf P}$ the orthogonal projection matrix onto its range), see \cite{NGX14}. 
In addition, since $\left|\left| P_{U^H(\gamma_N^L)}u^H(z)\right|\right|^H \leq ||u^H(z)||^H$, thus 
\begin{equation}
\label{PU}
\left|\left| P_{U^H(\gamma_N^L)}u^H(z) - u^B(z)\right|\right|^H \leq c\, ||u^H(z)||^H + 
\left|\left| \sqrt{{\bf G}^H}({\bf G}^L)^{-1} {\bf Q}\,{\bf f}^L\right|\right|. 
\end{equation} 

{\bf Projection error}. The rest of this section will study the estimate for the projection error (the first term on the right-hand-side of (\ref{total-error})) and conclude it in Theorem \ref{Thm1}. 
We now adapt the analysis in \cite[subsection 4.1]{NGX14} and incorporate our high-fidelity (Boltzmann) and low-fidelity (Euler) models, by utilizing the knowledge of (\ref{h_Error}). 
Denote
\begin{subequations}
\begin{align}
& \label{zH} z_n^H = \arg\max_{z\in I_{z}} d^H(u^H(z), U^H(\gamma_{n-1}^H)), \qquad \gamma_n^H = \gamma_{n-1}^H \cup \{ z_n^H \},  \\[2pt]
&\label{zL} z_n^L = \arg\max_{z\in I_{z}} d^L(u^L(z), U^L(\gamma_{n-1}^L)), \qquad \gamma_n^L = \gamma_{n-1}^L \cup \{ z_n^L \}. 
\end{align}
\end{subequations}
The ``best" achievable distance for approximation from a general $N$-dimensional subspace is the Kolmogorov $N$-width, defined by 
\begin{equation}\label{K-width}
d_N^H (u^H(I_{z})) = \inf_{\text{dim}(V_N)=N} \sup_{v\in u^H(I_{z})} d^H(v, V_N). 
\end{equation}
The following is similar to \cite[Lemma 4.1]{NGX14}, 
except that we use the inequality (\ref{h_Error}). Since it is lengthy, we present as Lemma \ref{LM} with its proof in the Appendix. 

\begin{theorem}
\label{Thm1}
If all the assumptions in Lemma \ref{LM} are satisfied for each $n=0, 1, \cdots, N-1$, then 
\begin{equation}\label{uP} 
\sup_{z\in I_{z}} ||u^H(z) - P_{U^H(\gamma_N^L)}u^H(z)||^H \leq C \sqrt{d_{N/2}(u^H(I_{z}))}\,, 
\end{equation}
where $d_N(u^H(I_{z}))$ is the $N$-width of the functional manifold $u^H(I_{z})$, and the constant 
$C=\sqrt{2}(\delta_1 - \delta_2\,\e)^{-1}$ 
with 
$0< \delta_1 - \delta_2\, \e<1$. 
\end{theorem}
\begin{proof}
Lemma \ref{LM} shows that it is a weak greedy procedure to use the nodal choices of $z_n^L$: 
$\exists\, 0< \delta_1 - \delta_2\,\e <1$ 
such that for all $n=0, 1, \cdots, N-1$, 
$$ d^H (u^H(z_n^L), U^H(\gamma_{n-1}^L)) \geq (\delta_1 - \delta_2\,\e)\sup_{z\in I_{z}}\, d^H(u^H(z), U^H(\gamma_{n-1}^L)), $$ 
and \cite[Corollary 3.3]{DPW13} indicates that (\ref{uP}) holds with 
$C=\sqrt{2}(\delta_1 - \delta_2\,\e)^{-1}$. 
\end{proof}

%---------------------------------------------------------------------------------------
\subsection{Smoothness of the solution and convergence analysis} 
\label{sec:p2}

In this section, we first study the smoothness of the high-fidelity solution 
$u^H: I_z \mapsto V^H$, where $z$ is a multivariate random parameter and 
$V^H$ is a Hilbert space with an inner product $\left\langle \,\cdot,\cdot \right\rangle^H$. 
Then we shall establish an estimate bound for the Kolmogorov width in Theorem \ref{Thm1}, finally combine all the arguments and show the convergence result for our bi-fidelity procedure in Theorem \ref{MainThm}.  

We introduce the standard multivariate notation. Denote the countable set of ``finitely supported" sequences of nonnegative integers by
$$ \mathcal F:= \left\{ \nu = (\nu_1, \nu_2, \cdots ): \nu_j \in \mathbb N, \, \text{ and } \nu_j \neq 0 \text {  for only a finite number of } j\right\}, $$
with $|\nu|:=\sum_{j \geq 1} |\nu_j|$. 
For $\nu \in \mathcal F$ supported in $\{1, \cdots, J \}$, one defines the partial derivative in $z$
\begin{equation}
\partial^{\nu} u = \frac{\partial^{|\nu|} u}{\partial^{\nu_1}z_1  \cdots \partial^{\nu_J}z_J},   
\end{equation} 
and the multi-factorial $\nu! := \prod_{j \geq 1}\nu_j !$, where $0! :=0$.  

{\bf Regularity of $u^H$.}
It is not hard to extend the sensitivity analysis in \cite{LJ-UQ} to the multi-dimensional random variable case, see 
Remark 2.8 in \cite{SJ-TwoPhase} on the extension. \cite{LJ-UQ} tells us that 1) the uncertainties from the initial data and collision kernel (under suitable assumptions) will eventually diminish and the solution will exponentially decay in time to the deterministic global equilibrium $\mathcal M$; 2) the regularity of the initial data in the random space is preserved at later time. 

Let $h_{\e}$ be a perturbed solution to the equation \eqref{h-PE}. 
If its random initial data satisfies {\color{black}for all $r$},
\begin{equation}\label{h-IC} 
||h_{\e}^{in}(z)||_{H_{x,v}^{s,r} L_{z}^{\infty}} \leq C_I, \end{equation}
then at time $t>0$, 
$$ ||h_{\e}(t,z)||_{H_{x,v}^{s,r} L_{z}^{\infty}} \leq {\color{black}C_r}\, e^{-\e \tau t},
 $$
%{\color{black}DELETE THIS: Since $||\partial^{\nu} h_{\e}(t,z)||_{H_{x,v}^s L_{z}^{\infty}} \leq ||h_{\e}(t,z)||_{H_{x,v}^{s,r} L_{z}^{\infty}}$. }
{\color{black}Moreover, $h$ is {\it analytic} in the random space, meaning that}
\begin{equation}\label{h-Z}
||\partial^{\nu} h_{\e}(t,z)||_{H_{x,v}^s L_{z}^{\infty}} \leq
{\color{black} C\, e^{- \e \tau t} (|\nu|!), \qquad \text{for any } \nu \geq 0. }
\end{equation}
{\color{black}See Appendix B for a discussion on the linearized Boltzmann case.}

We now use a weaker version by letting $s=1$ in \eqref{h-IC} and \eqref{h-Z}. 
If the initial data
$||h_{\e}^{in}(z)||_{H_{x,v}^{1,r} L_z^{\infty}} \leq C_I$, then
$$||\partial^{\nu} h_{\e}(t,z)||_{L^2_{x,v} L_z^{\infty}} \leq ||\partial^{\nu} h_{\e}(t,z)||_{H_{x,v}^1 L_z^{\infty}} 
\leq C\, e^{-\e \widetilde\tau t}. $$ 
By the definition of $u$ containing perturbed macroscopic quantities (see \cite[section 2.2.4]{MB15}) by the Cauchy-Schwarz inequality, 
one easily gets
\begin{equation} 
\label{u-Anal}
||\partial^{\nu} u(t,z)||_{L^2_x L_z^{\infty}} \lesssim ||\partial^{\nu} h_{\e}(t,z)||_{L^2_{x,v} L_z^{\infty}} \leq  C\, e^{- \e\tau t},
\end{equation}
for $|\nu|\leq r$, where $C$ and $\tau>0$ are all  generic constants independent of $\e$. 
We assume that the high-fidelity solution $u^H$ follows a similar behavior as the analytic solution $u$ (computed from $h_{\e}$) to the Boltzmann equation, with an error that depends on the numerical scheme used in the high-fidelity model. 
If the initial distribution of the high-fidelity model satisfies 
\begin{equation}
\label{u-IC}
 ||h_{\e}^{in}(z)||_{H_{x,v}^{1,r} L_z^{\infty}} \leq C_I,
\end{equation}
then for a fixed time $t>0$ and $|\nu|\leq r$, 
\begin{equation}\label{Anal}
 \sup_{z\in I_z}||\partial^{\nu} u^H(t,z)||^H \leq C^{\prime}\, e^{- \e \tau t} + \xi, 
\end{equation}
where $\xi$ depends on the order and discretization parameters $\Delta t$, $\Delta x$, $\Delta v$ used in the high-fidelity solver.
Thus for all $z\in I_z$, one gets 
$$ ||\partial^{\nu} u^H(t,z)||^H \leq C^{\prime}\, e^{- \e \tau t} + \xi. $$
Note that $C_I$, $C$, $C^{\prime}$ and $\tau$ in the inequalities \eqref{h-IC}--\eqref{Anal} 
are all positive generic constants independent of $\e$. 
\\[2pt]

We make the following assumption on the random collision kernel: 
\begin{assumption}
\label{assp1}
Assume the collision kernel take the form
\begin{equation}\label{RB1}B(|v-v^{\ast}|, \cos\theta, z) = \Phi(|v-v^{\ast}|) b(\cos\theta, z), \, \Phi(|v-v^{\ast}|) = C |v-v^{\ast}|^m, \, m \in [0,1], \end{equation}
and $\left(\psi_j \right)_{j\geq 1}$ be an affine representer (see definition in \cite{Cohen15}) of the cross section $b$, that is, 
\begin{equation}
\label{Assump-b}
  b(\eta, z) = \bar b(\eta) + \sum_{j \geq 1}z_j \psi_j(\eta), \qquad z :=\left( z_j \right)_{j \geq 1}, \quad \eta = \cos\theta,    
\end{equation}
where the sequence $\left( ||\psi_j||_{L^{\infty}(\eta)}\right)_{j\geq 1}\in \ell^p$ for $0<p<1$ (see \cite{Cohen15}). 
One also assumes that 
\begin{equation}\label{RB2}|b(\eta, z)| \leq C_0, \quad |\partial_{\eta}b(\eta, z)| \leq C_1, \quad
|\partial^{\nu}b(\eta, z)| \leq C_2, 
\end{equation}
for all $\eta\in[-1,1]$, $|\nu|\leq r$. Here $C$, $C_0$, $C_1$, $C_2$ are all positive constants. 
\end{assumption}
\noindent
The above (\ref{RB1})--(\ref{RB2}) extend the conditions in \cite[Theorem 4.4 (ii)]{LJ-UQ} from 
one-dimensional random space to the multi-dimensional random space. 
\hspace{2cm}

{\bf An upper bound for $N$-width.}
We can now utilize the result in \cite[Corollary 3.11]{Cohen15}, which works for general parametric PDEs. 
Define the norm 
$$ || u - v ||_{L^{\infty}(I_z, V)} := \sup_{z\in I_z}||u(z)- v(z)||_{V},$$
where $V$ is the physical space considered, in our case 
$V=L^2_x$.   
Under {\bf Assumption \ref{assp1}}, by the analyticity of $u^H$ given in (\ref{Anal}), for $z\in I_z$, one obtains
\begin{equation}
\label{Col}
\left| \left| u^H - \sum_{\nu\in\Lambda_N}w_{\nu}P_{\nu}\right|\right|_{L^{\infty}(I_z,V)} \leq C (N+1)^{-q}, \qquad q=\frac{1}{p}-1, 
\end{equation}
where $\left(P_k\right)_{k \geq 0}$ is the sequence of renormalized Legendre polynomials on $[-1,1]$, $\Lambda_N$ is the set of indices that corresponds to the $N$ largest $||w_{\nu}||_V$, and the constant 
$C := \left|\left|\left(||w_{\nu}||_V\right)_{\nu\in\mathcal F}\right|\right|_{\ell^p} <\infty$. By (\ref{Anal}), one {\color{black}formally gets} $C= c\, e^{-\e\tau t} + \xi$. 

Recall (\ref{K-width}) and the definition $d^H(v,V_N) :=\min_{w\in V_N}||v-w||_V$ (refer to \cite[equation (8.3)]{Cohen15}), 
the best achievable error in $L^{\infty}(I_z, V)$ is described by the $N$-width of the solution manifold $\mathcal M=u^H(I_z)$:
\begin{equation}
 d_N(\mathcal M)_V = \inf_{\text{dim}(V_N)=N}\sup_{v\in\mathcal M}\min_{w \in V_N} ||v - w||_V.    
\end{equation} 
One can use the polynomial approximation bound in $L^{\infty}(I_z,V)$ to estimate an upper bound for $d_N(\mathcal M)_V$, 
as studied in \cite[Section 4]{Cohen15}. 
Let the $N$-dimensional subspace $$ V_N := \text{span}\left\{ c_{\nu} : \nu\in\Lambda_N\right\}. $$ 
For $N\geq 1$, one observes that 
\begin{align}
\label{dN}
\begin{split}
\displaystyle d_N(\mathcal M)_V &\leq \sup_{v\in\mathcal M}\min_{w\in V_N} ||v-w||_V = \sup_{z\in I_z}\min_{w\in V_N} ||u^H(z) - w||_V \\[2pt]
\displaystyle & \leq ||u^H -\sum_{\nu\in\Lambda_N} w_{\nu}P_{\nu}||_{L^{\infty}(I_z,V)} \leq C (N+1)^{-q}, \quad q = \frac{1}{p}-1. 
\end{split}
\end{align}
where $\sum_{\nu\in\Lambda_N} w_{\nu}P_{\nu}$ is the truncated Legendre expansion and (\ref{Col}) is used in the last inequality. 
\\[2pt]

We now conclude with our main result on convergence analysis: 
\begin{theorem}
\label{MainThm}
If the assumptions for the random initial data, random collision kernel, namely (\ref{u-IC}) and {\bf Assumption \ref{assp1}} are satisfied, for fixed time $t>0$ and fixed numerical discretization parameters $\Delta t$, $\Delta x$ and $\Delta v$, then 
for all $z\in I_z$,
\begin{equation}
\label{MainError}
\left|\left|u^H(t,z)- u^B(t,z)\right|\right|^H \leq C_1\, \frac{e^{-\frac{\e \tau t}{2}}}{(N/2+1)^{q/2}} + C_2\, e^{-\e\tau t} + \chi
+ \left|\left| \sqrt{{\bf G}^H}({\bf G}^L)^{-1} 
{\bf Q}\,{\bf f}^L(z)\right|\right|,   
\end{equation} 
where $N$ is the size of the subspace $\gamma_N$ in Algorithm 1, {\color{black} and} $q$ is given in (\ref{Col}) with $p$ depending on 
the $\ell^p$-summability assumption of $(\psi_j)_{j \geq 1}$, 
$C_1 = \mathcal{O} \left(\frac{1}{\delta_1 - \delta_2\e}\right)$, 
$C_2$ and $\tau$ are constants that depend on the initial data $u^{in}$ and {\bf Assumption \ref{assp1}} on the collision kernel. $\delta_1$, $\delta_2$ are all sufficiently small with 
$0<\delta_1 - \delta_2\,\e<1$. Definitions of ${\bf G}^L$, ${\bf G}^H$, ${\bf Q}$ and $f^L$ are given below \eqref{TermII}. 
$\chi$ is associated to the order and discretization mesh in the high-fidelity solver.  
$\e$ is the Knudsen number in the Boltzmann equation (\ref{Boltz}). 
\end{theorem}
\begin{proof}
According to the inequalities (\ref{total-error}), (\ref{PU}), (\ref{dN}), (\ref{Anal}) and Theorem \ref{Thm1}, one gets 
for all $z\in I_z$,
\begin{align*}
\displaystyle \left|\left| u^H(t,z) - u^B(t,z)\right|\right|^H & \leq C \sqrt{d_{N/2}(u^H(I_{z}))} +  c \, ||u^H(z)||^H + \left|\left| \sqrt{{\bf G}^H}({\bf G}^L)^{-1} {\bf Q}\,{\bf f}^L(z)\right|\right| \\[4pt]
\displaystyle & \leq  C_1 \, 
\frac{\sqrt{e^{-\e\tau t} + \xi}}{(N/2+1)^{q/2}} + C_2\,e^{-\e\tau t} + \xi + \left|\left| \sqrt{{\bf G}^H}({\bf G}^L)^{-1} {\bf Q}\,{\bf f}^L(z)\right|\right| \\[4pt]
\displaystyle &  \leq C_1 \frac{e^{-\frac{\e \tau t}{2}}}{(N/2+1)^{q/2}} + C_2\,e^{-\e\tau t} + \chi + \left|\left| \sqrt{{\bf G}^H}({\bf G}^L)^{-1} {\bf Q}\,{\bf f}^L(z)\right|\right|, 
\end{align*}
here $C$, $C_1$, $C_2$, $\tau$ are all generic positive constants independent of 
$\e$, {\color{black}$C_2= c\, C^{\prime}$ which is small}, and $\chi=C_1\sqrt{\xi}+\xi$. 
\end{proof}

Theorem \ref{MainThm} indicates that the error between the bi- and high-fidelity solutions decays algebraically with respect to the number of high-fidelity runs $N$. 
The convergence rate $q/2$ is independent of the dimension of the random space and the regularity of the initial data; it only relates to the $\ell^p$ summability of the
affine representer $(\psi_j)_{j \geq 1}$ 
in {\bf Assumption 1}. 
\begin{remark}
We make the following remarks: 
 \begin{itemize}
\item[1.] 
The estimate in Theorem \ref{MainThm} may not be sharp. Deriving a sharper estimate requires a better understanding on the role of the Knudsen number $\e$ in the accuracy analysis. 
\item[2.] 
It is not our goal of the current work to establish stability and error analysis for the deterministic AP method for the multiscale Boltzmann
equation in \cite{FJ}, which is difficult due to the penalization used in the scheme thus has not been studied to our best knowledge. Thus 
deriving \eqref{Anal} from \eqref{u-Anal} rigorously remains challenging. 
\end{itemize}
We hope to report more results regarding the above two issues in future researches. 
\end{remark}

%------------------------------------------------------------------------------------
\section{Numerical Tests}
\label{sec:4}

To examine the performance of the method, we shall compute numerical errors in the following way: we choose a fixed set of points $\{ \hat z_i \}_{i=1}^n \subset I_z$ that is independent of the point sets $\Gamma$, 
and evaluate the following error between the bi-fidelity and high fidelity solutions at a final time $t$: 
 \begin{equation}\label{Error} ||u^H(t) - u^B(t)||_{L^2(D \times I_z)} 
 \approx\frac{1}{n}\sum_{i=1}^n ||u^H(\hat z_i, t) - u^B(\hat z_i, t)||_{L^2(D)}, 
 \end{equation}
 where $||\cdot||_{L^2(D)}$ is the $L^2$ norm in the physical domain $D=\Omega\times\mathbb R^2$. The error can be considered as an approximation 
 to the average $L^2$ error in the whole space of $D\times I_z$. 

Since our goal is to numerically solve the multiscale Boltzmann equation with random inputs, we solve the Boltzmann equation (\ref{Boltz}) as the high-fidelity AP solver discussed in Section~\ref{HLmodel}.
 We assume the random collision kernel in the form of 
 \begin{equation}
 \label{B-Form} B(v, v_{\ast}, \sigma, z) = b(z)|v-v_{\ast}|^{\lambda}, 
\end{equation}
and consider Maxwell molecules, i.e., $\lambda=0$ in \eqref{B-Form}. The low-fidelity model is chosen as the Euler equation solved by 
the forward Euler in time and second-order MUSCL scheme in space, by using the same spatial and temporal resolutions as the high-fidelity model, but with a different number of quadrature points $N_v^l$ in velocity discretization.

In all the examples, the spatial domain is chosen to be  $[0,1]$ with $N_x$ grid points, and periodic boundary condition is assumed except for the shock tube tests. 
The velocity domain is chosen as $[-L_v, L_v]^2$ with $L_v=8.4$ and $N_v$ grid points in each dimension. 
Without loss of generality, {\color{black}the} $d$-dimensional random variable ${\bf z}$ 
is assumed to follow the uniform distribution on $[-1,1]^d$. 
The training set $\Gamma$ is chosen to be $M=1000$ random samples of ${\bf z}$. 
We examine the error of bi-fidelity approximation with respect to the number of high-fidelity runs by computing the norm defined in (\ref{Error}) 
(evaluated over an independent set of $n=1000$ Monte Carlo samples). {\color{black} All numerical experiments are conducted by  MATLAB R2018b on macOS Mojave system with 2.4 GHz Intel Core i5 processor and 8GB DDR3 memory}),

\subsection{A double-peak initial data test}
\label{Test1}
We first consider the following initial data to mimic the Karhunen-Loeve expansion of the random field: 
\begin{align}
\label{IC}
\left\{
\begin{array}{l}
\displaystyle \rho_0(x, {\bf z}^{\rho})= \frac{1}{3}\left( 2 + \sin(2\pi x) + 0.2\sum_{k=1}^{d_1} \sin[2\pi(k+1)x]\, \frac{z_k^{\rho}}{2 k}\right), \\[10pt]
\displaystyle {\bf u}_0 = (0.2, 0), \\[10pt]
\displaystyle T_0(x, {\bf z}^T) = \frac{1}{4}\left( 3 + \cos(2\pi x) + 0.2\sum_{k=1}^{d_1} \cos[2\pi(k+1)x]\, \frac{z_k^T}{2 k}\right), \\[12pt]
\displaystyle f_0(x, {\bf v}, {\bf z}) = \frac{\rho_0}{4\pi T_0}\left( \exp(-\frac{|\bf{v} - \bf{u}_0|^2}{2 T_0}) + \exp(-\frac{|\bf{v} + \bf{u}_0|^2}{2 T_0}) \right). 
\end{array}\right.
\end{align}
The uncertain collisional cross section is given by 
\begin{equation}\label{R-K}
b(z)=1+0.5z_1^b, 
\end{equation}
Here ${\bf z}^{\rho}=\left(z_1^{\rho}, \cdots, z_{d_1}^{\rho}\right)$, ${\bf z}^T=\left(z_1^T, \cdots, z_{d_1}^T\right)$, and ${\bf z}^b = z_1^b$ represent the random variables in the collision kernel, initial density and temperature. Let the initial distribution $f_0$ follow a double-peak non-equilibrium initial data \cite{FJ}. 
Set $d_1=7$, thus this is a $d=15$ dimensional problem in the random space. 
We use the Boltzmann equation as the high-fidelity model and the Euler system as the low-fidelity model, set $\Delta x=0.01$, $\Delta t = 8\times 10^{-4}$ 
(in both the high- and low-fidelity models), $N_v^h=16$, and the final time $t=0.1$. 
 
In Figure \ref{Fig1}, we consider the fluid regime with $\e=10^{-4}$. 
This figure shows the mean $L^2$ errors of $\rho$, $u_1$, $T$ 
 between the high- and bi-fidelity solutions with different quadrature points in velocity space. Here $\bf{u}_1$ in the figures below stands for the first component of the two-dimensional bulk velocity ${\bf u}$. It is clear that the error decays fast with the number of high-fidelity runs. In addition, when $N_v^l$ increases, the error between the high- and bi-fidelity solutions decreases. This is expected because the Euler equation solved by more quadrature points in velocity space can capture more information about the high-fidelity model.  
   
In Figure \ref{Fig3}, fluid regime is considered and we vary $\e$  from $\e=10^{-2}$ to $\e=10^{-4}$. The Euler equation is chosen as the low-fidelity model, solved by the
same forward Euler in time and second-order MUSCL scheme in space, and the same spatial and temporal meshes as the Boltzmann equation in the high-fidelity model, 
and with $N_v^l=8$ velocity quadrature points. 
One observes that the smaller $\e$ is, the lower level the errors saturate. This is expected, because when the Knudsen number $\e$ approaches to zero, the Euler equation as the low-fidelity model commits less modeling error and can capture more information of the high-fidelity model. 

In Figure \ref{Fig2}, we investigate the performance of the bi-fidelity approximation for the kinetic regime with $\e=1$. 
Fast convergence of the mean $L^2$ errors with respect to the number of high-fidelity runs is observed. 
Even though $\e$ is relatively large compared to the previous two tests, a satisfactory accuracy in characterizing the behaviors of the solution in the random space is achieved in both cases: $N_v^l=8$ and  $N_v^l=4$; and the errors with $N_v^l=8$ is smaller than that of $N_v^l=4$. 
On the right column of Figure \ref{Fig2}, we plot the high-, low- and the corresponding bi-fidelity solutions (with $r=20$, $N_v^l=8$) for a particular sample point $z$. One observes that the high- and bi-fidelity solutions match quite well, whereas the low-fidelity solutions appear inaccurate at some spatial points. 
This example seems to indicate that although
in the kinetic regime, the fluid description 
breaks down  in the {\it physical space}, 
the bi-fidelity solution can still capture important variations of the high-fidelity model (Boltzman equation) in the {\it random space}. 
\begin{figure}[H]
\centering
    \begin{subfigure}[b]{1\textwidth}
    \centering
\includegraphics[width=0.55\linewidth]{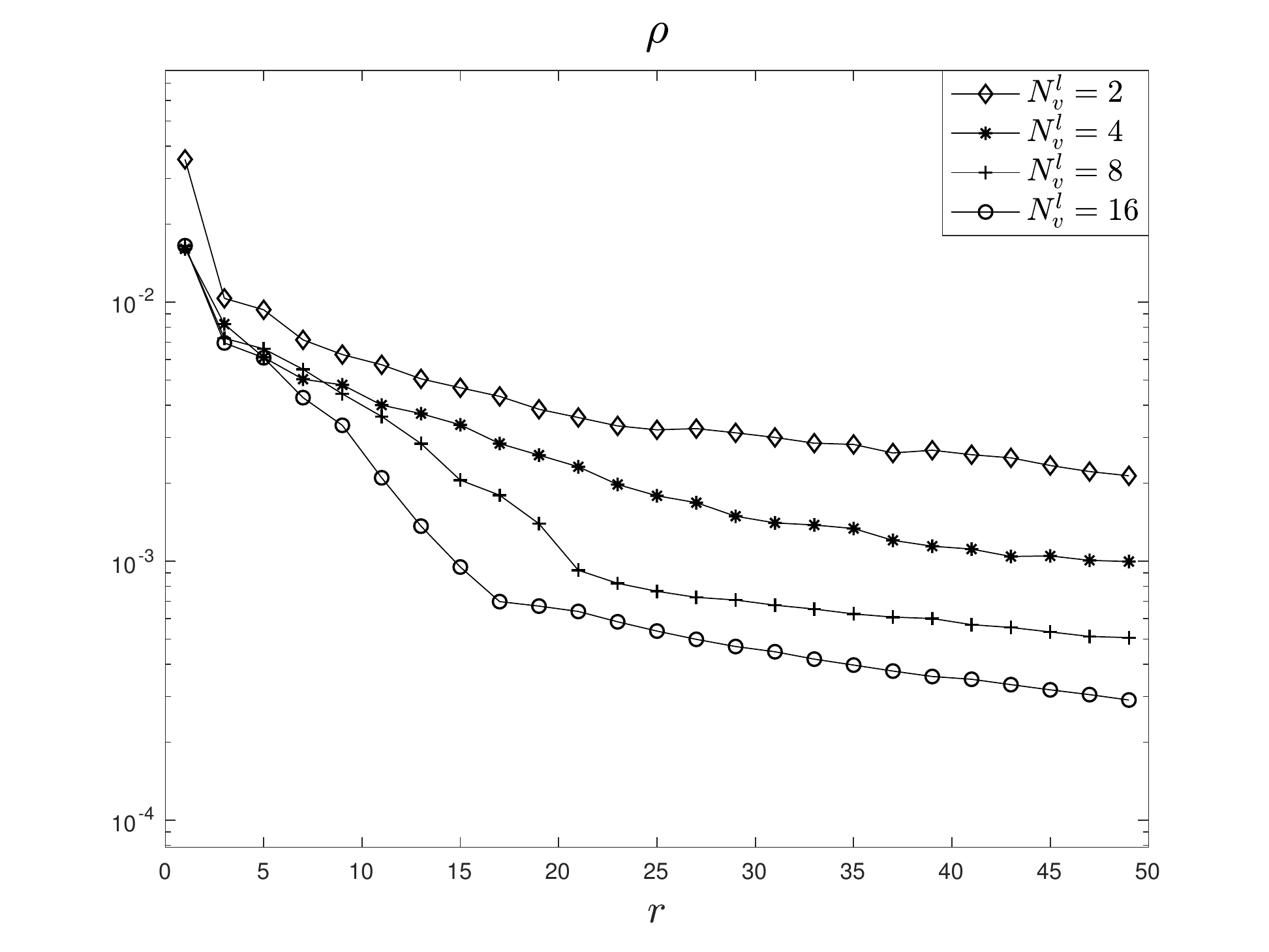}
\caption{}
\end{subfigure}
\begin{subfigure}[b]{1\textwidth}
    \centering
\includegraphics[width=0.55\linewidth]{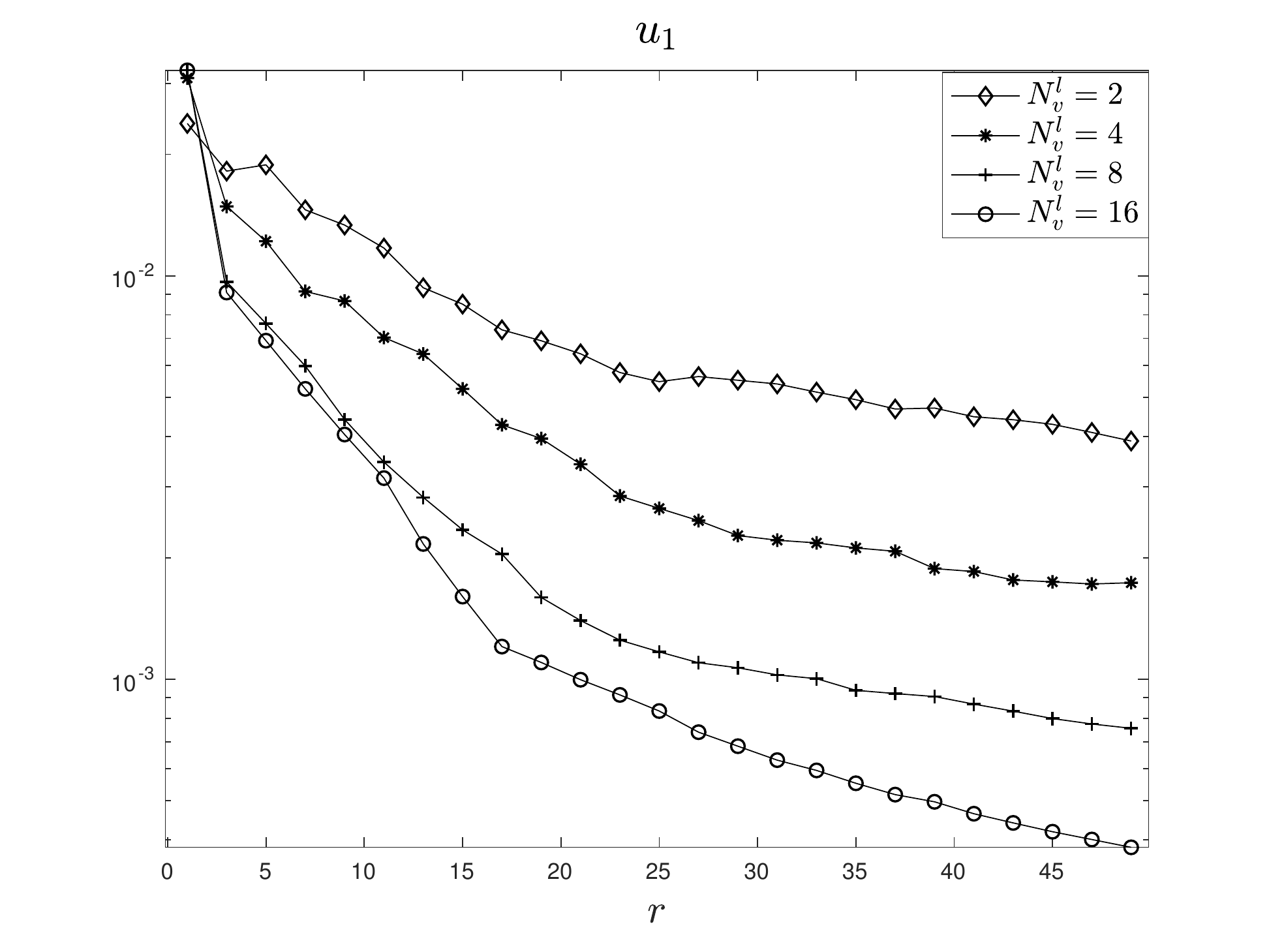}
\caption{}
\end{subfigure}
\begin{subfigure}[b]{1\textwidth}
    \centering
\includegraphics[width=0.55\linewidth]{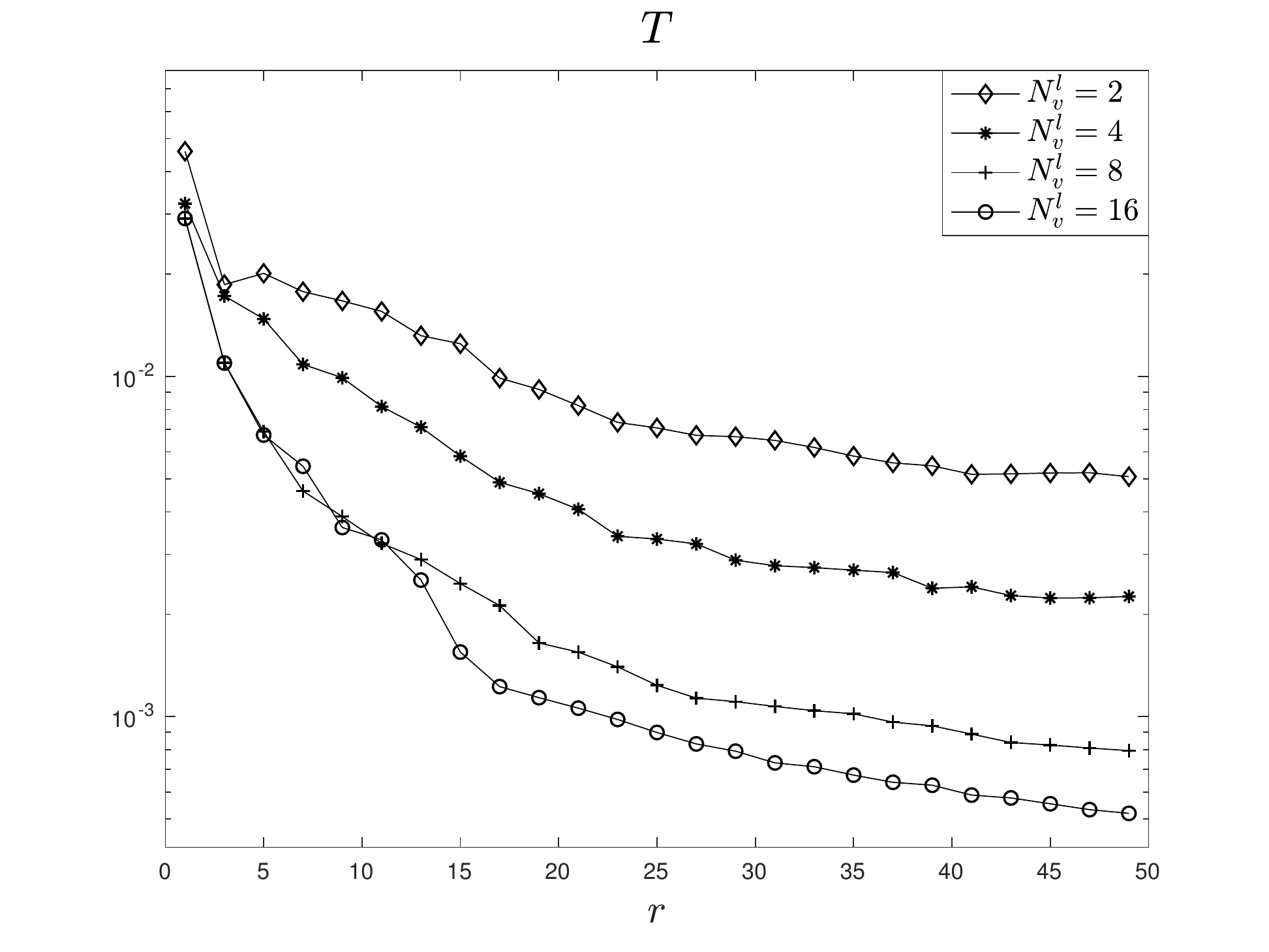}
\caption{}
\end{subfigure}
\caption{The mean $L^2$ error of the bi-fidelity approximation of $\rho$, $u_1$, $T$  with respect to the number of high-fidelity simulation runs, based on the low-fidelity model with different $N_v^l$. }
\label{Fig1}
\end{figure}

\begin{figure}[H]
\centering
\includegraphics[width=0.55\linewidth]{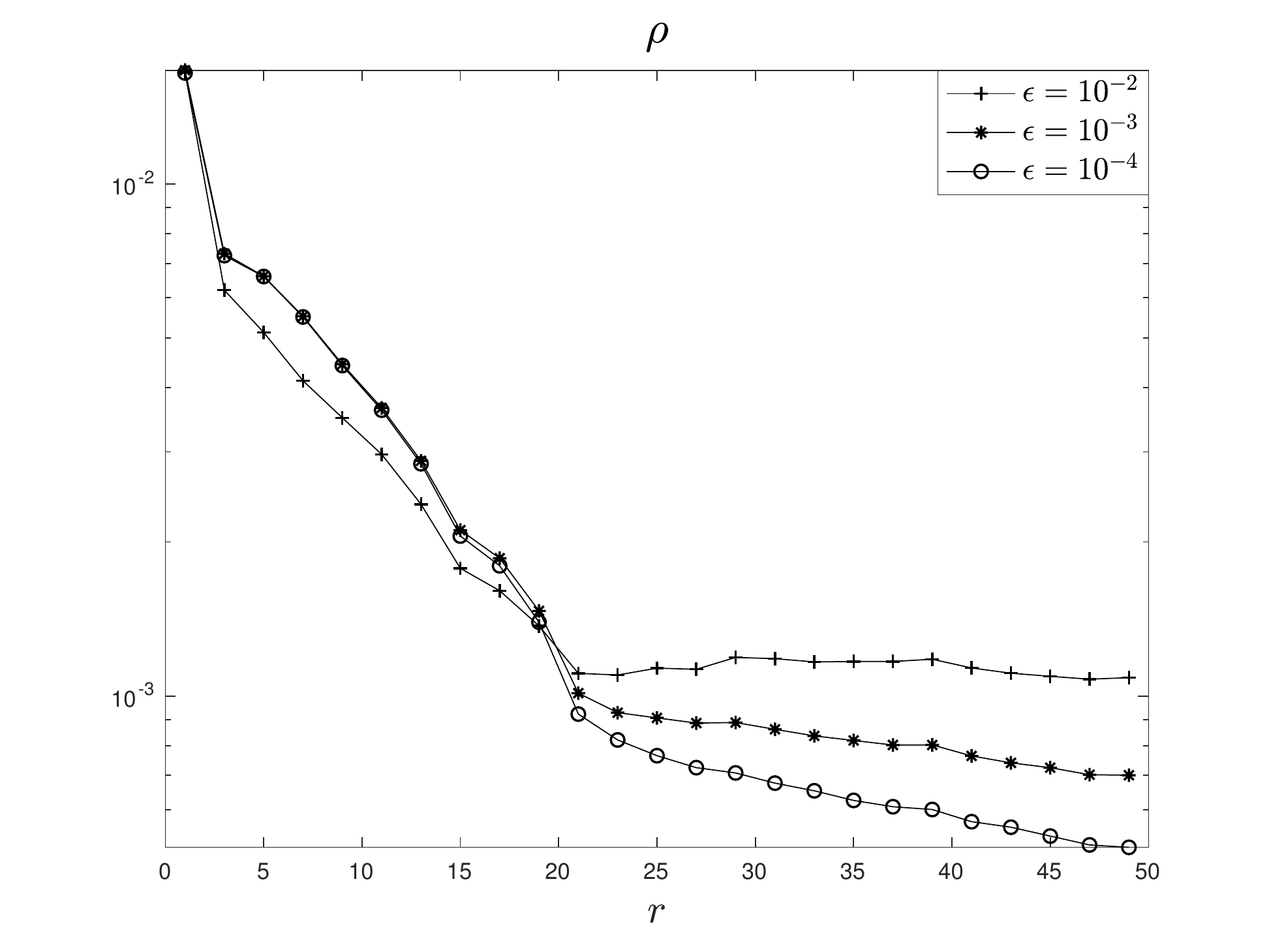}
\includegraphics[width=0.55\linewidth]{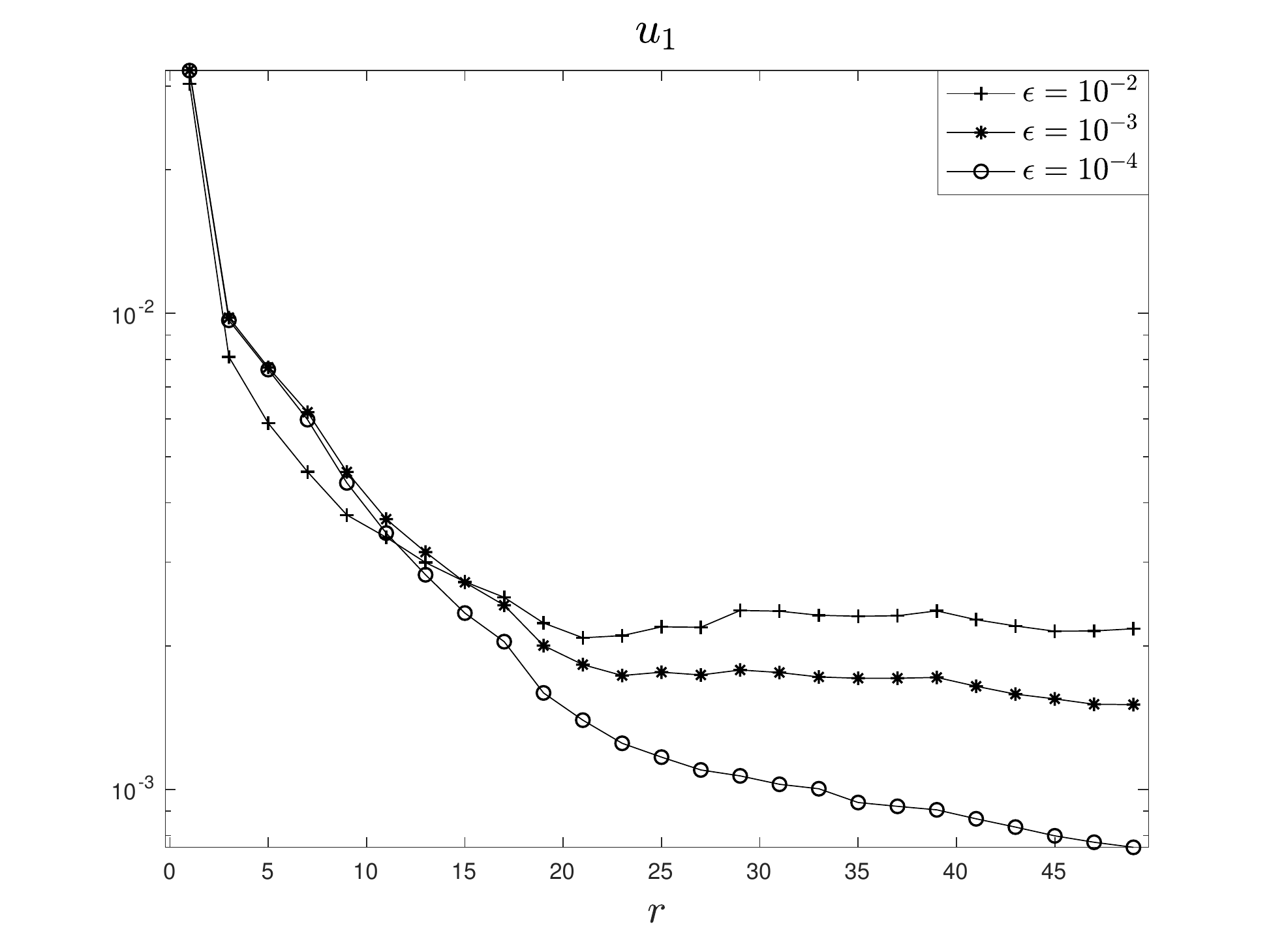}
\includegraphics[width=0.55\linewidth]{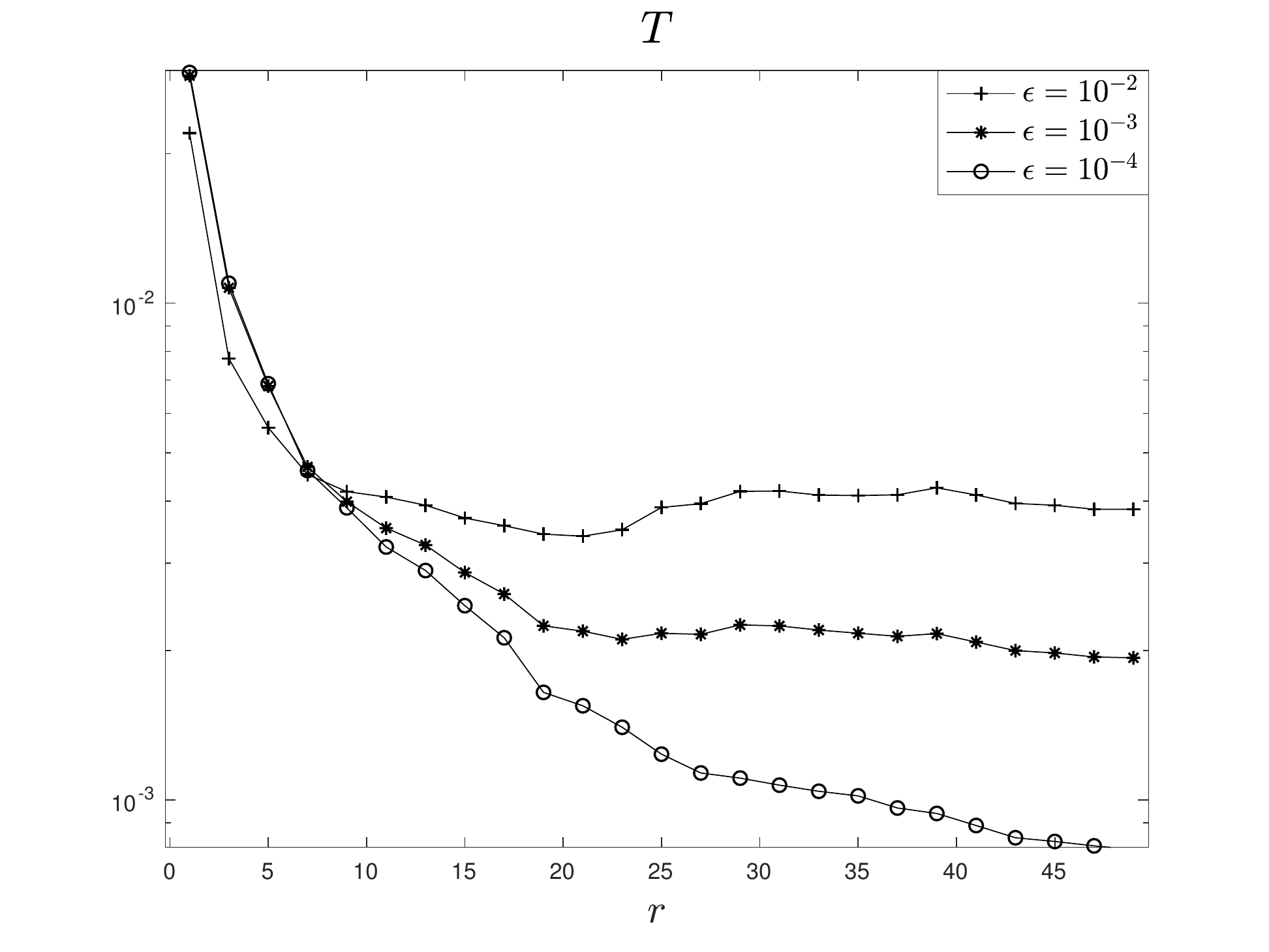}
\caption{The mean $L^2$ error of the bi-fidelity approximation of $\rho$, $u_1$, $T$  with respect to the number of high-fidelity simulation runs, based on the low-fidelity model with $N_v^l=8$ for different $\e$. 
}
\label{Fig3}
\end{figure}

\begin{figure}[H]
\centering
\includegraphics[width=0.45\linewidth]{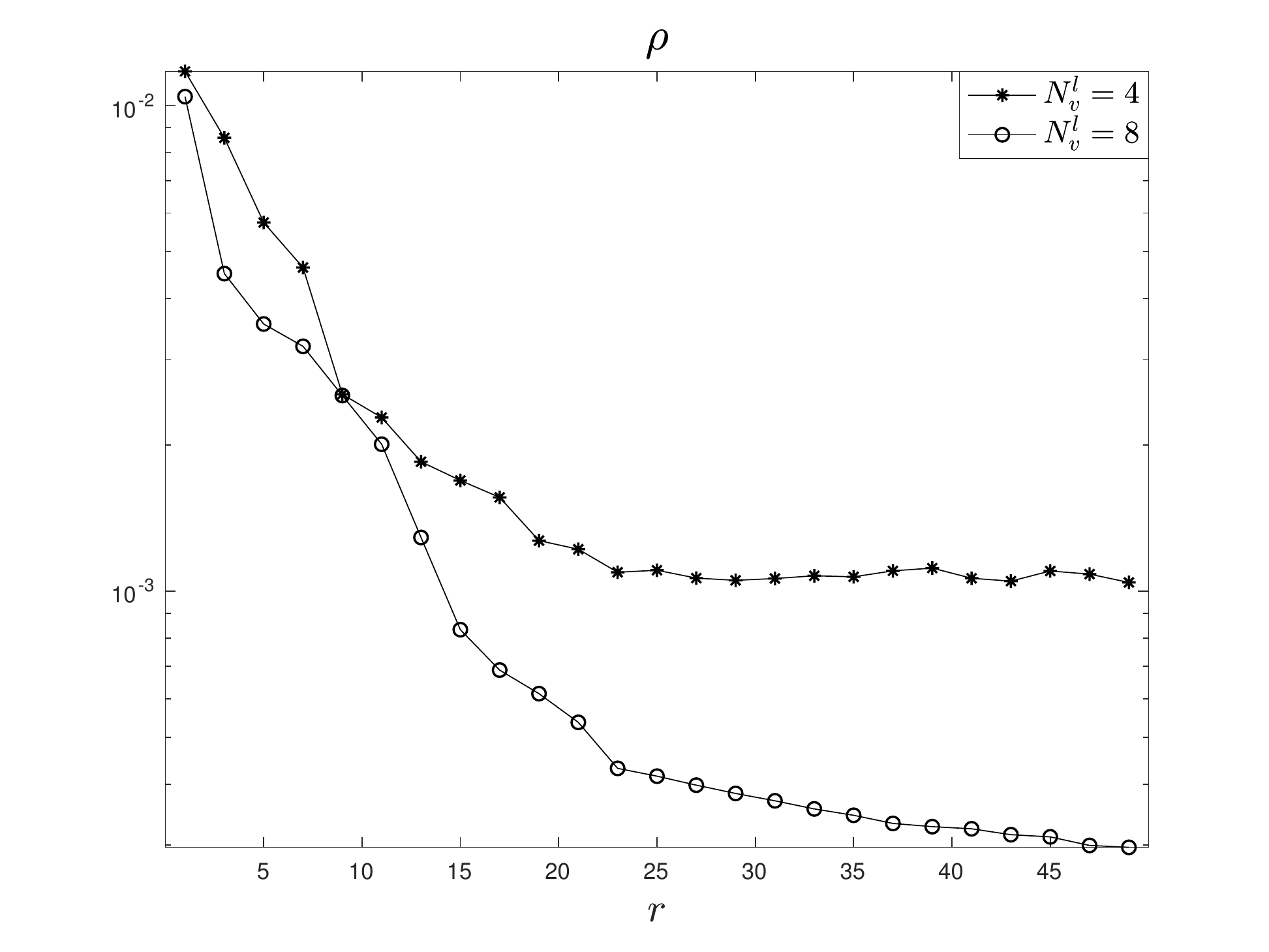}
\includegraphics[width=0.45\linewidth]{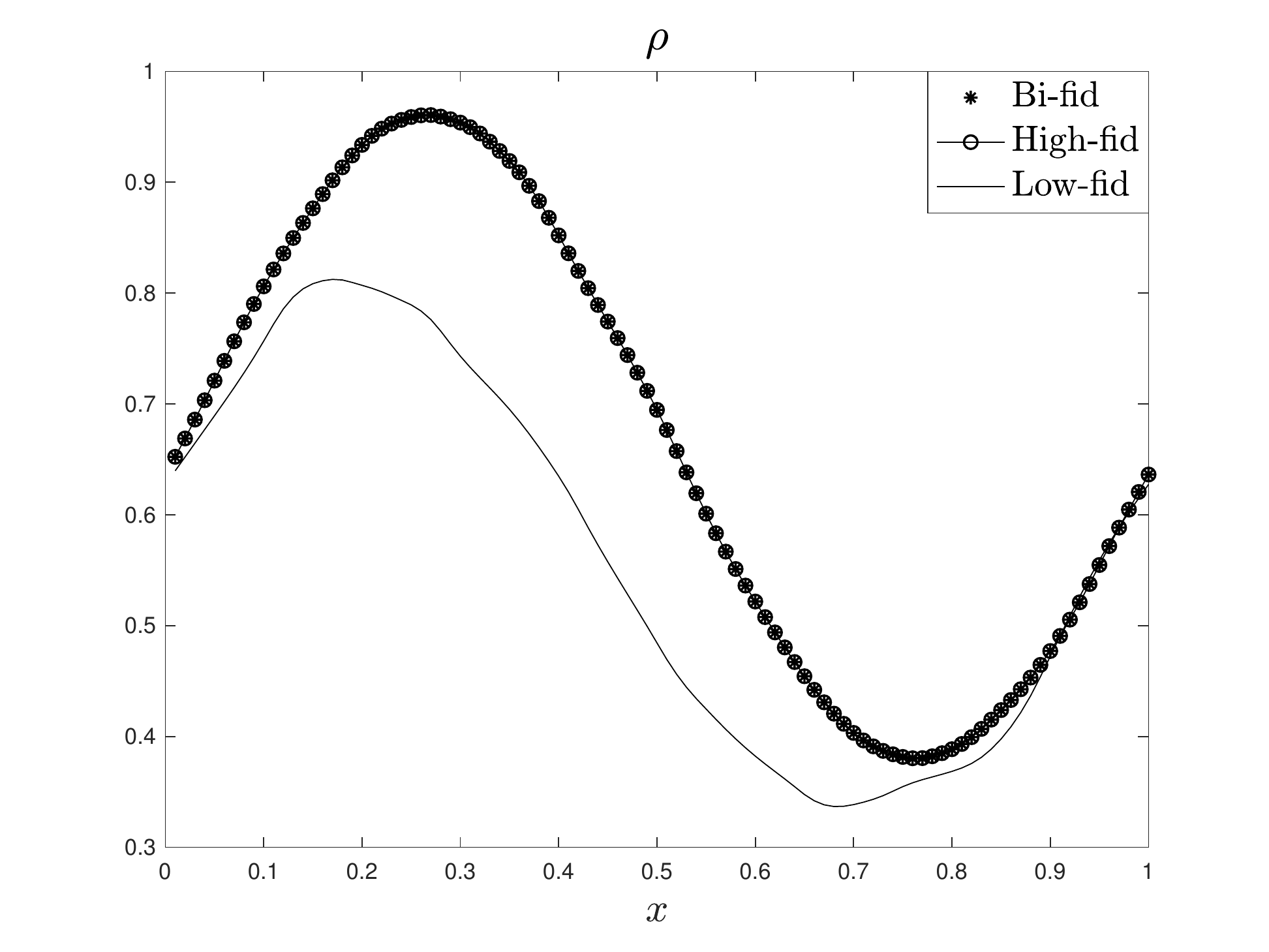}
\includegraphics[width=0.45\linewidth]{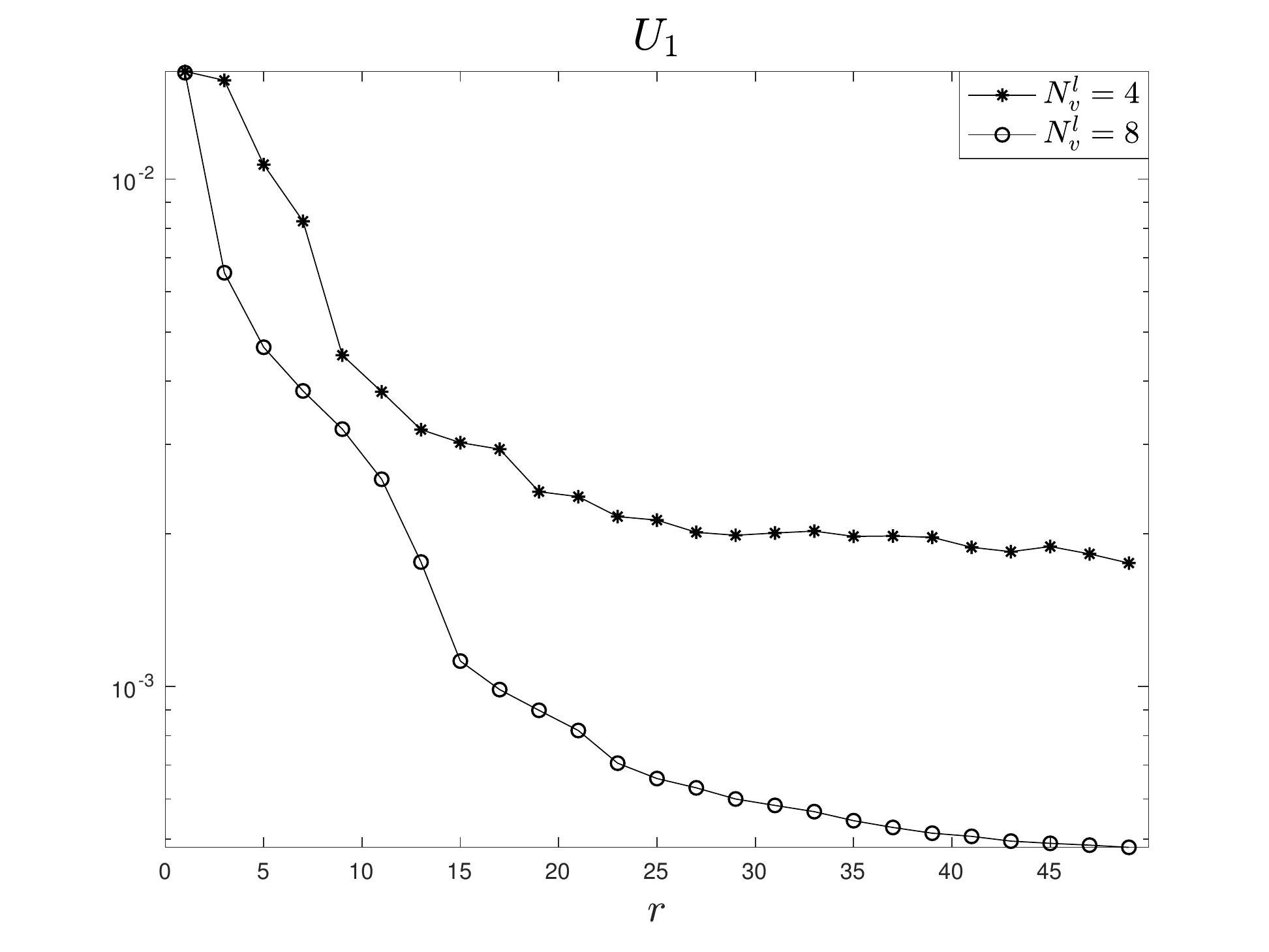}
\includegraphics[width=0.45\linewidth]{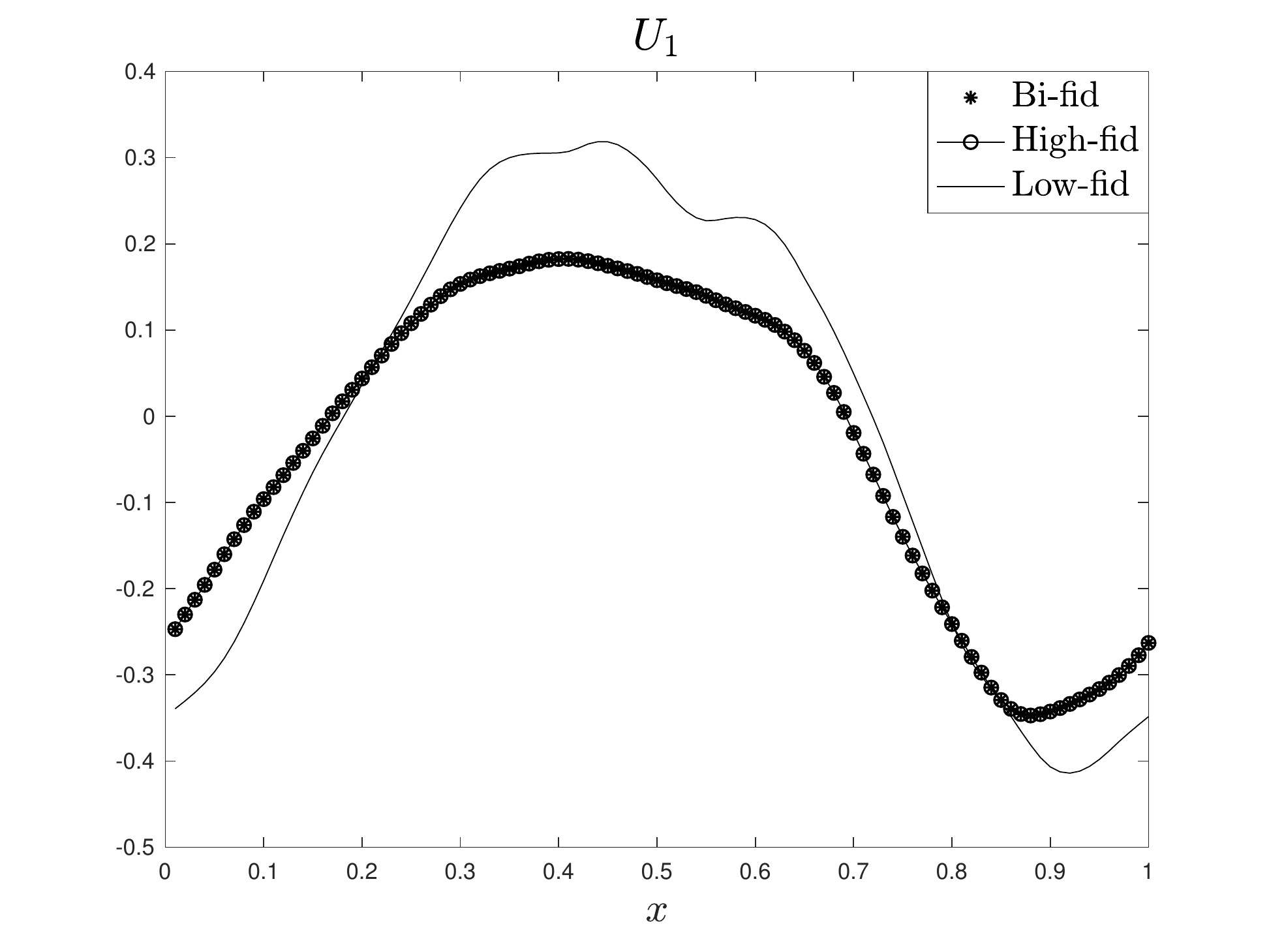}
\includegraphics[width=0.45\linewidth]{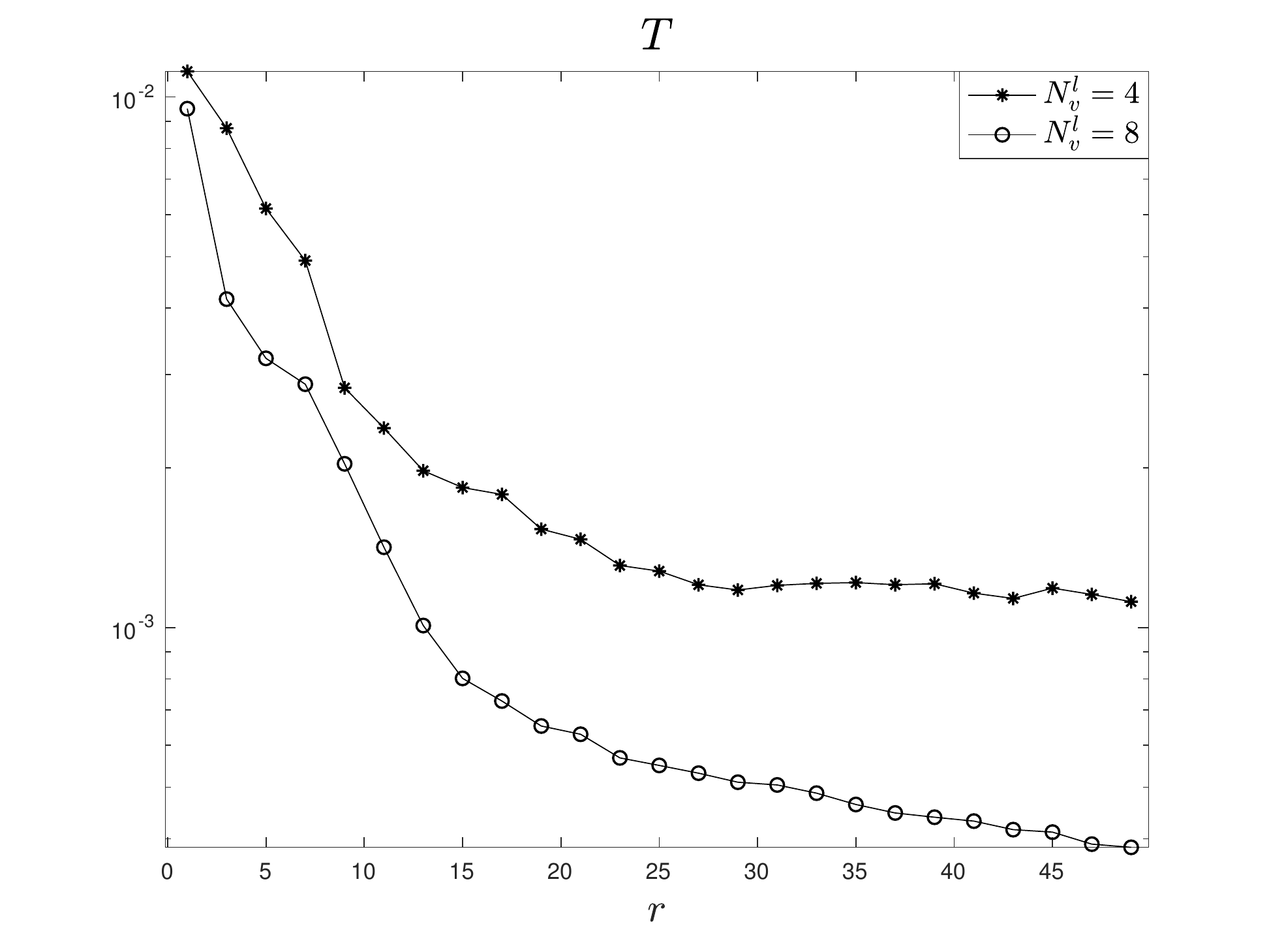}
\includegraphics[width=0.45\linewidth]{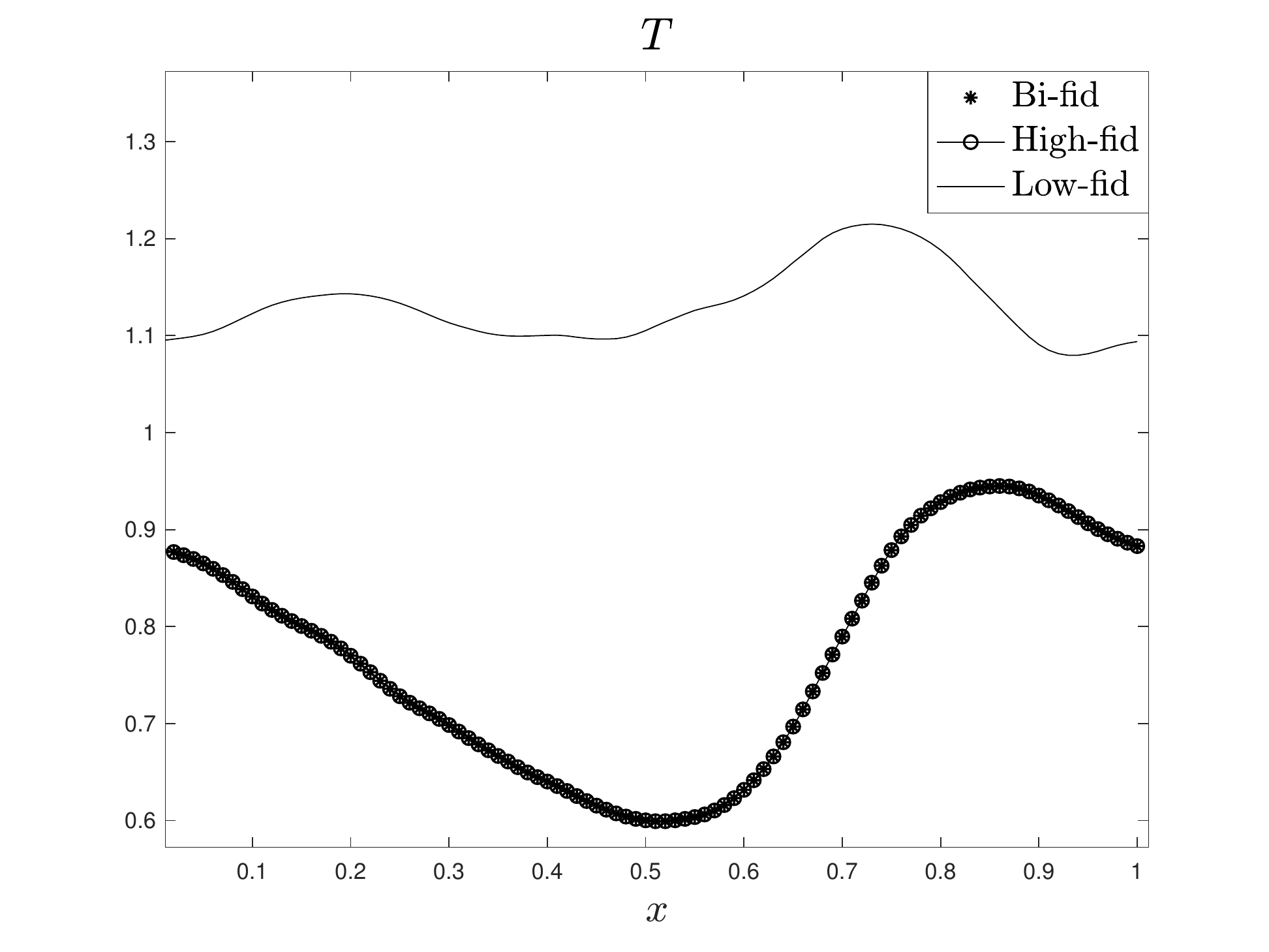}
\caption{(Left) The mean $L^2$ error of the bi-fidelity approximation of $\rho$, $u_1$, $T$  with respect to the number of high-fidelity simulation runs, based on  the low-fidelity model with different $N_v^l$. (Right) Comparison of the low-fidelity solution ($N_v^l=8$), high-fidelity solutions ($N_v^l=16$), and the corresponding bi-fidelity approximations with $r=20$ for a fixed $z$. }
\label{Fig2}
\end{figure}

%--------------------------------------------------------
\subsection{Sod shock tube test}
\label{Test2}
We next consider a more challenging problem where the initial data is discontinuous. Assume the random collision kernel in the form 
$$ b({\bf z}^b) = 1 + 0.5 \sum_{k=1}^{d_1+1}\frac{z_k^b}{2k}, $$
and the random initial distribution $$ f^0(x,{\bf v},{\bf z}) = \frac{\rho^0}{2\pi T^0}\, e^{-\frac{|{\bf v}-u^0|^2}{2T^0}}, $$
where the initial data for $\rho^0$, $u^0$ and $T^0$ is given by 
\begin{align*}
\begin{cases}
&\displaystyle\rho_l=1, \qquad u_l=(0,0),  \qquad T_l({\bf z}^T)=1+0.4\sum_{k=1}^{d_1} \frac{z_k^T}{2k}, \qquad x\leq 0.5, \\[4pt]
&\displaystyle\rho_r=\frac{1}{8}, \qquad u_r=(0,0),  \qquad T_r({\bf z}^T)=\frac{1}{8}(1 + 0.4\sum_{k=1}^{d_1}\frac{z_k^T}{2k}),  \qquad x>0.5. 
\end{cases}
\end{align*}
Here ${\bf z}^b=\left(z_1^b, \cdots, z_{d_1+1}^b\right)$ and ${\bf z}^T=\left(z_1^T, \cdots, z_{d_1}^T\right)$ represent the random variables
in the collision kernel and initial temperature. Set $d_1=7$, then the total dimension $d$ of the random space is $15$. We use the Boltzmann equation as the high-fidelity model, and solve it by $\Delta x=0.01$, $\Delta t = 8\times 10^{-4}$, and $N_v^h=24$, until the final time $t=0.15$. 
We shall employ the Euler equation as the low-fidelity model, and solve it with the same spacial and temporal resolution with the high-fidelity model but with $N_v^l=12$. We consider the fluid regime with $\e=10^{-4}$ in this test. 

From the left column of Figure \ref{Fig6-conv}, we see a fast convergence of $L^2$ errors between the high- and bi-fidelity solutions. With only 10 high-fidelity runs, the bi-fidelity approximation can reach an accuracy level  $\mathcal{O}(10^{-3})$ for a 15-dimensional  problem in random space, while the low-fidelity approximation is quite poor with an accuracy level  $\mathcal{O}(10^{-1})$.
To further illustrate the performance of our bi-fidelity method, we compared the high-, low- and the corresponding bi-fidelity solutions (with $r=10$) for a particular sample point $z$. One observes that the high- and bi-fidelity solutions match really well, whereas the low-fidelity solutions seem to be quite inaccurate at some points in the spatial domain. 
Even in this case, the bi-fidelity solutions can approximate the high-fidelity solutions very well.

Figure \ref{Fig6-mv} shows clearly that the mean and standard deviation of the bi-fidelity approximation of $\rho$, $u_1$ and $T$ agree well with the high-fidelity solutions by using only $10$ high-fidelity runs. 
The result is a bit surprising yet reasonable, suggesting that even though the Euler model may be inaccurate in the physical space, it still can capture the behaviors and characteristics of the solution to the Boltzmann equation
in the random space. Moreover, since the high-fidelity model (Boltzmann) with $N_v^h=24$ costs approximately $43$ times of the low-fidelity solver (Euler) with $N_v^l=12$ (the former takes $30.6$ seconds, the latter takes 
$0.7$ seconds for one single run;  a significant speedup is quite noticeable in this case. 

\begin{figure}[H]
\includegraphics[width=0.48\linewidth]{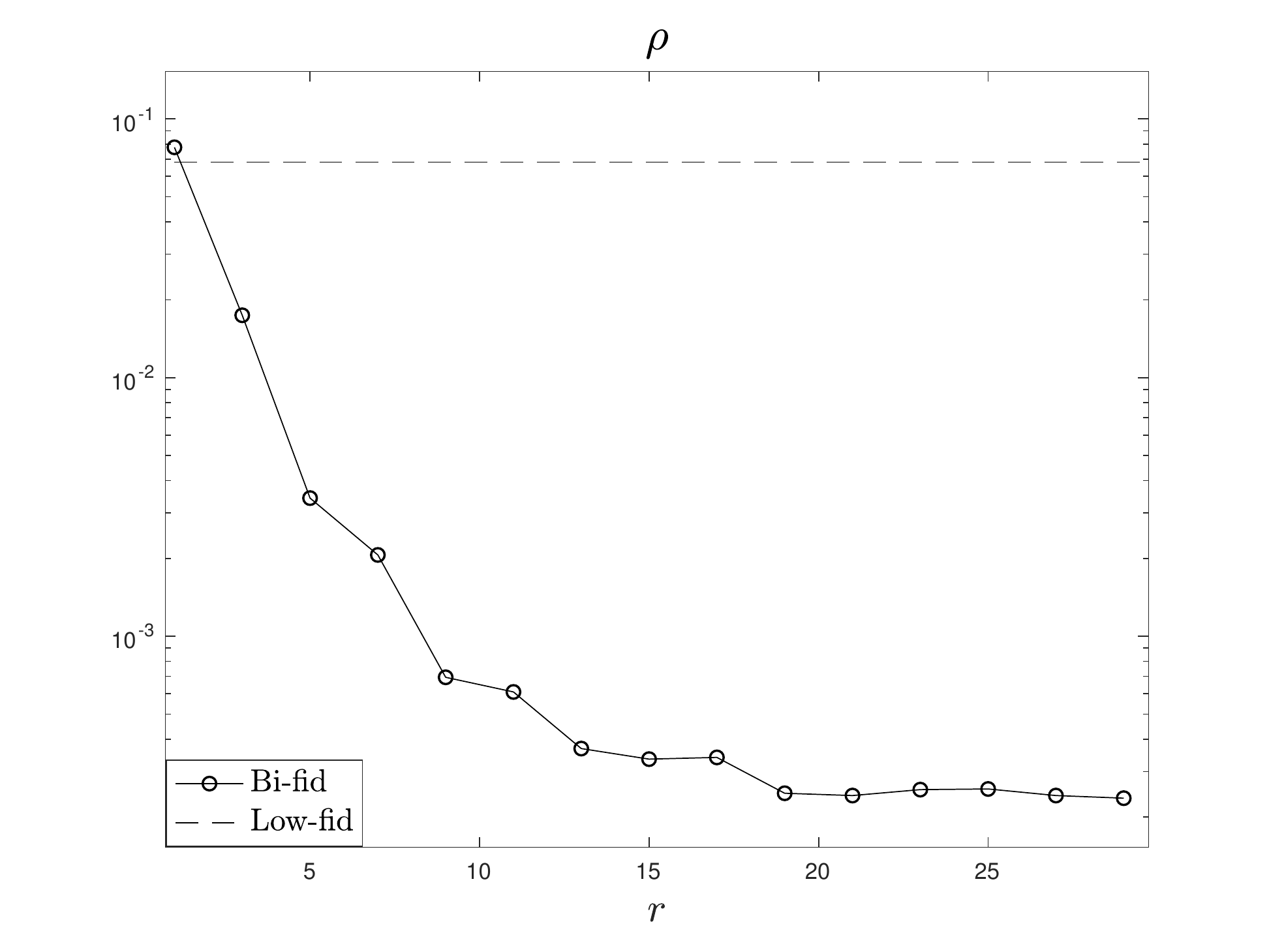}
\includegraphics[width=0.48\linewidth]{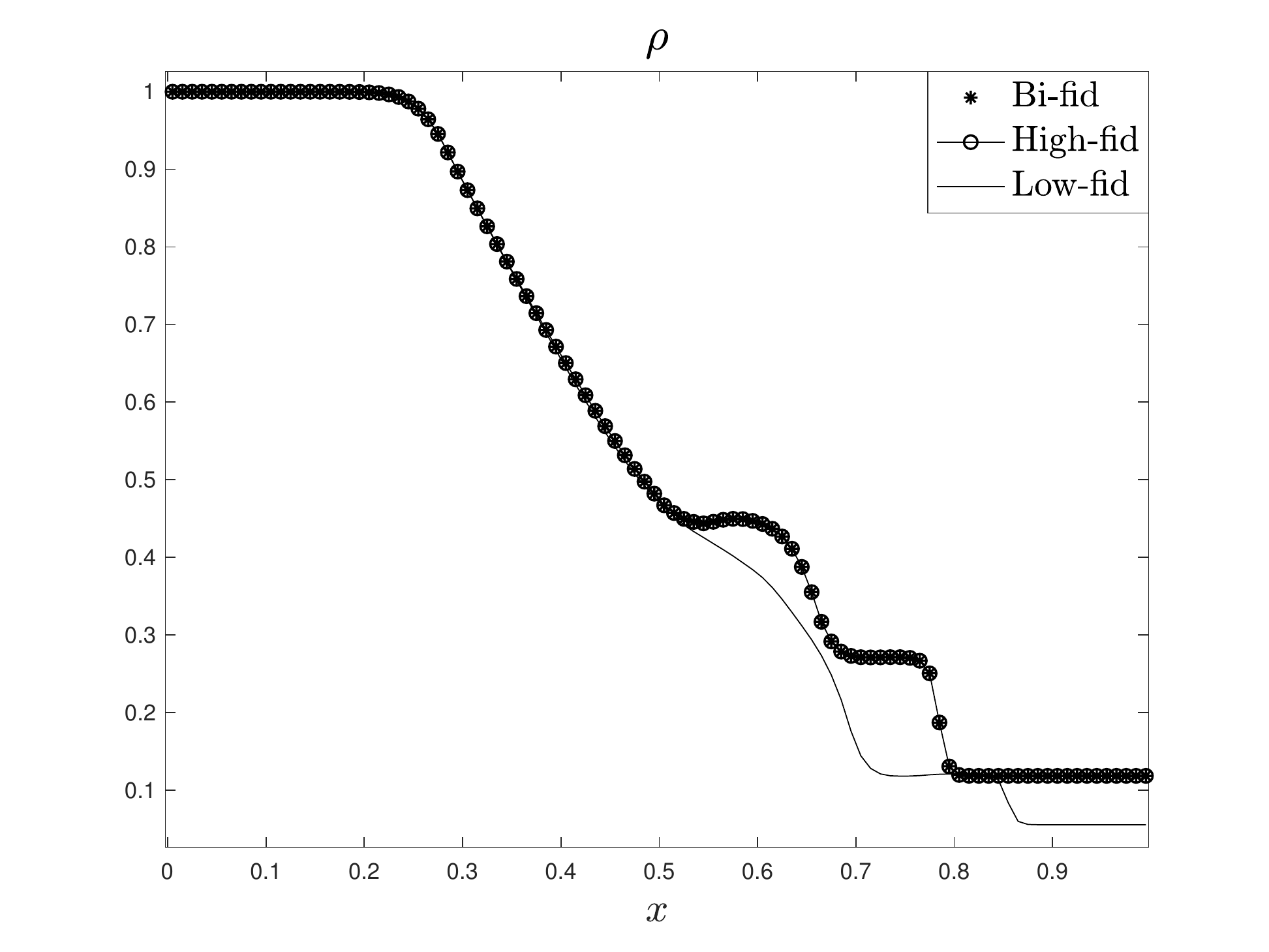}
\includegraphics[width=0.48\linewidth]{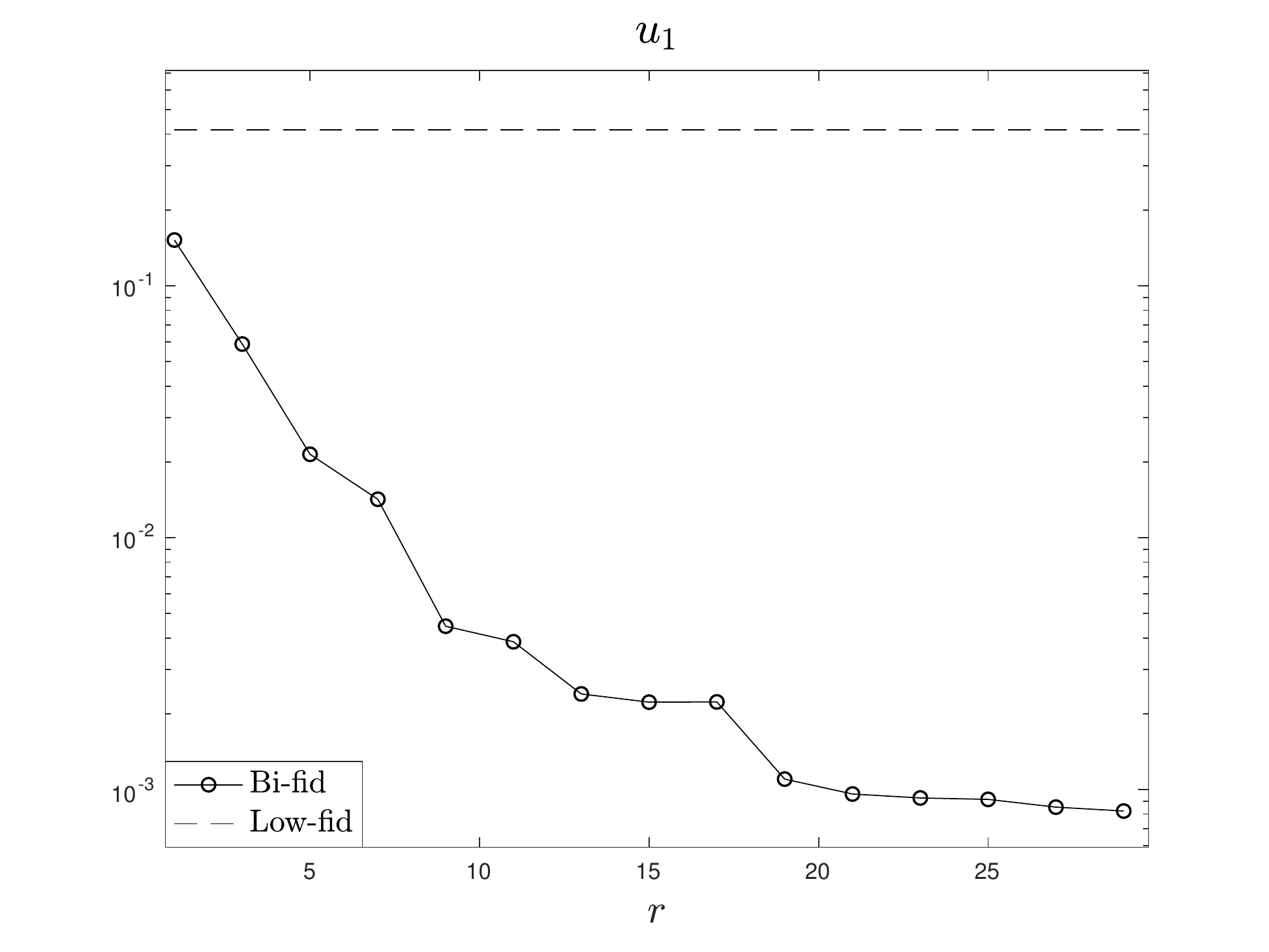}
\includegraphics[width=0.48\linewidth]{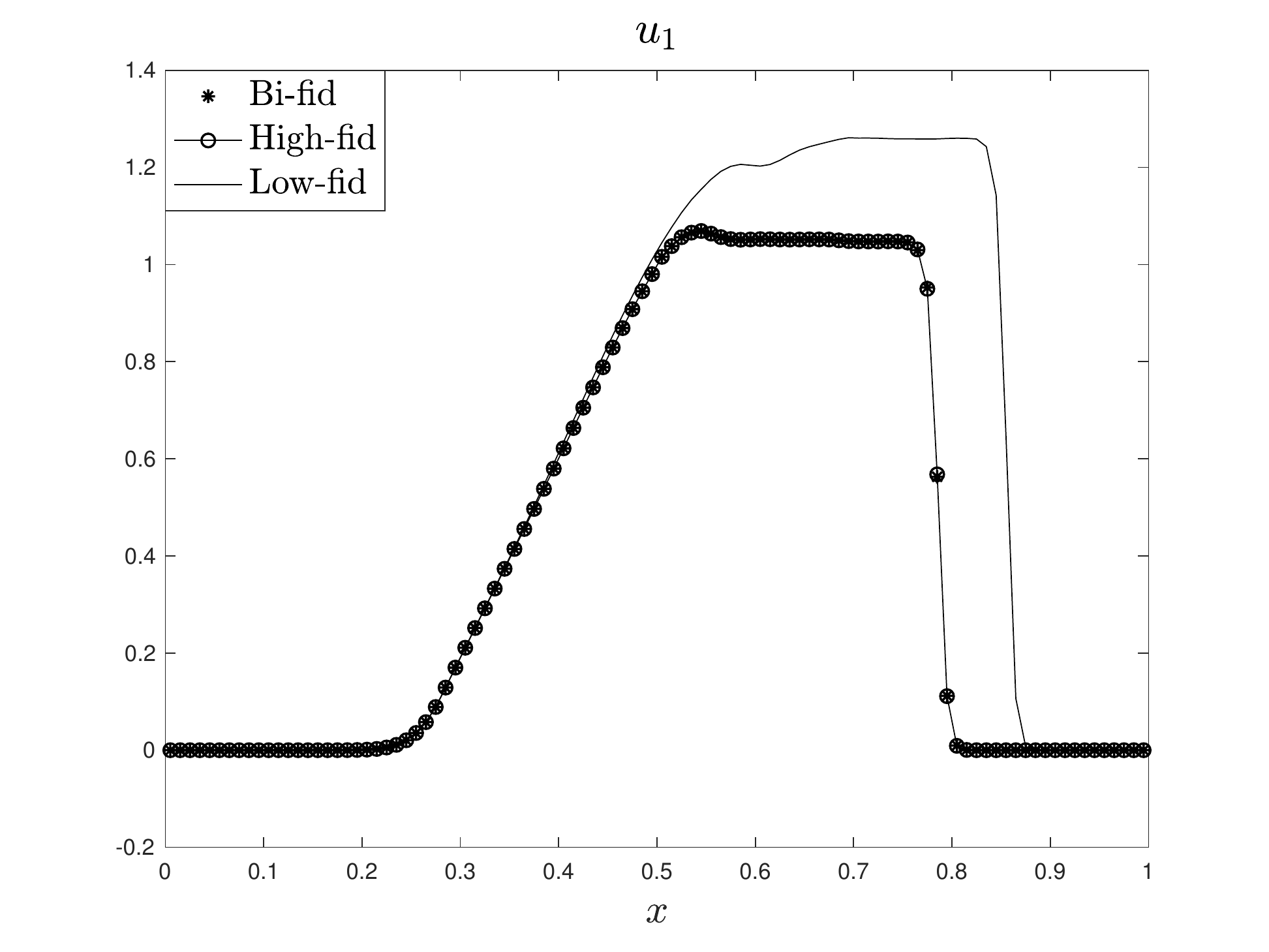}
\includegraphics[width=0.48\linewidth]{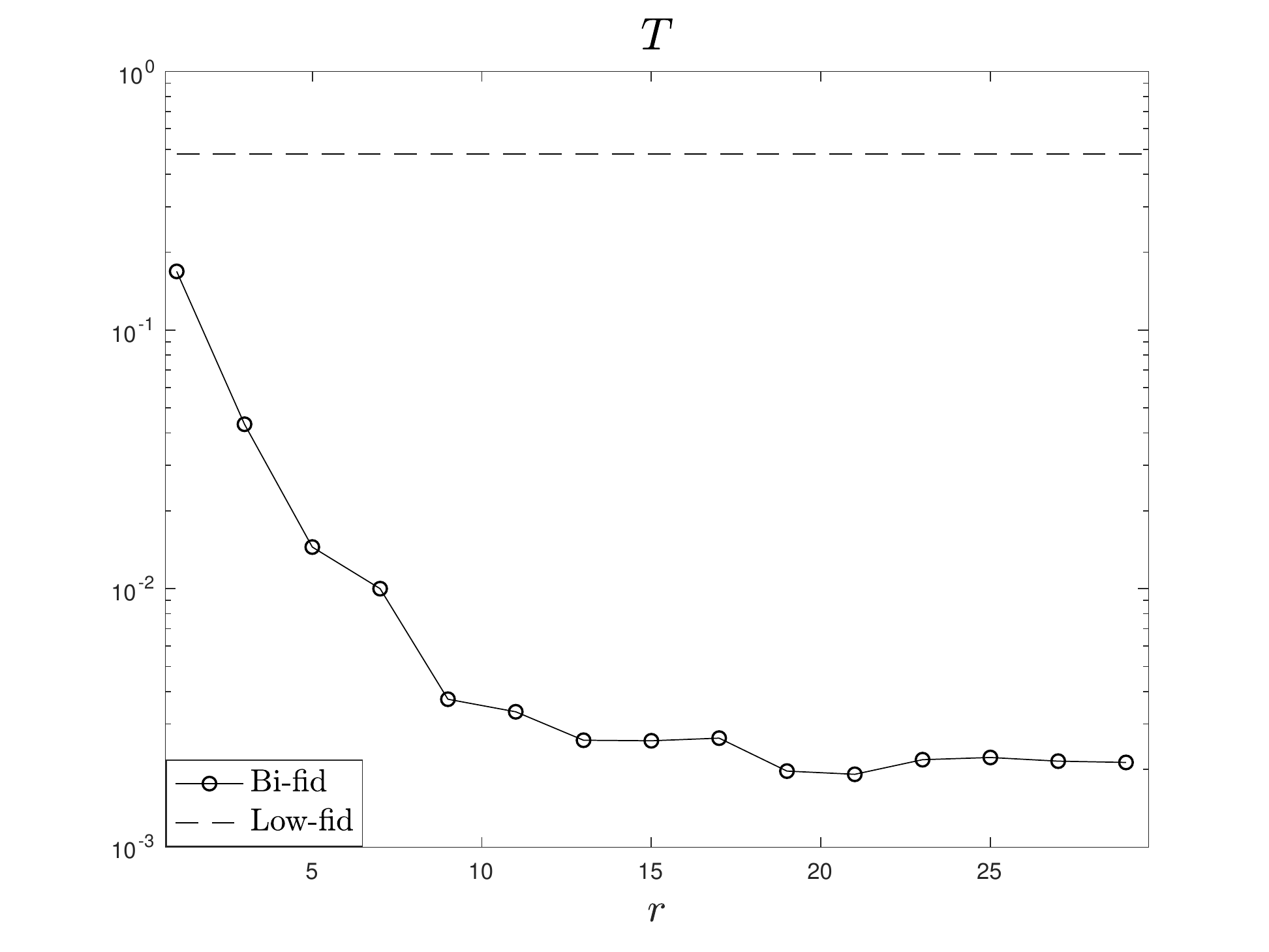}
\includegraphics[width=0.48\linewidth]{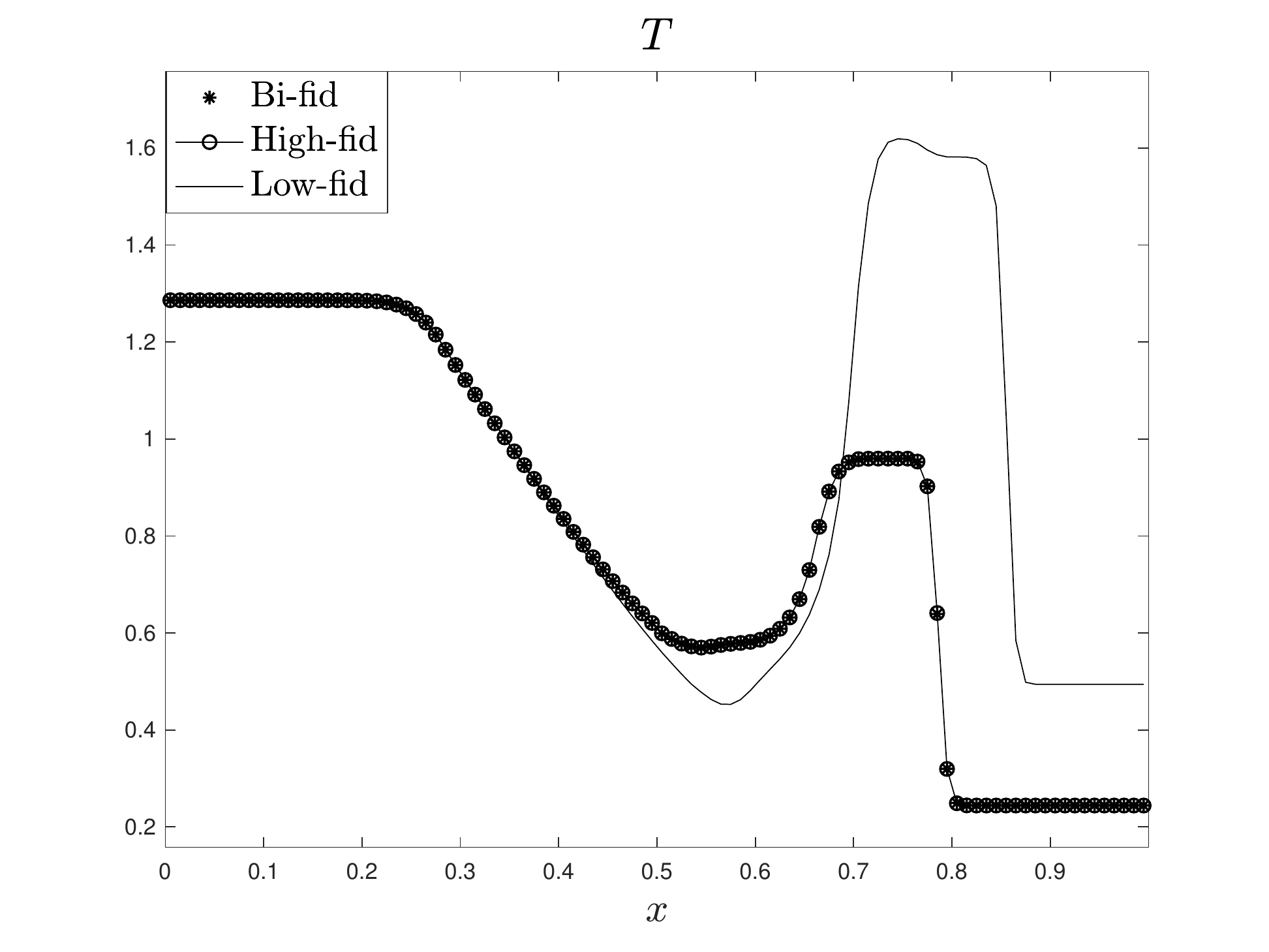}
\caption{(Left) The mean $L^2$ errors between high-fidelity and 
low- or bi-fidelity solutions with respect to the number of high-fidelity runs; (Right) Comparison of the low-fidelity solution ($N_v^l=12$), high-fidelity solutions ($N_v^l=24$), and the corresponding bi-fidelity approximations $r=10$ for a fixed $z$.  }
\label{Fig6-conv}
\end{figure}

\begin{figure}[H]
\centering
\includegraphics[width=0.49\linewidth]{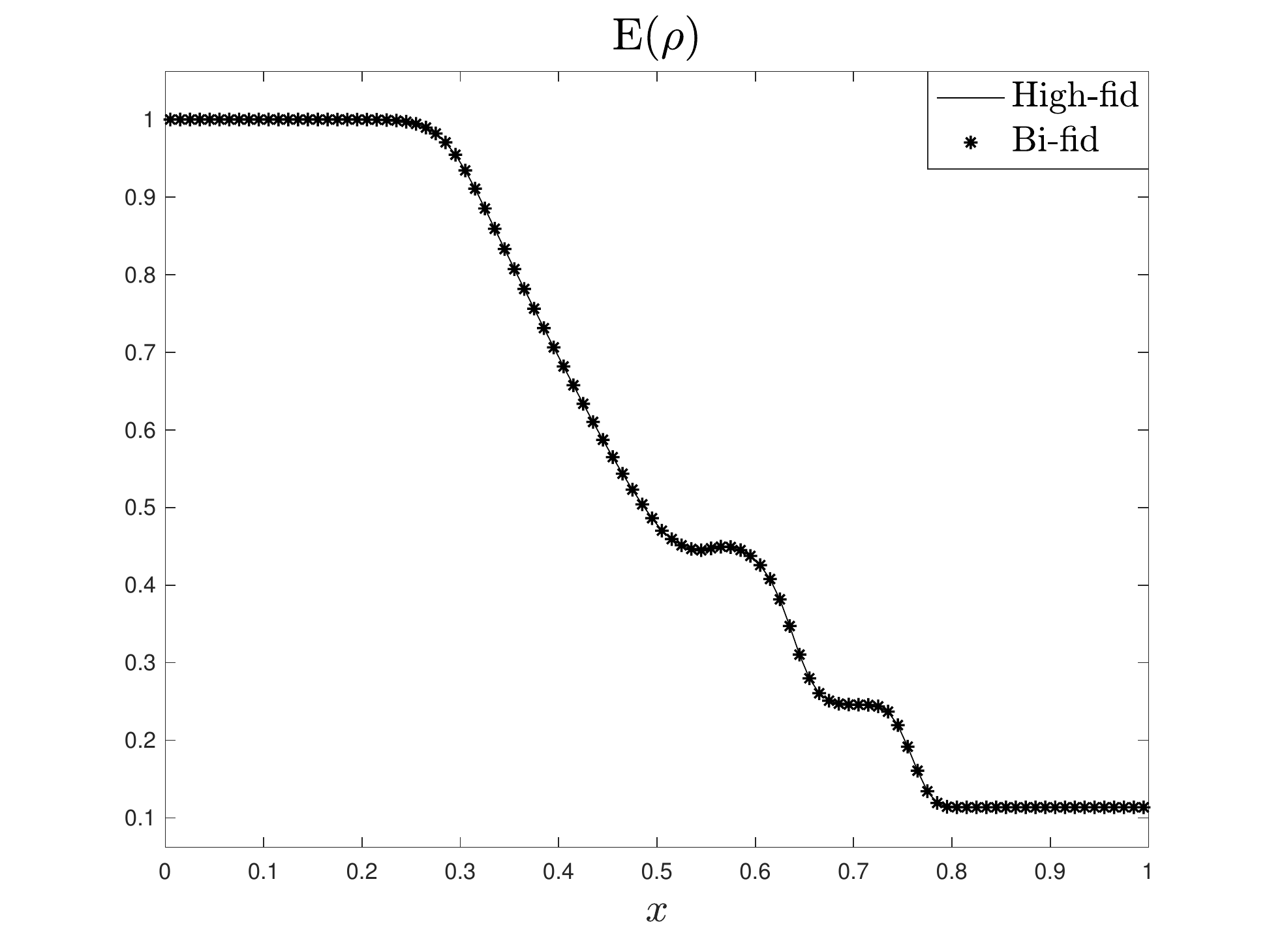}  
\includegraphics[width=0.49\linewidth]{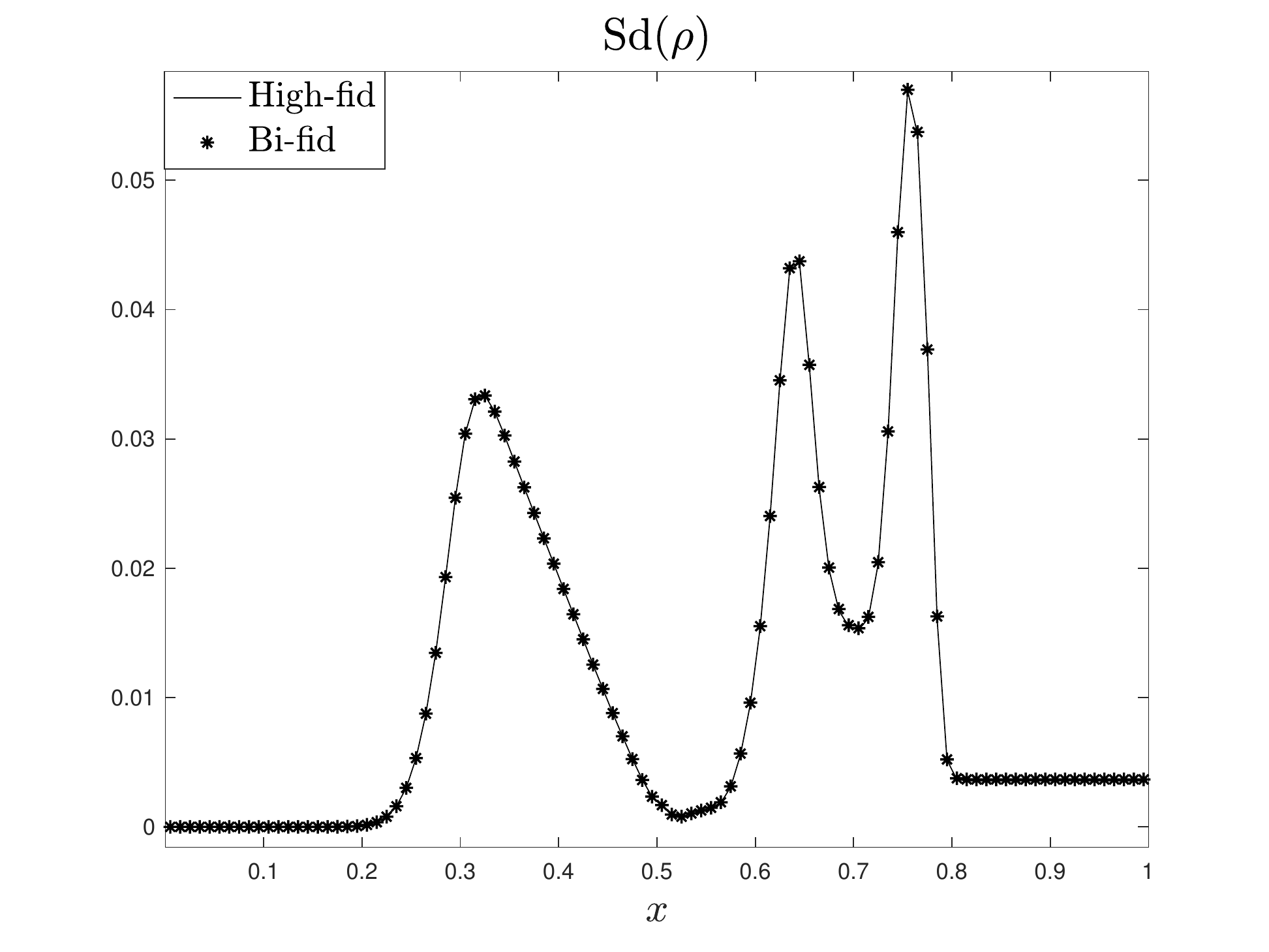}
\includegraphics[width=0.49\linewidth]{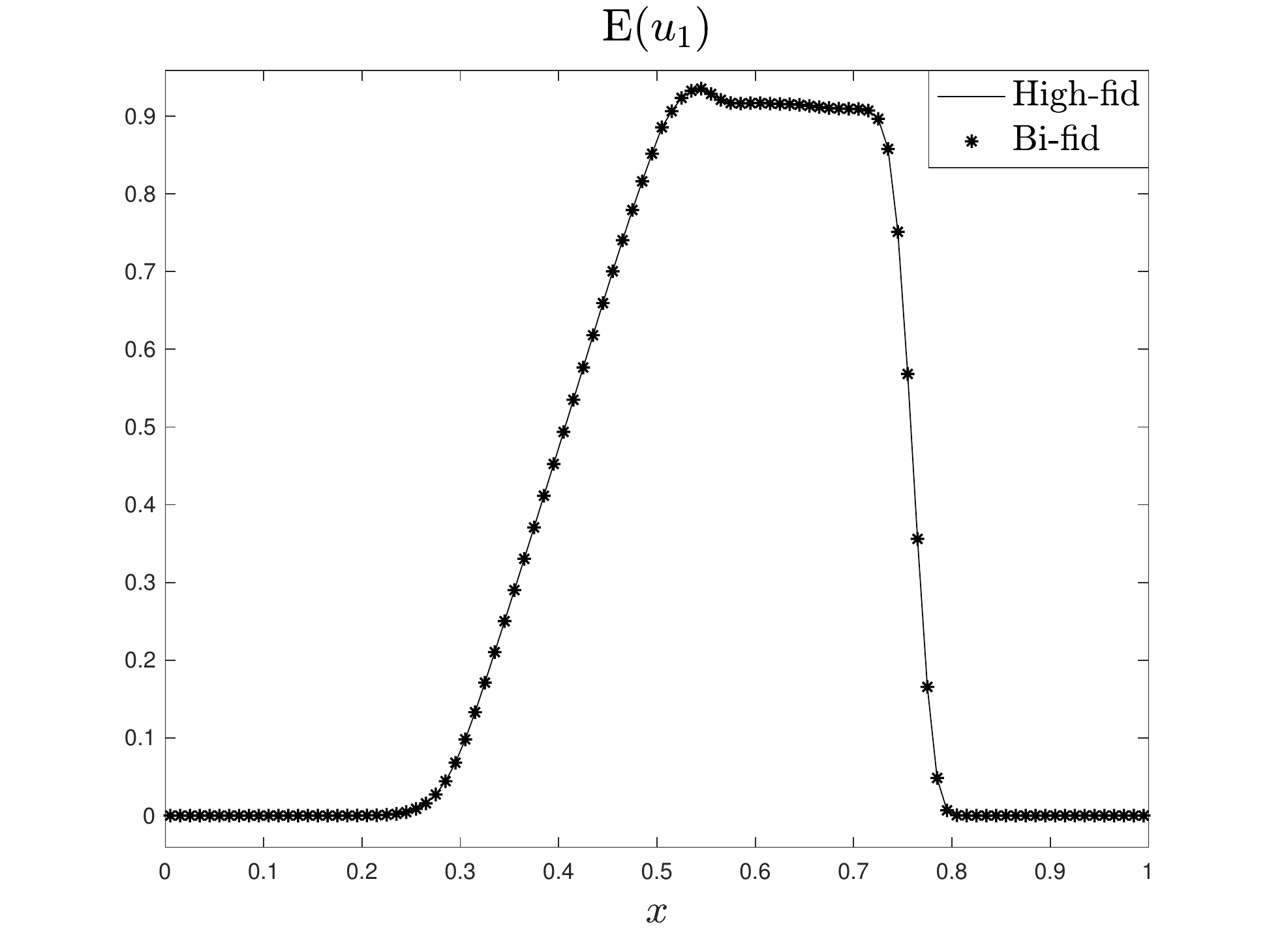}
\includegraphics[width=0.49\linewidth]{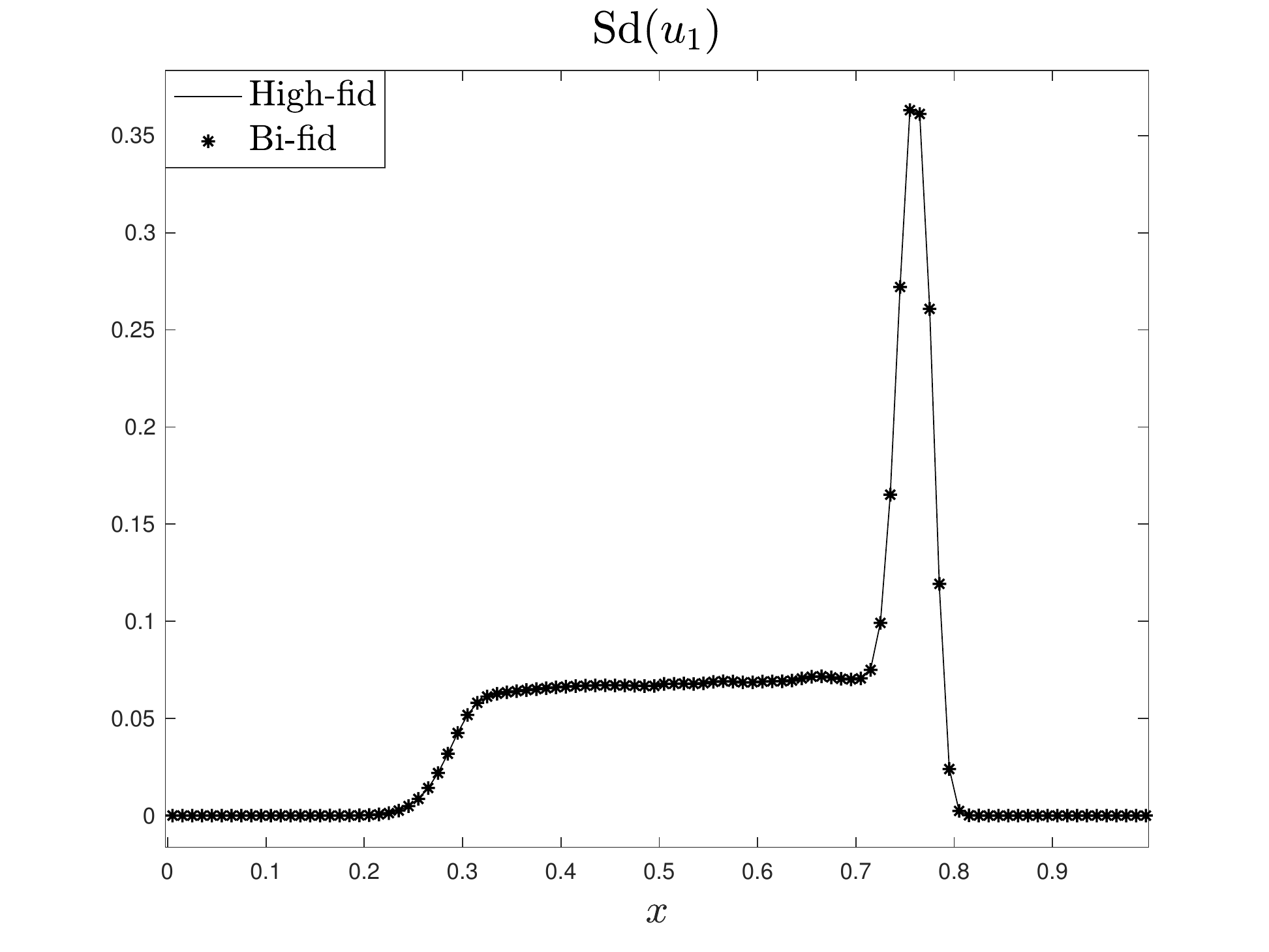}
\includegraphics[width=0.49\linewidth]{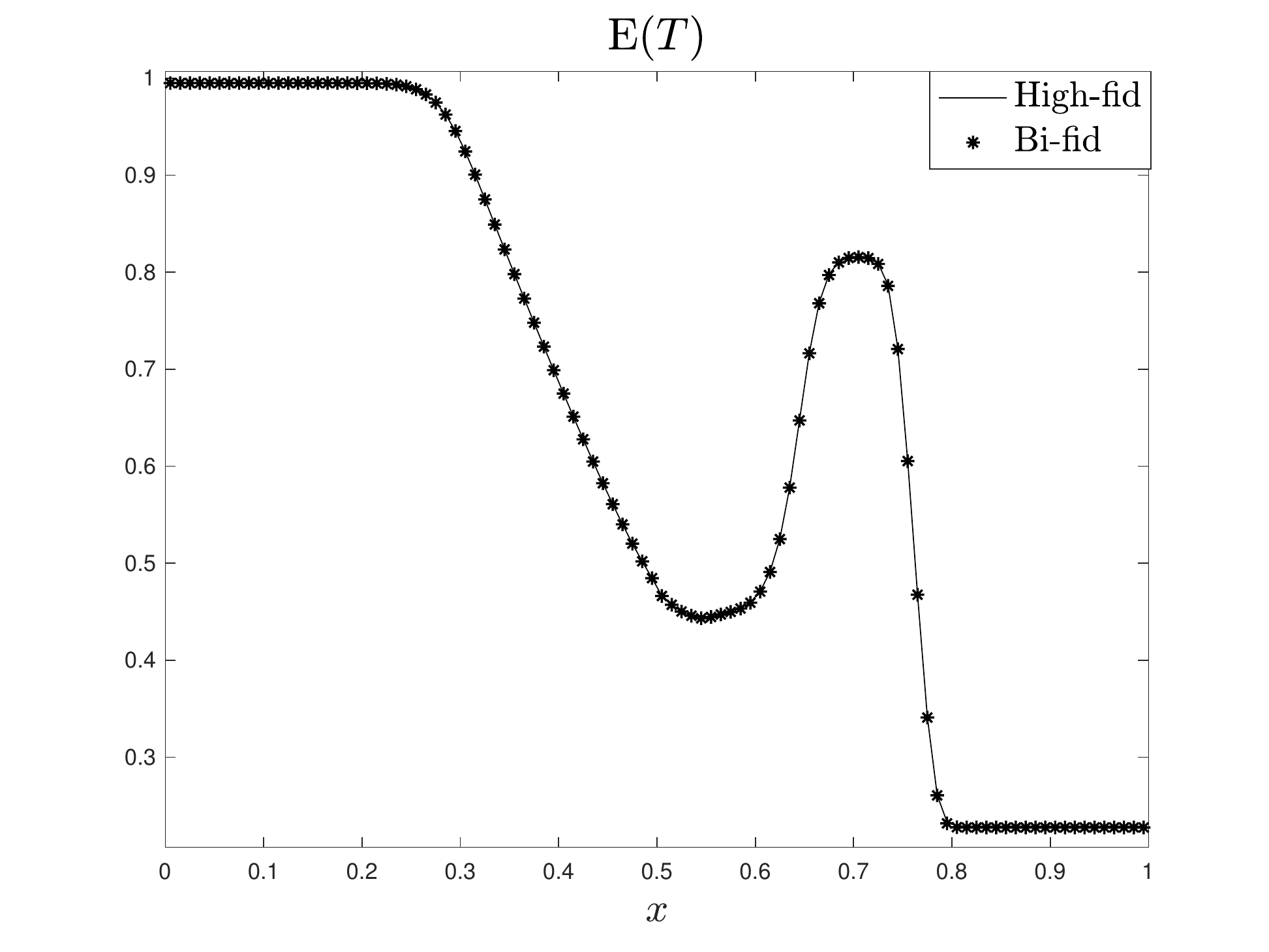}
\includegraphics[width=0.49\linewidth]{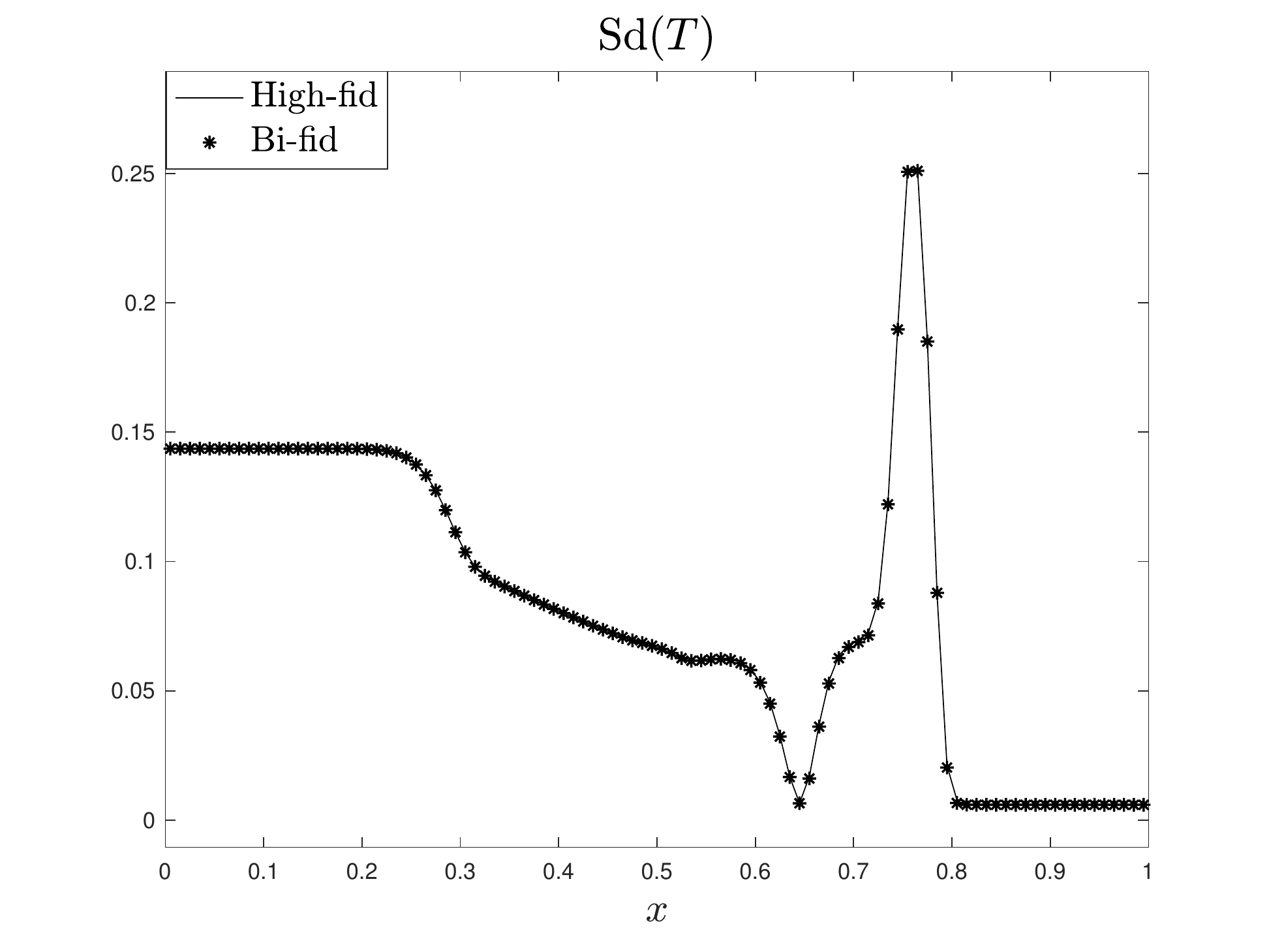}
\caption{Mean and standard deviation of $\rho$, $u_1$, $T$ of high-fidelity solutions and bi-fidelity solutions with $r=10$. }
\label{Fig6-mv}
\end{figure}

%------------------------------------------------
\subsection{Mixed regime test}
\label{Test3}
The next test we shall consider is more challenging than the previous two tests. Because various scales are involved, good accuracy of the AP scheme for the Boltzmann equation is required for all ranges of $\e$. 
We consider a mixed regime with the Knudsen number 
$\varepsilon$ varying in space show in Figure.~\ref{figmixe} and given by 
\begin{equation}
\label{mixe}
 \e(x) = 10^{-3} + \frac{1}{2}\left[\tanh \left(1-\frac{11}{2}(x-0.5) \right) + \tanh \left(1+\frac{11}{2}(x-0.5) \right)\right].
\end{equation}
The random initial data and collision kernel are given by (\ref{IC}) and (\ref{R-K}). The total dimension 
of the random space is $d=15$. 

\begin{figure}[H]
\centering
\includegraphics[width=0.5\linewidth]{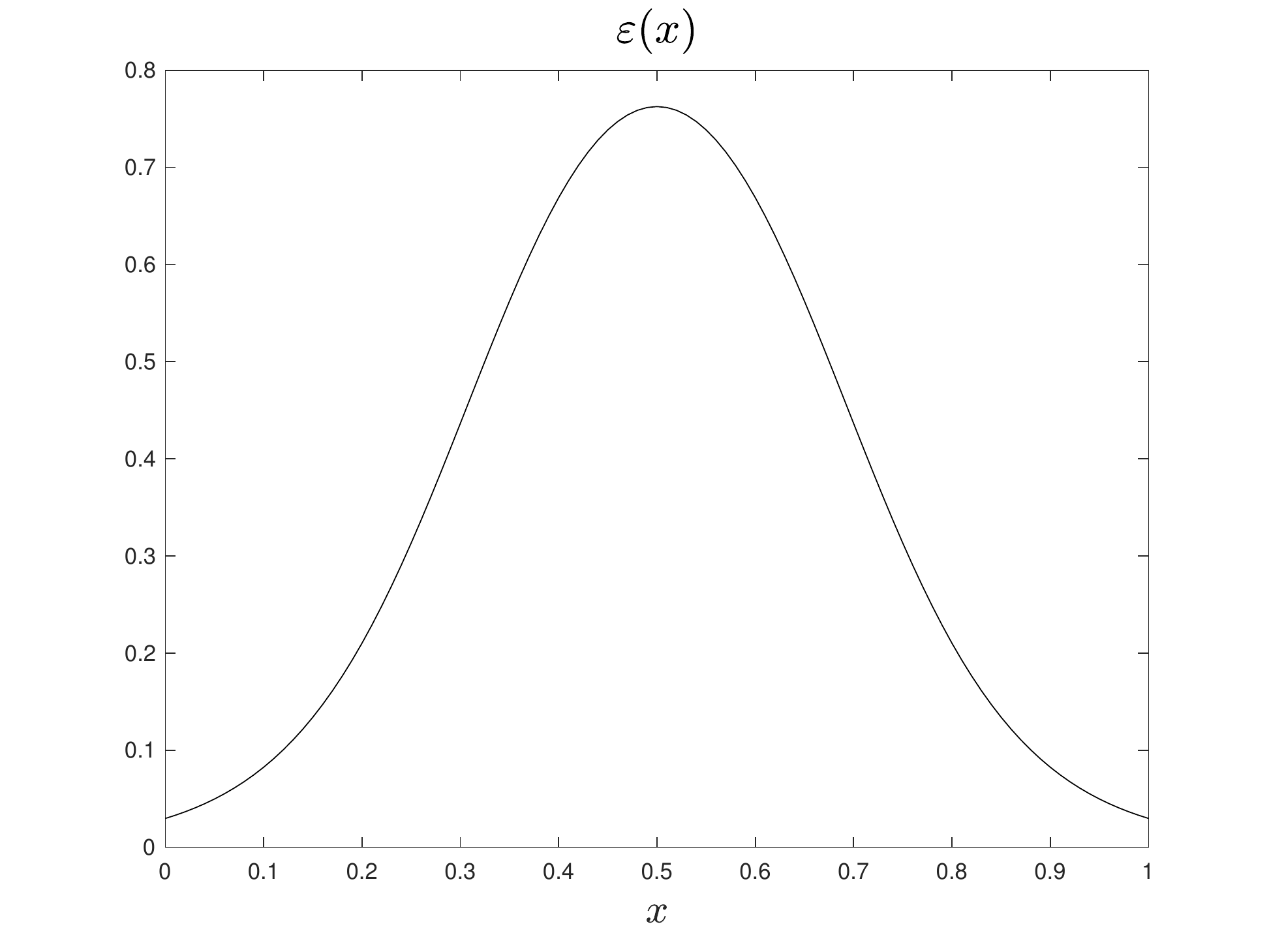}
\caption{The distribution of $\e(x)$ in \eqref{mixe}.}
\label{figmixe}
\end{figure}

All the numerical parameters used in temporal and spatial discretizations are the same as that in Section \ref{Test1}. 
We solve the Boltzmann equation (\ref{Boltz}) for the high-fidelity solution with $N_v^h=24$, and the Euler system for the low-fidelity solution with 
$N_v^l=8$. 

From the left column of Figure \ref{Fig5-conv}, we observe a fast convergence of $L^2$ errors between the high- and bi-fidelity solutions, where 
they saturate quickly when $r$ reaches about $25$. It is worth noting that the dotted lines that represent the errors between the high- and low-fidelity solutions      
are much larger $\mathcal{O}(10^{-1})$ compared to that between the high- and bi-fidelity solutions. This indicates that even though the low-fidelity solutions alone are relatively not accurate in the spatial domain, it might be still  able to behave similarly in the random space, therefore the resulted bi-fidelity approximation based on a small number of high-fidelity runs (say $r=25$) can reach a reasonable accuracy level up to $\mathcal{O}(10^{-3})$. 

The right column of Figure \ref{Fig5-conv} shows the high-, low- and bi-fidelity solutions at a randomly chosen sample point $z$. 
One can see that the high- and bi-fidelity solutions match really well, whereas the low-fidelity solutions are not accurate. In addition, with $N=1000$ low-fidelity runs of the Euler model, together with only $25$ runs of the AP solver to the Boltzmann model, one can get the bi-fidelity solutions which are able to capture behavior of the solutions to the Boltzmann equation in the random space, up to an accuracy of $10^{-3}$; on the other hand, using the low-fidelity model (Euler equation) alone can not achieve this result, especially under the multiple scalings where $\e$ ranges from $10^{-3}$ to 1 (since the errors between macroscopic quantities calculated from the Boltzmann and Euler equation deteriorate when $\e$ becomes large). This observation certainly highlights the merits of our bi-fidelity method.
Figure \ref{Fig5-mv} presents the mean and standard deviation of $\rho$, $u_1$, $T$ by using $25$ high-fidelity runs. One can see that the high- and bi-fidelity solutions match well, indicating that the bi-fidelity solutions have captured well the characteristics of the macroscopic quantities in the random space.

Once we construct the bi-fidelity model, the online computational cost can be significantly reduced. In this example, the high-fidelity model (Boltzmann) with $N_v^h=16$ takes  about $50$ times computation time of the low-fidelity model (the former takes $12.3$ seconds, while the latter takes $0.23$ seconds for a single run). Since the dominant cost of the bi-fidelity reconstruction for a high-fidelity solution lies in the corresponding low-fidelity run, a significant amount of computational cost is saved in our method. 

\begin{figure}[H]
\includegraphics[width=0.48\linewidth]{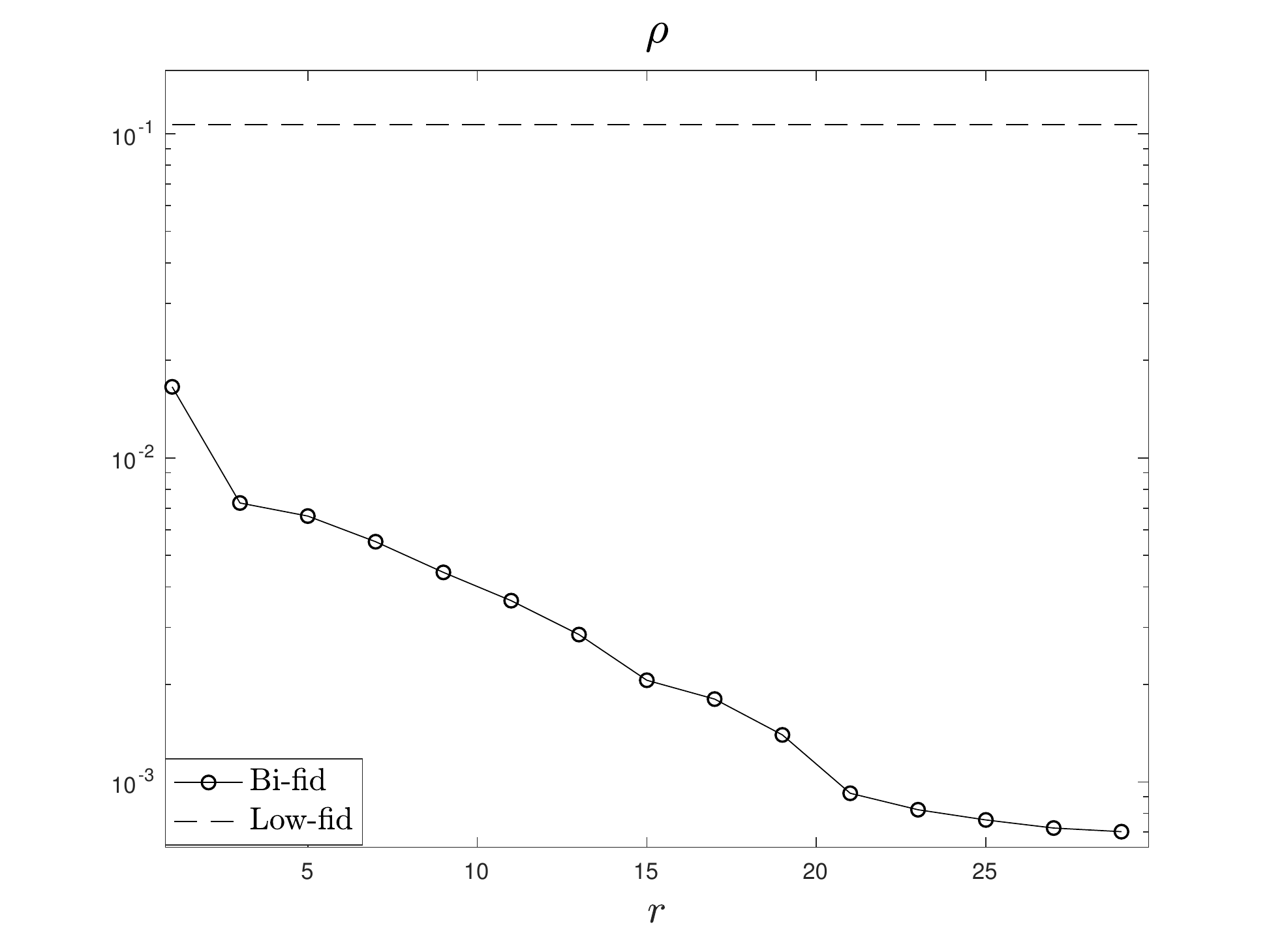}
\includegraphics[width=0.48\linewidth]{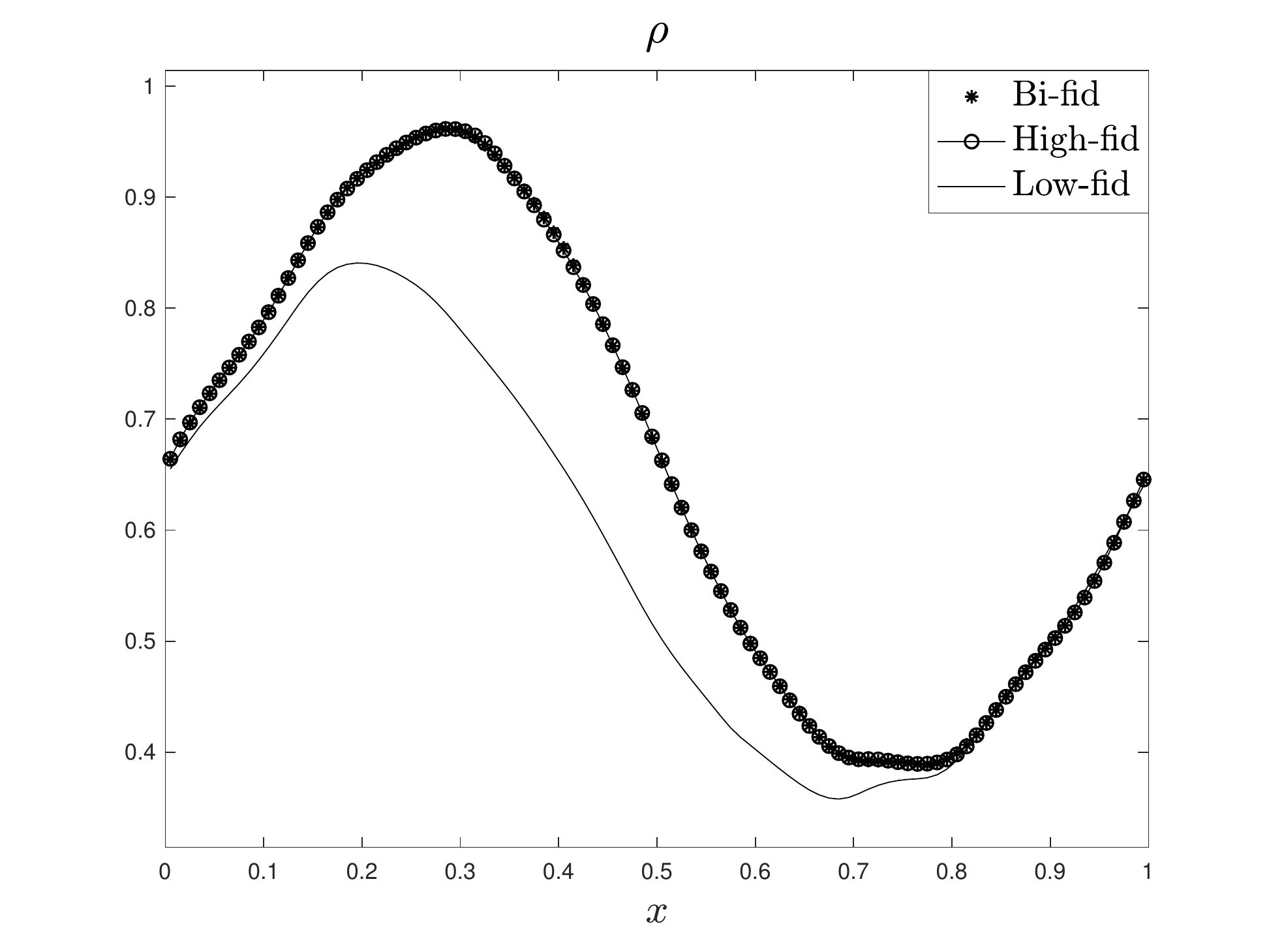}
\includegraphics[width=0.48\linewidth]{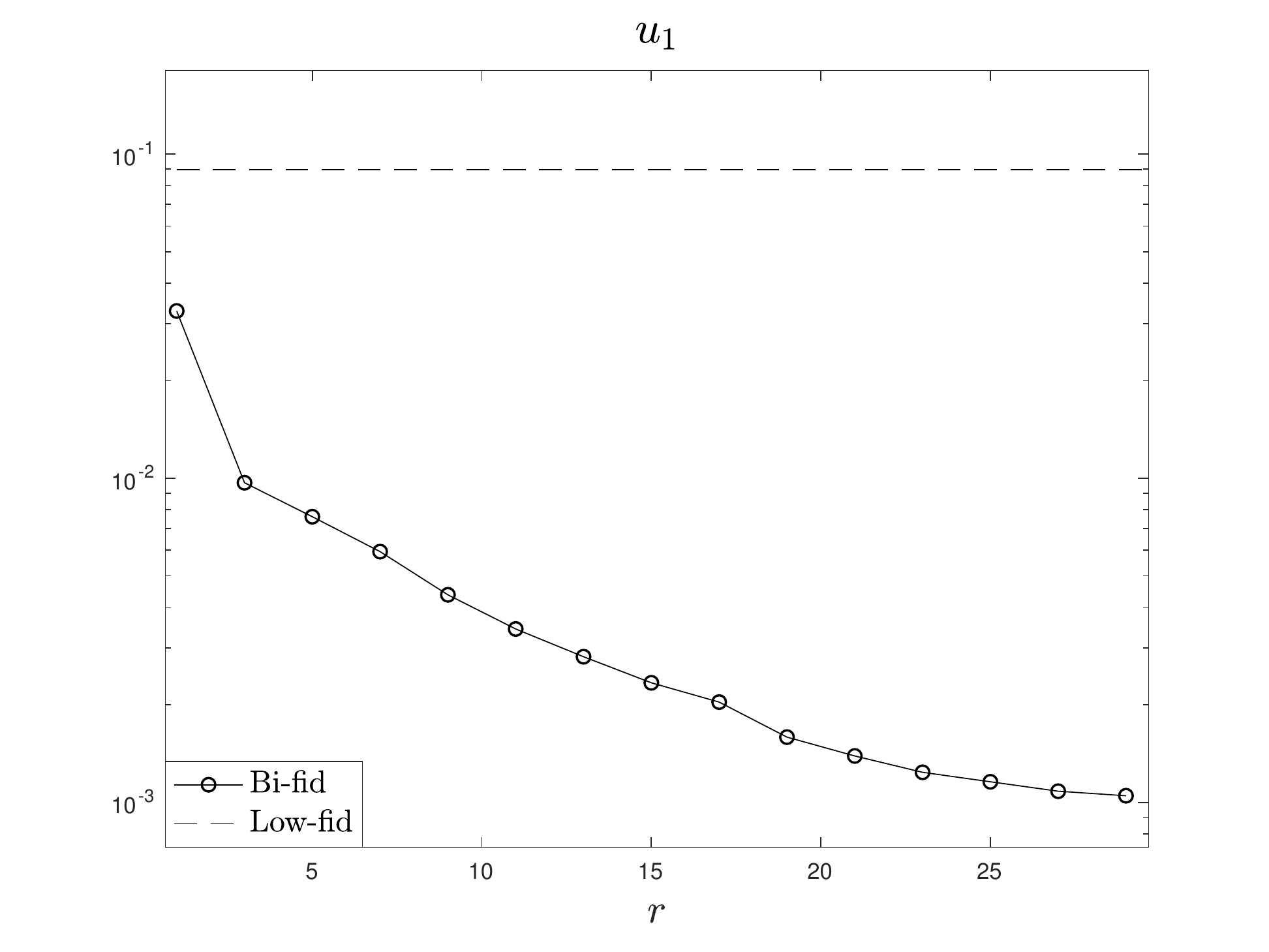}
\includegraphics[width=0.48\linewidth]{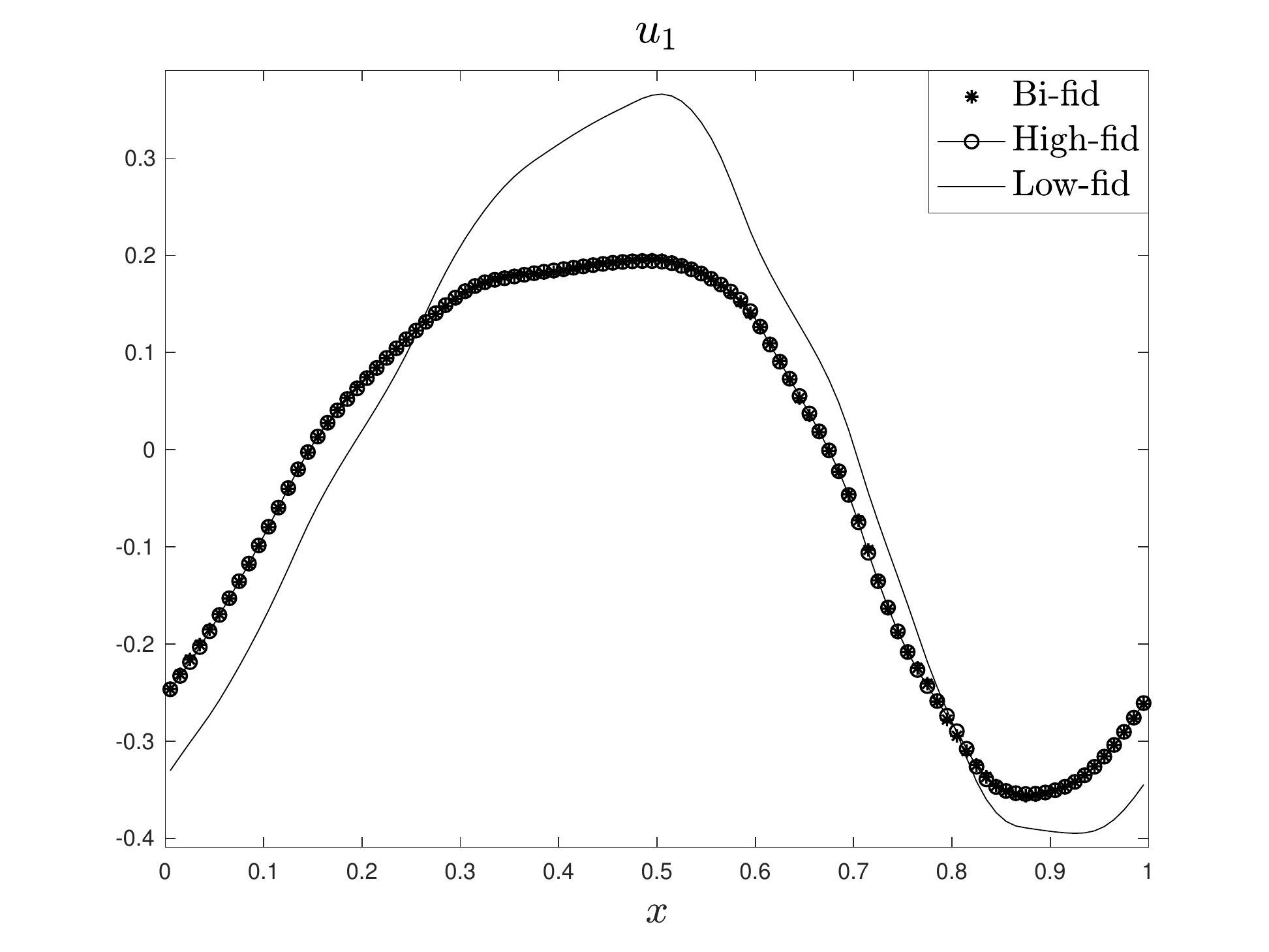}
\includegraphics[width=0.48\linewidth]{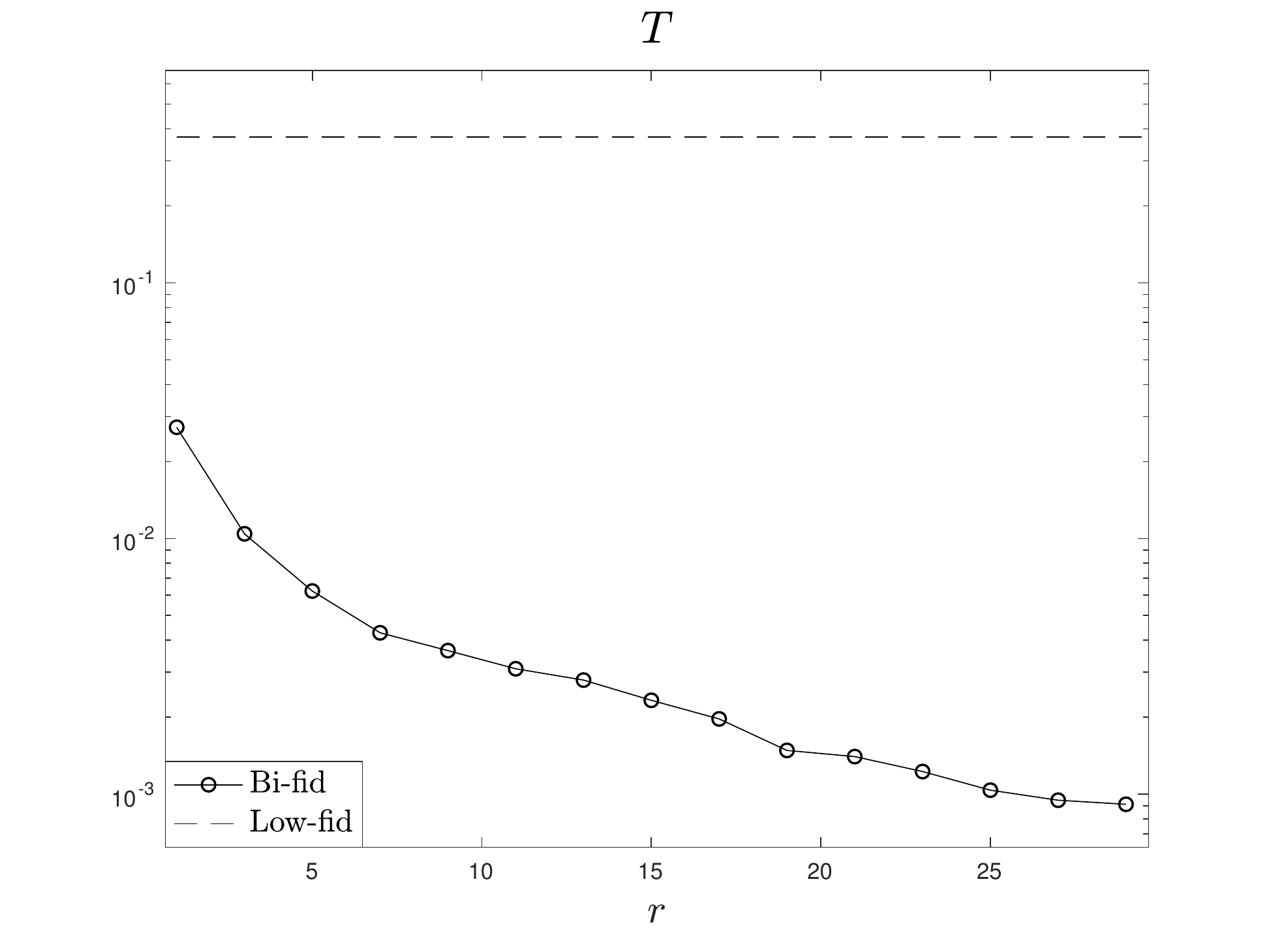}
\includegraphics[width=0.48\linewidth]{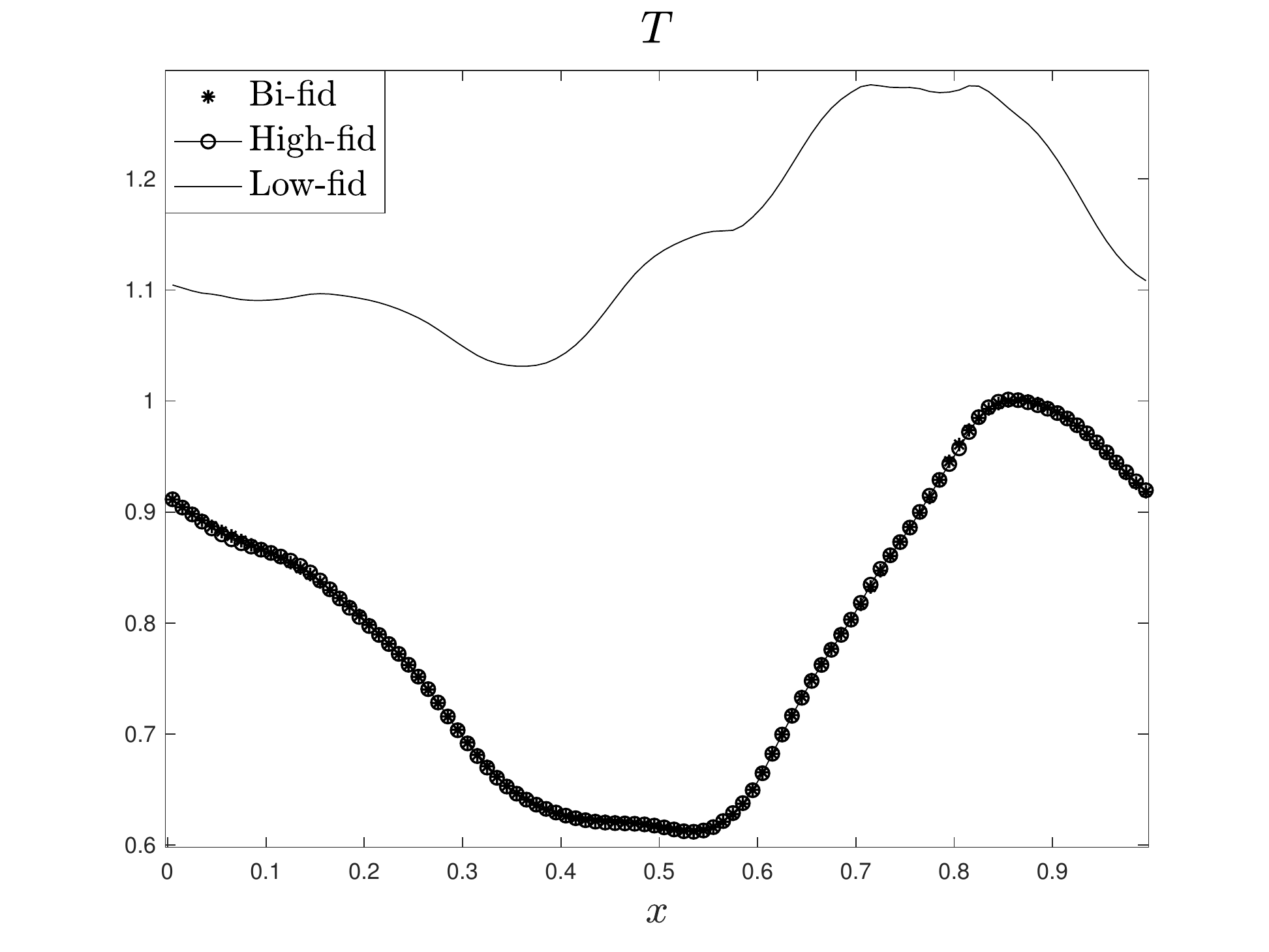}
\caption{(Left) The mean $L^2$ errors between high-fidelity and 
low- or bi-fidelity solutions with respect to the number of high-fidelity run; (Right) Comparison of the low-fidelity solution ($N_v^l=8$), high-fidelity solutions ($N_v^l=16$), and the corresponding bi-fidelity approximations $r=25$ for a fixed $z$. 
}
\label{Fig5-conv}
\end{figure}

\begin{figure}[H]
\centering
\includegraphics[width=0.45\linewidth]{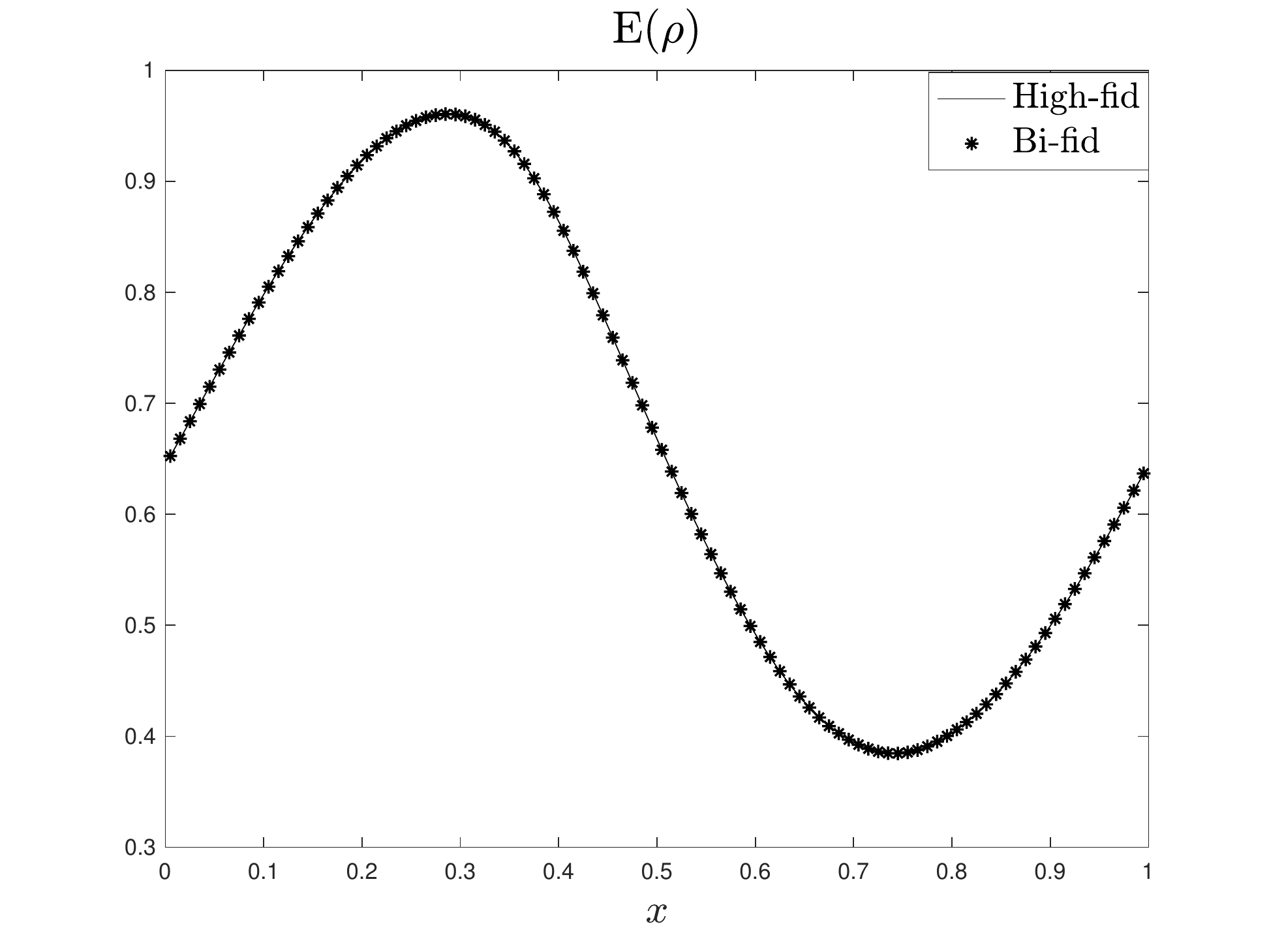}   
\includegraphics[width=0.45\linewidth]{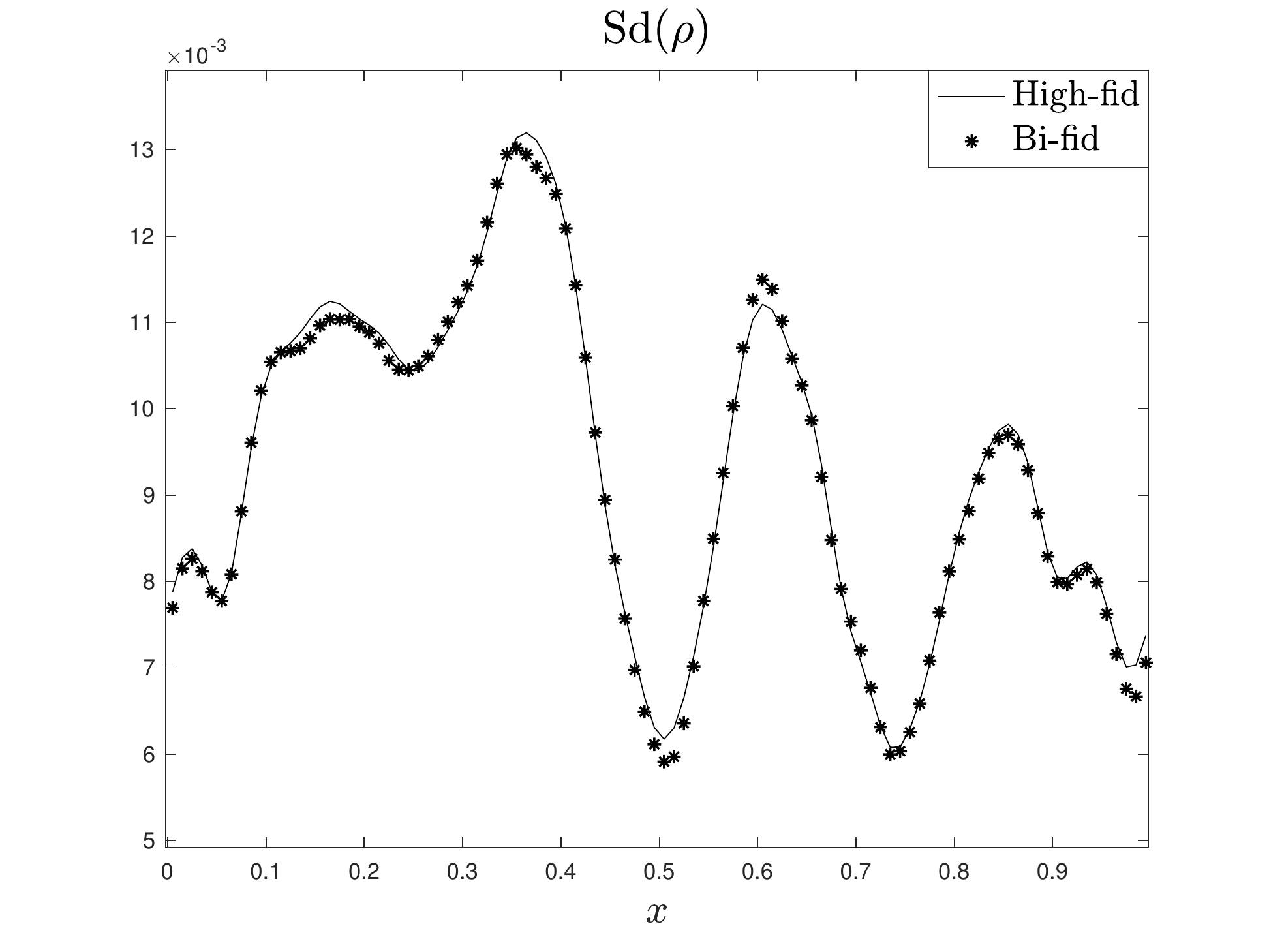}
\includegraphics[width=0.45\linewidth]{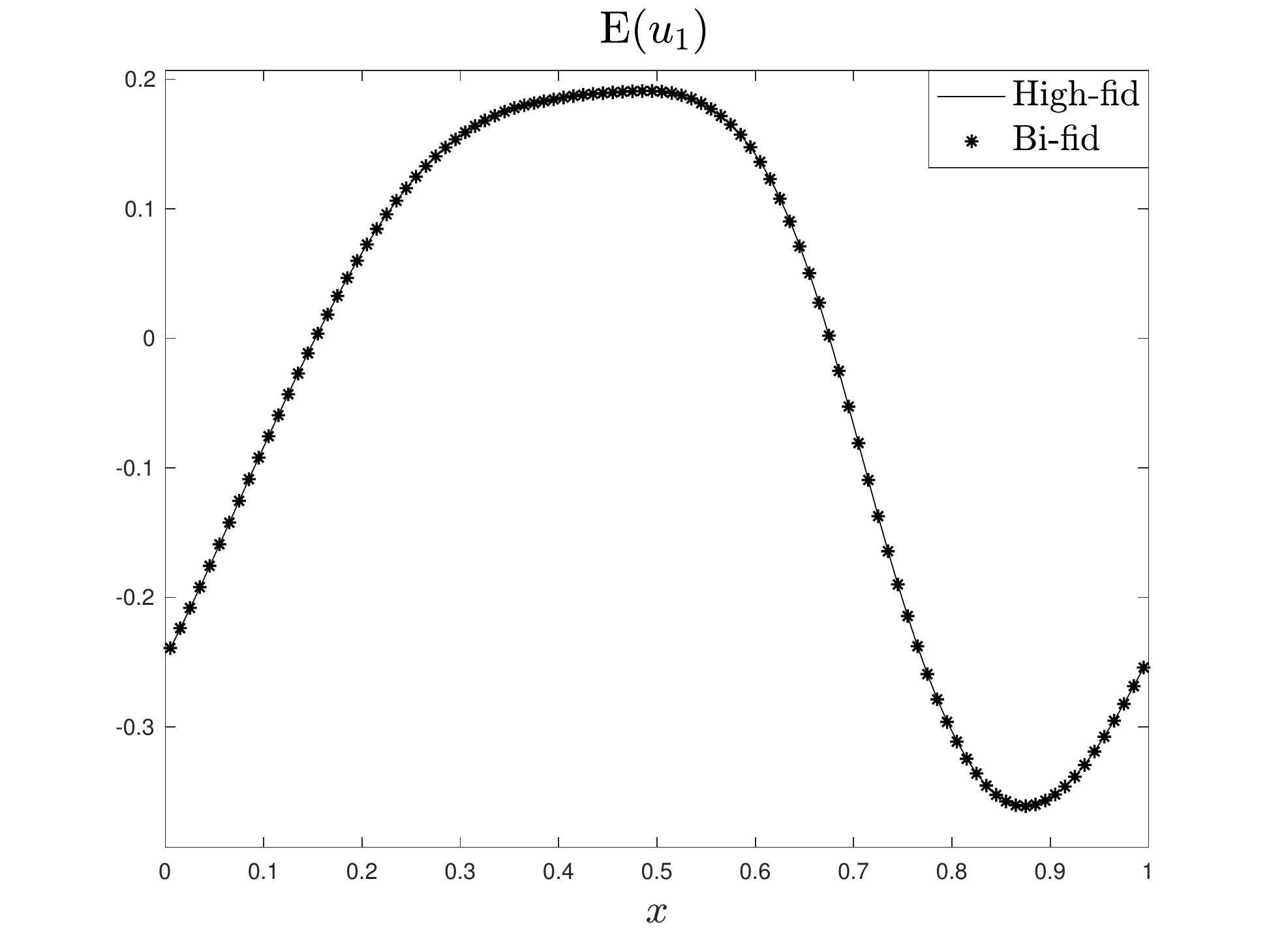}
\includegraphics[width=0.45\linewidth]{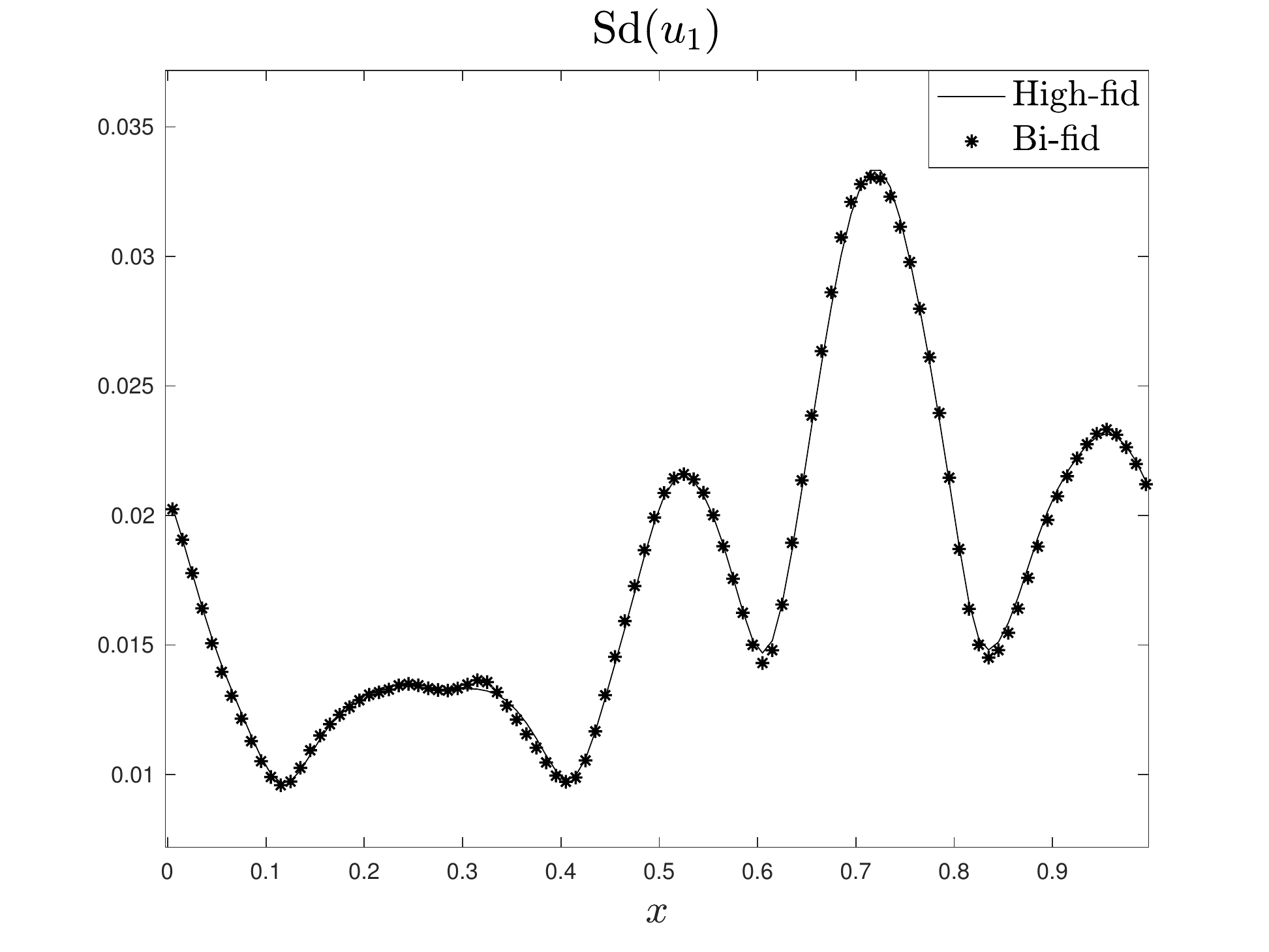}
\includegraphics[width=0.45\linewidth]{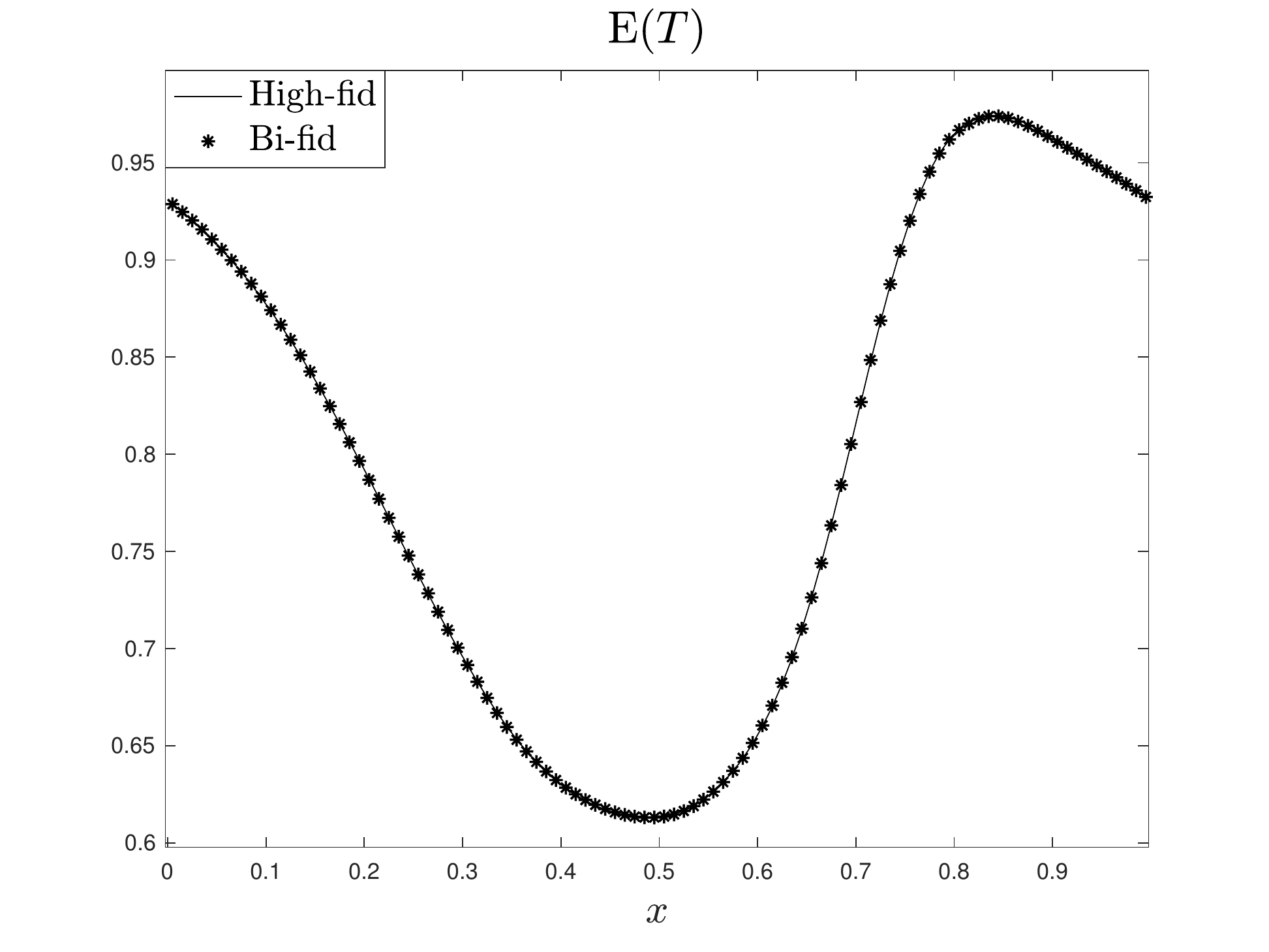}
\includegraphics[width=0.45\linewidth]{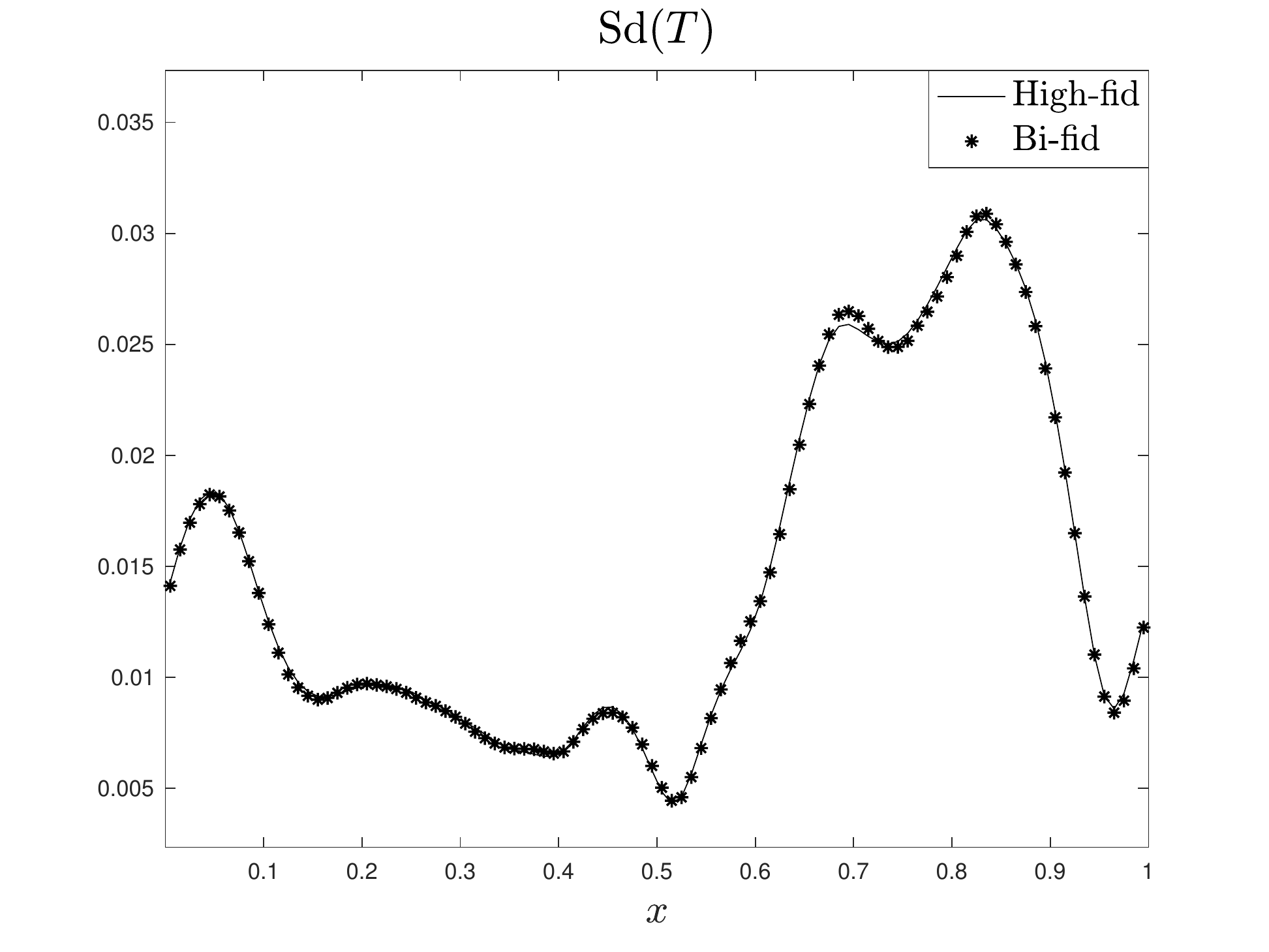}
\caption{The mean and standard deviation of $\rho$, $u_1$, $T$ of high-fidelity solutions and bi-fidelity solutions with $r=25$. 
}
\label{Fig5-mv}
\end{figure}

%-------------------------------------------------------------------------
\section{Conclusion}
\label{sec:5}

In this work, we study the multiscale Boltzmann equation with multi-dimensional random parameters by a bi-fidelity collocation method \cite{NGX14, ZNX14}. By choosing the low-fidelity solution as the solution of the corresponding first order macro-model 
-- the compressible Euler equations with a consistent initial data, our bi-fidelity approximation can capture well the variations of macroscopic quantities computed from the high-fidelity AP solver of the Boltzmann equation with multiple scales,
at a much reduced computational cost and memory footprint. An error analysis developed in \cite{NGX14} has been extended by incorporating the knowledge of
the regularity of our high-fidelity and low-fidelity solutions. The computational accuracy and efficiency are demonstrated in various numerical examples and holds promise to accelerate the computation
for more complex problems in multiscale kinetic equations with uncertainty. 

\begin{appendices}
\section{The Lemma proof}
\begin{lemma}
\label{LM}
Let $z_{\ast}^H$, $z_{\ast}^L$ be the maximizers of (\ref{zH}) and (\ref{zL}), respectively. 
One assumes that 

1. $\exists$ positive constants $K^H$, $K^L$, $C_1$, $C_2$ such that 
$$\left| ||u^L(z_{\ast}^H)||^L - K^H ||u^H(z_{\ast}^H)||^H \right| \leq C_1\,\e, $$ 
$$\left| ||u^H(z_{\ast}^L)||^H - K^L ||u^L(z_{\ast}^L)||^L\right| \leq C_2\,\e, $$

2. $\exists\, \eta^L\leq 1$ and $\eta^H \leq 1$ such that 
$$ || P_{U_n^L}u^L(z_{\ast}^H)||^L \leq \eta^H ||P_{U_n^H}u^H(z_{\ast}^H)||^H, $$
$$ || P_{U_n^H} u^H(z_{\ast}^L)||^H \leq \eta^L ||P_{U_n^L}u^L(z_{\ast}^L)||^L. $$ 

3. $\exists$ constants $A^H$ and $A^L$ satisfying 
$$ 1 \leq A^H < \frac{\eta^H}{\max\{0, \eta^H - K^H\}}, \qquad
1\leq A^L < \frac{\eta^L}{\max\{0, \eta^L - K^L\}} $$
such that 
$$ ||u^H(z_{\ast}^H)||^H \leq A^H d^H(u^H(z_{\ast}^H), U_n^H), $$ 
$$ ||u^L(z_{\ast}^L)||^L \leq A^L d^L(u^L(z_{\ast}^L), U_n^L). $$
Then $$ d^H(u^H(z_{\ast}^L), U_n^H) \geq (\delta_1 - \delta_2\,\e)\, d^H (u^H(z_{\ast}^H), U_n^H), $$ 
with $0< \delta_1 - \delta_2\,\e <1$. 
\end{lemma}
\begin{proof}
Using the above assumptions, one gets
\begin{align*}
d^L(u^L(z_{\ast}^H), U_n^L) & = ||u^L(z_{\ast}^H)||^L - ||P_{U_n^L}u^L(z_{\ast}^H)||^L \\[4pt]
& \geq  K^H ||u^H(z_{\ast}^H)||^H - C_1\,\e - ||P_{U_n^L}u^L(z_{\ast}^H)||^L \\[4pt]
& \geq  K^H ||u^H(z_{\ast}^H)||^H  - C_1\,\e  - \eta^H ||P_{U_n^H}u^H(z_{\ast}^H)||^H \\[4pt]
& = \eta^H \left[ ||u^H(z_{\ast}^H)||^H - ||P_{U_n^H}u^H(z_{\ast}^H)||^H\right]  
-  C_1\,\e + (K^H - \eta^H) ||u^H(z_{\ast}^H)||^H \\[4pt]
& \geq \eta^H d^H(u^H(z_{\ast}^H), U_n^H) + C_1\,\e - \max\{0, \eta^H - K^H\} A^H d^H(u^H(z_{\ast}^H), U_n^H) \\[4pt]
& = \left[\eta^H -  \max\{0, \eta^H - K^H\} A^H \right] d^H(u^H(z_{\ast}^H), U_n^H) -  C_1\,\e, 
\end{align*}
Similarly, we have
$$ d^H(u^H(z_{\ast}^L), U_n^H) \geq \left[ \eta^L - \max\{0, \eta^L - K^L\}A^L\right] d^L(u^L(z_{\ast}^L), U_n^L) - C_2\,\e. $$ 
Therefore, 
\begin{align*}
d^H(u^H(z_{\ast}^L), U_n^H) & \geq \left[ \eta^L - \max\{0, \eta^L - K^L\}A^L\right] d^L(u^L(z_{\ast}^L), U_n^L) - C_2\,\e\\[4pt]
& \geq \left[ \eta^L - \max\{0, \eta^L - K^L\}A^L\right] d^L(u^L(z_{\ast}^H), U_n^L) -  C_2\,\e \\[4pt]
& \geq \left[ \eta^L - \max\{0, \eta^L - K^L\}A^L\right] \cdot \left[\eta^H -  \max\{0, \eta^H - K^H\} A^H \right] 
d^H(u^H(z_{\ast}^H), U_n^H)  \\[4pt]
&\quad  - \left\{\left[ \eta^L - \max\{0, \eta^L - K^L\}A^L\right] C_1 + C_2 \right\} \e \\[4pt]
& \geq (\delta_1 - \delta_2\,\e)\, d^H(u^H(z_{\ast}^H), U_n^H), 
\end{align*}
where we assume that
$$0< \delta_1 = \left[ \eta^L - \max\{0, \eta^L - K^L\}A^L\right] \cdot \left[\eta^H - \max\{0, \eta^H - K^H\} A^H \right] <1, $$ 
and $\exists\, \delta_2$ such that $0< \delta_1 - \delta_2\,\e <1$ with 
$$ \left[ \eta^L - \max\{0, \eta^L - K^L\}A^L\right] C_1 + C_2 \leq \delta_2\, d^H (u^H(z_{\ast}^H), U_n^H). $$
\end{proof}

\section{{\color{black} Proof of analyticity for the linearized equation }}
For simplicity, consider the linearized Boltzmann equation under the acoustic scaling. 
The perturbative solution satisfies 
\begin{equation}\label{h-App} \partial_t h + v \cdot \nabla_x h = \frac{1}{\e} \mathcal L(h). \end{equation}
For simplicity, we write out the proof for one-dimensional random variable $z$. 
Assume the collision cross-section has the form
$ b(\cos\theta) = b_0(\cos\theta) + b_1(\cos\theta)\, z$, and we require $|b_1|\leq C_b$. 
Define the operator $\mathcal T_{\e}:= \frac{1}{\e}\mathcal L - v\cdot\nabla_x$, and the notation $h_l := \partial_z^l h$ for all $l\geq 0$. 
Take $\partial_z^l$ of the equation \eqref{h-App}, then 
$$ \partial_t h_l = \mathcal T_{\e}(h_l) + \frac{1}{\e} \mathcal L_{b_1}(h_{l-1}), $$
where $\mathcal L_{b_1}$ is defined by substituting $b(\cos\theta)$ by $b_1(\cos\theta)$ in the linearized collision operator $\mathcal L$. 
Thus
$$ \frac{d}{dt} ||h_l||_{H^k}^2 = 2 \langle \mathcal T_{\e}(h_l), h_l \rangle_{H^k} + \frac{2}{\e} l \, \langle \mathcal L_{b_1}(h_{l-1}), h_l \rangle_{H^k}. $$
By the hypocoercivity theory and $\mathcal L$ being a bounded operator, 
$$  \frac{d}{dt} ||h_l||_{H^k}^2  \leq - \e C_0\, ||h_l||_{H_{\Lambda}^k}^2 + \frac{C_1}{\e} l \, ||h_{l-1}||_{H^k} ||h_l||_{H^k}, $$
where 
$\|h\|_{H_{\Lambda}^k}^2=\sum_{|j|+|l|=k}\left\|\frac{\partial^2}{\partial v_j \partial x_l}h\right\|_{\Lambda}^2$ 
and $\|h\|_{\Lambda}=\left\|h(1+|v|)^{\gamma / 2}\right\|_{L^{2}_{x,v}}$, with $\gamma$ shown in
\eqref{Kernel}. Since $ ||h_l||_{H_{\Lambda}^k}$ controls $||h_l||_{H^k}$, then by setting $g_l = ||h_l||_{H^k}$, 
$$ \frac{d}{dt}g_l \leq -\e C_0\, g_l + \frac{C_1}{\e} l\, g_{l-1}. $$
Using \cite[Lemma 6]{Qin-Li} in a similar way, one gets
$$ g_l(t) \leq e^{-\e C_0 t} \sum_{m=0}^l \frac{l!}{(l-m)!\, m!}\left(\frac{C_1}{\e}t\right)^m\, h_{l-m}(0). $$
 Following \cite[Theorem 5]{Qin-Li}, if the initial data satisfies 
 $ ||\partial_z^l h_0(z) ||_{H^k} \leq R^l$, $\forall l \geq 0$, 
 then at time $t$, 
 $$ ||\partial_z^l h(z)||_{H^k} \leq C e^{-\e C_0 t}\left( R + \frac{C_1}{\e} t \right)^l, $$
 where $C$, $C_0$, $C_1$ are independent of $l$ and $\e$. 
  The convergence radius for $h$ at any point $z_0$ is defined by
 $$r\left(z_{0}\right)=\frac{1}{\limsup _{l \rightarrow \infty}\left(g_{l}(z_{0})/\, l !\right)^{1 / l}}=\infty, $$
 which is independent of $z_0$, thus $h$ (or $f$) is {\it analytic} in $z$. 
\end{appendices}

\bibliographystyle{siam}
\bibliography{MultiFid.bib}

\end{document}